\documentclass[a4paper,11pt]{amsart}
\usepackage[english]{babel}
\usepackage{verbatim}
\usepackage{mathrsfs}
\usepackage{amstext}
\usepackage{amsthm}
\usepackage{amssymb}
\usepackage{amsfonts}
\usepackage{amsmath}
\usepackage{amscd}
\usepackage{esint}
\usepackage{enumerate}
\usepackage[unicode=true,pdfusetitle,
bookmarks=true,bookmarksnumbered=false,bookmarksopen=false,
breaklinks=false,pdfborder={0 0 1},colorlinks=false]
{hyperref}

\usepackage{a4wide}
\usepackage[foot]{amsaddr}
\usepackage{mathtools}
\usepackage{dsfont}
\usepackage[shortlabels]{enumitem}
\usepackage{bm}
\usepackage{tikz}
\usepackage{nicematrix}
\usepackage{caption}

\makeatletter
\numberwithin{figure}{section}
\theoremstyle{plain}
\newtheorem{thm}{\protect\theoremname}[section]
\theoremstyle{definition}
\newtheorem{rem}[thm]{\protect\remarkname}
\theoremstyle{definition}
\newtheorem{defn}[thm]{\protect\definitionname}
\theoremstyle{plain}
\newtheorem{prop}[thm]{\protect\propositionname}
\theoremstyle{plain}
\newtheorem{lem}[thm]{\protect\lemmaname}
\theoremstyle{plain}

\newtheorem*{SSOYconjecture1}{SSOY predictability conjecture}
\newtheorem*{SSOYconjecture2}{SSOY prediction error conjecture}
\theoremstyle{plain}
\newtheorem{cor}[thm]{\protect\corollaryname}

\theoremstyle{definition}

\theoremstyle{definition}

\theoremstyle{definition}

\theoremstyle{definition}

\theoremstyle{plain}

\usepackage{soul}

\usetikzlibrary{shapes,arrows}

\DeclareMathOperator{\dist}{dist}

\DeclareMathOperator{\Leb}{Leb}

\DeclareMathOperator{\Id}{Id}
\DeclareMathOperator{\supp}{supp}

\DeclareMathOperator{\conv}{conv}

\DeclareMathOperator{\Lip}{Lip}

\DeclareMathOperator{\Orb}{Orb}
\DeclareMathOperator{\rank}{rank}

\DeclareMathOperator{\Pred}{Pred}
\DeclareMathOperator{\Var}{Var}
\DeclareMathOperator{\idim}{ID}
\DeclareMathOperator{\uid}{\overline{\idim}}
\DeclareMathOperator{\lid}{\underline{\idim}}

\newcommand{\R}{\mathbb R}

\newcommand{\N}{\mathbb N}

\newcommand{\II}{\mathcal I}

\newcommand{\eps}{\varepsilon}

\newcommand{\mH}{\mathcal{H}}

\newcommand{\mB}{\mathcal{B}}

\newcommand{\hdim}{\dim_H}

\newcommand{\udim}{\overline{\dim}_B}
\newcommand{\lbdim}{\underline{\dim}_B}

\makeatother

\providecommand{\conjecturename}{Conjecture}
\providecommand{\corollaryname}{Corollary}
\providecommand{\definitionname}{Definition}
\providecommand{\examplename}{Example}
\providecommand{\lemmaname}{Lemma}
\providecommand{\problemname}{Problem}
\providecommand{\propositionname}{Proposition}
\providecommand{\remarkname}{Remark}
\providecommand{\theoremname}{Theorem}
\providecommand{\taskname}{Task}

\def\N{{\mathbb N}}

\def\be{\begin{equation}}
	\def\ee{\end{equation}}

\makeatletter
\renewcommand*\env@matrix[1][*\c@MaxMatrixCols c]{%
	\hskip -\arraycolsep
	\let\@ifnextchar\new@ifnextchar
	\array{#1}}
\makeatother

\newcommand{\greyzero}{{\color{lightgray} 0}}
\newcommand{\fatgreyzero}{{\color{lightgray} \;0}}

\NiceMatrixOptions
{
	custom-line =
	{
		letter = I ,
		tikz = dashed ,
		width = \pgflinewidth
	}
}

\author[K. Bara\'{n}ski]{Krzysztof Bara\'{n}ski$^*$}
\address{$^*$Institute of Mathematics, University of Warsaw, ul.~Banacha 2, 02-097 Warszawa, Poland}
\email{baranski@mimuw.edu.pl}

\author[Y. Gutman]{Yonatan Gutman$^\dagger$}
\address{$^\dagger$Institute of Mathematics, Polish Academy of Sciences,
	ul.~\'Sniadeckich 8, 00-656 Warszawa, Poland}
\email{gutman@impan.pl}

\author[A. \'{S}piewak]{Adam \'{S}piewak$^{\dagger\S}$}
\address{$^\S$Department of Mathematics, Bar-Ilan University, Ramat Gan, 5290002, Israel}
\email{ad.spiewak@gmail.com}


\title[Prediction of dynamical systems]{Prediction of dynamical systems from time-delayed measurements with self-intersections}
\date{\today}
\begin{document}

\begin{abstract} In the context of predicting the behaviour of chaotic systems, Schroer, Sauer, Ott and Yorke conjectured in 1998 that if a dynamical system defined by a smooth diffeomorphism $T$ of a Riemannian manifold $X$ admits an attractor with a natural measure $\mu$ of information dimension smaller than $k$, then $k$ time-delayed measurements of a one-dimensional observable $h$ are generically sufficient for $\mu$-almost sure prediction of future measurements of $h$. In a previous paper we established this conjecture in the setup of injective Lipschitz transformations $T$ of a compact set $X$ in Euclidean space with an ergodic $T$-invariant Borel probability measure $\mu$. In this paper we prove the conjecture for all (also non-invertible) Lipschitz systems on compact sets with an arbitrary Borel probability measure, and establish an upper bound for the decay rate of the measure of the set of points where the prediction is subpar. This partially confirms a second conjecture by Schroer, Sauer, Ott and Yorke related to empirical prediction algorithms as well as algorithms estimating the dimension and number of required delayed measurements (the so-called embedding dimension) of an observed system. We also prove general time-delay prediction theorems for locally Lipschitz or H\"older systems on Borel sets in Euclidean space.
\end{abstract}

\maketitle


\section{Introduction}

\subsection{Time-delayed measurements and Takens-type delay embedding theorem}\label{subsec:intro} 

Suppose we are given a dynamical system
\[
T\colon X \to X
\]
on some \emph{phase space} $X$ with a transformation (\emph{evolution rule}) $T$.\footnote{ In this paper we consider only discrete-time dynamical systems, but the analysis can also be applied to continuous-time systems of the form $\varphi^t \colon X \to X$, $t \in [0, +\infty)$, by taking $T = \varphi^{t_0}$ for some $t_0 > 0$.} For an experimentalist, direct knowledge on the system $(X, T)$ may be lacking or non-existing. As a consequence, information on its behaviour is often provided by a finite sequence of \emph{time-delayed measurements} 
\begin{equation}\label{eq:timeseries}
h(x_i), h(Tx_i),\ldots, h(T^{m}x_i), \qquad x_1,\ldots, x_r \in X,\quad r, m \in \N
\end{equation}
of a function (\emph{observable}) $h\colon X \to \R$. In general, the information contained in \eqref{eq:timeseries} might not be sufficient for reconstructing the system $(X, T)$. However, it is natural to ask, how well one may \emph{approximately} reconstruct $(X, T)$ (or determine its important attributes such as its dimension) from \eqref{eq:timeseries}, at least for some observables $h$. This problem has been widely studied from both theoretical and applied
points of view, in natural, social and medical sciences as well as
engineering, see e.g.~\cite{PCFS80,FarmerSidorowich87,KY90,sm90nonlinear,SYC91,KBA92,SeaClutter,SSOY98,Voss03,hgls05distinguishing, Rob11,BioclimaticBuildings, HBS15, BaghReddy23}. 

A useful mathematical abstraction is given by studying the sequence of the form 
\begin{equation}\label{eq:inf_time_series}
h(x), h(Tx), \ldots, h(T^m x)
\end{equation}
for \emph{all} initial points $x\in X$, commonly known as the \emph{time series} defined by $h$. For $k \in \N$, we define $k$-\emph{delay coordinate map}\footnote{ Note that $\phi$ depends on $T,h$ and $k$, but for simplicity, we do not include this dependence in notation.}
\begin{equation}\label{eq:coord_map} 
\phi\colon X \to \R^k, \qquad  \phi(x) = (h(x), h(Tx), \ldots, h(T^{k-1} x)).
\end{equation}

In 1981, in the seminal paper \cite{T81} (see also \cite{N91,StarkEmbedSurvey,Huke-report} for an accessible overview and an extended version of the result), Takens showed that if $X$ is a smooth manifold, then for a typical (generic in the Baire sense) pair $(T, h)$ of a smooth diffeomorphism $T\colon X \to X$ and a smooth observable $h\colon X \to \R$, the $k$-delay coordinate map $\phi$ is injective on $X$ for $k > 2\dim X$. In particular, it follows that for $k > 2\dim X$, the map $\phi$ is a (smooth) conjugation between the system $(X, T)$ and $(\phi(X), S)$, where
\[
S = \phi \circ T \circ \phi^{-1}.
\]
In particular, the diagram 
\[
\begin{CD}
X @>T>> X\\
@VV\phi V @VV \phi V\\
\phi(X) @>S>> \phi(X)
\end{CD}
\]
commutes and the map $S$ is well-defined as
\[
S(h(x), h(Tx), \ldots, h(T^{k-1}x)) = (h(Tx), h(T^2 x) \ldots, h(T^k x)). 
\]
In this way the system $(X, T)$ is \emph{reconstructed} (or \emph{embedded}) in $\R^k$ as $(\phi(X), S)$ via a (finite) sequence of $k$ time-delayed measurements of the observable $h$. 

This result, known nowadays as the \emph{Takens delay embedding theorem}, was extended in a number of papers \cite{SYC91, StarkEmbedSurvey, 1999delay, CaballeroEmbed, TakensPlato03, StarkStochEmbed, Rob05, Gut16, GQS18, NV18, BGS20, BGS22} to various settings and categories of systems. We call results of this type \emph{Takens}(-\emph{type}) \emph{delay embedding theorems}. Let us quote such a theorem, due to Sauer, Yorke and Casdagli \cite{SYC91} (see also \cite[Theorem~14.5]{Rob11}), valid for arbitrary injective Lipschitz dynamics on a compact set in Euclidean space (in this article we work in the similar setting). The result states that if $T\colon X\rightarrow X$ is an injective Lipschitz map on a compact set $X \subset \R^N$, $N \in \N$, then $\phi$ is injective on $X$ for a typical (\emph{prevalent}) Lipschitz observable $h \colon \R^N \to \R$, provided $k > 2\udim X$ and $2\udim(\{ x \in X : T^p x = x \}) < p$ for $p=1, \ldots, k-1$, where $\udim X$ denotes the upper box-counting dimension of $X$ (see Definition~\ref{defn:dim}).

\begin{rem}\label{rem:periodic points}
As the map $T\colon X\rightarrow X$ can be chosen arbitrarily, to obtain the injectivity of $\phi$ one obviously has to impose some condition on the size of the sets of $T$-periodic points with low periods. However, the assumption on the upper box-counting dimension of such sets, quoted above, is satisfied for many systems. For instance, the assumption holds for a generic $C^r$-diffeomorphism ($r\geq 1$) of a compact manifold, as by the Kupka--Smale and Hartman--Grobman theorems, for such a map the number of periodic points of a given period is finite (cf.~\cite[p.~4943]{BGS20}).
Furthermore, if $X \subset \R^N$ is a compact set, which is invariant under a flow $\{\varphi^t\}_t$ given by an autonomous differential equation $\dot{x}=F(x)$ defined in a neighbourhood of $X$, where $F$ is Lipschitz with a constant $L >0$, and all the zeroes of $F$ are isolated, then the map $T = \varphi^t$ for $0 < |t| < \pi/(L\,\udim X)$ has only a finite number of fixed points and no periodic points of minimal periods $p = 2, \ldots, 2\udim X$, as follows from a result of Yorke \cite{y69} (see also \cite[Remark~1.2]{Gut16}). Hence, for such maps the assumption on the upper box-counting dimension of the sets of points with low periods is again satisfied.
\end{rem}

We note that the condition $k > 2\dim X$, needed for a dynamical embedding of the system in Takens-type delay embedding theorems, agrees with a known phenomenon in classical dimension theory. Indeed, by the Menger--N\"obeling embedding theorem \cite[Theorem~5.2]{HW41}, a generic continuous map $\phi \colon M \rightarrow \R^{2d+1}$ on a compact topological manifold $M$ of dimension $d$, is a continuous embedding. Similarly, the Whitney embedding theorem \cite[Theorem~2.15.8]{narasimhan1985analysis} asserts that a generic $C^r$-map $\phi \colon M \rightarrow \R^{2d+1}$ ($r \ge 1$) on a smooth manifold $M$ of dimension $d$, is a $C^r$-embedding. In both cases, the dimension $2d+1$ cannot be diminished.

For a more detailed introduction to time-delayed embeddings, see e.g.~\cite[Chapter~13]{AlligoodSauerYorkeBook} or \cite[Chapter~6]{BT11}.

\subsection{Prediction algorithms and the predictability conjectures of Schroer, Sauer, Ott and Yorke} 

Let us look at the time-delayed measurements from a different perspective. Instead of trying to reconstruct the original dynamics $(X,T)$ from time-delayed measurements of an observable $h$ (determining whether the map $\phi$ is injective), consider the problem of \emph{predicting} from the first $k$ terms of the time series $h(x), h(Tx), \ldots, h(T^{k-1} x)$ its future values $h(T^{k} x), h(T^{k+1} x),\ldots$.  Notice that such prediction takes place not in the original phase space $X$, but in $\phi(X) \subset \R^k$, considered as a model space for the system. One of the basic questions in this context is to determine for which values $k$ such prediction is possible.  

It should be observed that the prediction problem can be considered for both invertible and non-invertible transformations $T$, while the possibility of dynamical reconstruction (embedding) by delay coordinates maps is generally restricted to invertible systems. 
It is important to note that from a theoretical point of view, reconstruction implies prediction but not vice versa. 
Thus, the problem of predicting future values of a dynamically
generated time series may possibly be more accessible than the problem
of reconstructing the entire system. Moreover, the problem of
prediction is relevant from the point of view of applications (see
e.g.~\cite{Rainfall94, hgls05distinguishing, Streamflow07, SolarCycle18,Dlotko_et_al}). 
In this context, creating reliable prediction algorithms is a matter of major importance. Let us present one of
these algorithms, proposed by Farmer and Sidorowich in \cite{FarmerSidorowich87}.
To describe it, consider a sequence of measurements $h(x), \ldots, h(T^{n+k-1}(x))$ of an observable $h$ for a point $x \in X$ and some $k, n \in \N$. This defines a sequence $z_0, \ldots, z_n \in \phi(X)$ of $k$-delay coordinate vectors, where
\begin{equation}\label{eq:y_i}
z_i = z_i(x) = \phi(T^ix) = (h(T^ix), \ldots, h(T^{i+k-1}x)), \qquad i = 0, \ldots, n.
\end{equation}
Given such a sequence, for $y\in \R^k$, $n \in \N$ and $\eps > 0$
define
\begin{equation}\label{eq:y_hat}
\Pred_{x,n,\eps}(y) = \frac{1}{\#\mathcal I_n} \sum_{i \in \mathcal I_n} z_{i+1} \qquad \text{for} \quad \mathcal I_n = \mathcal I_n(x,y,\eps) =\{ 0 \leq i < n : z_i \in B(y, \eps) \},
\end{equation}
assuming $\II_n \ne \emptyset$, where $B(y, \eps)$ denotes the open ball of centre $y$ and radius $\eps$. 
Now, knowing the values of $z_0, \ldots, z_n$, we predict the value of the next point $z_{n+1} = (h(T^{n+1}x), \ldots, h(T^{n+k}x))$ by $\Pred_{x,n,\eps}(z_n)$.
In other words, the predicted value of $z_{n+1}$ is taken to be the average of the values $z_{i+1}$, $i = 0, \ldots, n-1$, where we count only those $i$, for which $z_i$ are $\eps$-close to the last known point $z_n$. In this way, we predict the one-step future of the dynamics in the model space $\phi(X) \subset \R^k$.
The Farmer--Sidorowich algorithm, as well as its variants, like the \emph{simplex algorithm} \cite{sm90nonlinear}, are important tools for \emph{non-parametric} prediction (see e.g.~\cite{Stavroglou2020,CarbonStock22}). 

When studying the Farmer--Sidorowich and other prediction algorithms, a useful and realistic approach is to consider a \emph{probabilistic} setting, where there is an (explicit or implicit) \emph{random} process, given by a probability measure $\mu$ on $X$, determining which initial states are accessible to the experimentalist.  
In this context, Schroer, Sauer, Ott and Yorke \cite{SSOY98} introduced in 1998 the following notion of probabilistic \emph{predictability} for systems defined by smooth diffeomorphisms on Riemannian manifolds. We present it here in a general setup of transformations of Borel sets in Euclidean space.

\begin{defn}\label{def:dynamical predictability} Let $X \subset \R^N$, $N \in \N$, be a Borel set admitting a Borel probability measure $\mu$. Let $T \colon X \to X$ be a Borel transformation, $h \colon X \to \R$ a Borel observable and $k \in \N$. Consider the $k$-delay coordinate map $\phi \colon X \to \R^k$ given by \eqref{eq:coord_map}. For $y \in \R^k$ and $\eps>0$ such that $\mu\big(\phi^{-1}(B(y, \eps))\big) > 0$ define
\begin{align*}
\chi_{\eps}(y) &=  \frac{1}{\mu\big(\phi^{-1}(B(y, \eps))\big)} \int \limits_{\phi^{-1}(B(y, \eps))} \phi\circ T \: d\mu,\\
\sigma_{\eps}(y) &= 
\bigg(\frac{1}{\mu\big(\phi^{-1}(B(y, \eps))\big)} \int \limits_{\phi^{-1}(B(y, \eps))} \|\phi(T(x)) -\chi_{\eps}(y)\|^2d\mu(x)\bigg)^{\frac{1}{2}}
\end{align*}
(provided the integrals exist).\footnote{For a relation of the Farmer--Sidorowich algorithm to another notion of predictability, see Corollary~\ref{cor:FS-predict} and the discussion afterwards.}
\end{defn}

\begin{defn}[{\bf Almost-sure predictability}{}]\label{def:SSOYpredictability}
For $y$ in the support\footnote{See Definition~\ref{defn:support}.} of $\phi_*\mu$ define 
\[ \sigma(y) = \lim \limits_{\eps \to 0} \sigma_{\eps}(y) \]
as the limit of \emph{prediction errors} $\sigma_{\eps}(y)$ at $y$ (provided the limit exists). A point $y$ is said to be $k$-\emph{predictable}, if  $\sigma(y)=0$. If $\phi_*\mu$-almost every point $y \in \R^k$ is predictable, then we say that the observable $h$ is \emph{almost surely $k$-predictable} with respect to $\mu$.
\end{defn}

A relation of $\chi_\eps$ and $\sigma_\eps$ to the Farmer--Sidorowich algorithm for ergodic\footnote{ See Definition~\ref{defn:ergodic}.} systems is described in the following proposition, where we set
\begin{equation}\label{eq:Var}
\Var_{x,n,\eps}(y) = \frac{1}{\# \II_n} \sum_{i \in \II_n} \| z_{i+1}(x) - \Pred_{x,n,\eps}(y)\|^2.
\end{equation}
for $x \in X$, $y \in \R^k$, $n \in \N$ and $\eps > 0$, with $z_i = z_i(x)$ and $\II_n = \II_n(x,y,\eps)$ defined in \eqref{eq:y_i}--\eqref{eq:y_hat}.

\begin{prop}\label{prop:FS} Let $X$ be a Borel set, $T\colon X \to X$ a Borel map and $\mu$ an ergodic $T$-invariant Borel probability measure on $X$. Let $h \colon X \to \R$ be a Borel observable, such that $h \in L^1(\mu)$. Then 
\[
\chi_\eps(y) =  \lim_{n \to \infty} \Pred_{x,n,\eps}(y)
\]
for $\mu$-almost every $x \in X$, $\phi_*\mu$-almost every $y \in \R^k$ and $\eps > 0$. Moreover, if $h \in L^2(\mu)$, then
\[
\sigma_\eps(y) = \lim_{n \to \infty} \big(\!\Var_{x,n,\eps}(y)\big)^{\frac 1 2}
\]
for $\mu$-almost every $x \in X$, $\phi_*\mu$-almost every $y \in \R^k$ and $\eps > 0$.
\end{prop}

The proof of Proposition~\ref{prop:FS} is presented in Section~\ref{sec:proof_predict_prob}.

We can now state the first conjecture of Schroer, Sauer, Ott and Yorke ({\emph{SSOY predictability conjecture}) in its original form. It is stated for a special class of \emph{natural measures}, which are `physically observed' $T$-invariant probability measures on attractors for diffeomorphisms on Riemannian manifolds, see Definition~\ref{defn:nat_measure} for details. The symbol $\idim$ denotes the information dimension of a measure (see Definition~\ref{defn:dim}).

\begin{SSOYconjecture1}[{\cite[Conjecture 1]{SSOY98}}]\label{con:1}
Let $T\colon M \to M$ be a smooth diffeomorphism of a compact Riemannian manifold $M$ with a compact $T$-invariant attractor $X \subset M$ and a natural measure $\mu$ on $X$ of information dimension $\idim(\mu) = D$. Then a generic observable $h\colon X \to \R$ is almost surely $k$-predictable with respect to $\mu$ for $k>D$.
\end{SSOYconjecture1}
Note that in this formulation some details, including the type of genericity and the smoothness class of the dynamics and observable, are not specified precisely.  Apparently, the most important feature of the conjecture is the fact that the bound for the minimal number of measurements is reduced (at least) by half compared to various versions of Takens-type delay embedding theorems, i.e.~from $k > 2\dim X$ to $k  > \dim \mu$. The possibility of such reduction is the main difference between the deterministic and probabilistic settings (where in the latter case one can neglect sets of measure zero).
In fact, the SSOY predictability conjecture has been invoked in the literature as a theoretical argument for reducing the number of time-delay measurements of observables (see \cite{OrtegaLouis98, McSharrySmith04, Liu10}), also in direct applications (e.g.~in \cite{Epilepsy} studying neural brain activity in focal epilepsy).

In our previous paper \cite[Corollary~1.10]{BGS22} we proved the SSOY predictability conjecture for arbitrary ergodic $T$-invariant Borel probability measures $\mu$. On the other hand, we constructed an example of a $C^\infty$-smooth diffeomorphism with a non-ergodic natural measure, for which the conjecture does not hold (see \cite[Theorem~1.11]{BGS22}). However, after replacing the information dimension $\idim(\mu)$ by the Hausdorff dimension $\dim_H\mu$ (see Definition~\ref{defn:dim}), we verified the conjecture for $C^r$-generic ($r \ge 1$) diffeomorphisms $T$ (see \cite[Corollary~1.9]{BGS22}). In fact, in \cite[Corollaries~1.9--1.10]{BGS22} we also showed that the suitable $k$-delay coordinate map $\phi$ is injective on a set of full $\mu$-measure for a generic observable $h$, inducing an \emph{almost sure embedding} of the system into $\R^k$. To this aim, we established a predictable embedding theorem \cite[Theorems~1.7 and~3.1]{BGS22} for injective Lipschitz maps $T$ on compact sets in Euclidean space and arbitrary Borel probability measures, assuming a suitable condition on the size of the sets of $T$-periodic points of low period, similar to the one in the Takens-type delay embedding theorem by Sauer, Yorke and Castagli, quoted in Subsection~\ref{subsec:intro}. 

In this paper we prove a general version of the SSOY predictability conjecture, valid for an arbitrary dynamical system defined by a Lipschitz transformation of a compact set in Euclidean space with a Borel probability measure, and Lipschitz observables. In particular, we do not assume the injectivity of the dynamics nor any bound on the size of the sets of periodic points. Similarly to \cite{SYC91, Rob11, BGS20, BGS22}, in this and subsequent results we consider the genericity of observables in terms of \emph{prevalence} (with a \emph{polynomial probe set}) in the space of Lipschitz or locally Lipschitz maps $h \colon X \to \R$ (see Definition~\ref{defn:preval} and the discussion afterwards). More precisely, we prove the following.

\begin{thm}[{\bf General SSOY predictability conjecture}{}]\label{thm:SSOY1}
Let $X \subset \R^N$, $N \in \N$, be a compact set, $\mu$ a Borel probability measure on $X$ and $T\colon X \to X$ a Lipschitz map. Then a prevalent Lipschitz observable $h\colon X \to \R$ is almost surely $k$-predictable with respect to $\mu$ for $k > \hdim \mu$. 
\end{thm}

The proof of Theorem~\ref{thm:SSOY1} is presented in Section~\ref{sec:proof_ssoy}. In fact, we show an extended version of the result (Theorem~\ref{thm:SSOY1_ext}), valid for H\"older observables.

\begin{rem}\label{rem:SSOY-ID} Since for a Lipschitz map $T\colon X \to X$ on a closed set $X \subset \R^N$ and ergodic $T$-invariant Borel probability measure $\mu$, we have $\dim_H \mu \le \underline{\idim}(\mu)$ (see e.g.~\cite[Proposition~2.1]{BGS22}), Theorem~\ref{thm:SSOY1} shows that the SSOY predictability conjecture holds in its original form (i.e.~with the information dimension of the measure $\mu$) for an arbitrary Lipschitz map $T$ on a compact set $X \subset \R^N$, an ergodic $T$-invariant Borel probability measure $\mu$ and a prevalent Lipschitz observable $h\colon X \to \R$.
\end{rem}

In \cite{SSOY98}, Schroer, Sauer, Ott and Yorke formulated also another conjecture concerning the decay rate of the prediction error $\sigma_\eps$ for a natural measure on the attractor of the system.

\begin{SSOYconjecture2}[{\cite[Conjecture 2]{SSOY98}}]\label{con:2}
Let $T\colon M \to M$ be a smooth diffeomorphism of a compact Riemannian manifold $M$ with a compact $T$-invariant attractor $X \subset M$ and a natural measure $\mu$ on $X$ of information dimension $\idim(\mu) = D$. Fix a generic observable $h\colon X \to \R$, $k \in \N$ and a sufficiently small $\delta > 0$. Then for the $k$-delay coordinate map $\phi$ corresponding to $h$ and sufficiently small $\eps>0$, the following hold.
\begin{enumerate}[$($i$)$]
\item If $k < D$, then 
\[
\mu(\{x \in X: \sigma_\eps(\phi(x)) > \delta\}) \geq C \quad \text{ for some } C > 0.
\]

\item\label{it: ssoy2 ii} If $D < k < 2D$ and $\phi$ is not injective on $X$, then 
\[
C_1  \eps^{k-D} \le \mu(\{x \in X: \sigma_\eps(\phi(x)) > \delta\}) \le C_2 \eps^{k-D} \quad \text{ for some } C_1, C_2 > 0.
\]
\item If $D < k < 2D$ and $\phi$ is injective on $X$, then 
\[
\mu(\{x \in X: \sigma_\eps(\phi(x)) > \delta\}) = 0.
\]
\item\label{it: ssoy2 iiI} If $k > 2D$, then 
\[
\mu(\{x \in X: \sigma_\eps(\phi(x)) > \delta\}) = 0.
\]
\end{enumerate}
\end{SSOYconjecture2}

The main result of this paper is establishing a part of the SSOY prediction error conjecture, with the information dimension of $\mu$ replaced by the box-counting dimension of $X$ (see Definition~\ref{defn:dim}). More precisely, we prove the key estimate for the decay rate of the measure of the set of points with given prediction error (i.e.~the upper bound in assertion~(ii) of the conjecture) with an exponent arbitrarily close to $k - D$, and confirm assertions~(iii)--(iv).
The case considered in~(ii), where $\phi$ is not injective, is precisely the case referred to in
the title of this paper, i.e.~`time-delayed measurements with self-intersections'.

Furthermore, in Section \ref{sec:ex} we provide an example showing that assertion~\ref{it: ssoy2 ii} of the conjecture does not hold for the information or Hausdorff dimension of $\mu$ for an arbitrary Borel probability measure $\mu$ (we should however emphasize that the counterexamples are not in the class of natural measures). This is in contrast with the SSOY predictability conjecture, which holds for the Hausdorff dimension of $\mu$. Similarly as our previously mentioned results, the theorem is formulated in a general category of Lipschitz transformations on compact sets in Euclidean space and arbitrary Borel probability measures. Note that in the formulation of the result, there appear both upper and lower box-counting dimension of $X$, denoted respectively by $\udim$ and $\lbdim$ (see Definition~\ref{defn:dim}).

\begin{thm}[{\bf Prediction error estimate}{}]\label{thm:SSOY2}
Let $X \subset \R^N$, $N \in \N$, be a compact set, $\mu$ a Borel probability measure on $X$ and $T\colon X \to X$ a Lipschitz map. Set $\overline{D} = \udim X$, $\underline{D} = \lbdim X$. Then there is a prevalent set $\mathcal S$ of Lipschitz observables $h \colon X \to \R$, such that for every $h \in \mathcal S$, $k \in \N$, $\delta, \theta > 0$, there exist $C, \eps_0 > 0$, such that for the $k$-delay coordinate map $\phi$ defined in \eqref{eq:coord_map} and every $0 < \eps < \eps_0$, the following hold.
\begin{enumerate}[$($i$)$]
\item\label{it:SSOY2 error bound} If $k > \overline{D}$, then
\[ 
\mu\left( \{ x \in X : \sigma_{\eps}(\phi(x)) > \delta \} \right) \leq C \eps^{k - \overline{D} - \theta}.
\]
\item\label{it:SSOY2 deterministic} If $k > 2\underline{D}$ or $\phi$ is injective, then 
\[ 
\{ x \in X : \sigma_{\eps}(\phi(x)) > \delta \} = \emptyset.
\]
\end{enumerate}
\end{thm}

\begin{rem}\label{rem:SSOY2 invariant}
If $\mu$ is additionally $T$-invariant, then the restriction of $T$ to the support of $\mu$ also fulfils the assumptions of the theorem, which implies that in this case in assertions~(i)--(ii) one can replace the upper (resp.~lower) box dimension of $X$ by the upper (resp.~lower) box dimension of the support of $\mu$.
\end{rem}

\begin{rem}\label{rem:dynamical conj2 udim existence}
In (i) we actually prove the corresponding bound for the measure of a larger set, i.e.~$\{ x \in X :  \|\phi(x) - \phi(z)\| \leq \eps  \text{ and }  \|\phi(Tx) - \phi(Tz)\| > \delta \text{ for some } z \in X\}$. Note also that (ii) strengthens assertions~(iii)--(iv) of the SSOY prediction error conjecture, showing that the considered sets are not only of $\mu$-measure zero, but are in fact empty. 
\end{rem}

The proof of Theorem~\ref{thm:SSOY2} is presented in Section~\ref{sec:proof_ssoy}. Again, we show an extended version of the result (Theorem~\ref{thm:SSOY2_ext} and Corollary~\ref{cor:Zero prediction error_extended}), valid for H\"older observables.

\subsection{Another perspective on predictability}

As explained above, the approach to predictability taken in \cite{SSOY98} through the definition of the prediction errors $\sigma_\eps$ can be motivated by its relation to prediction algorithms, like the Farmer--Sidorowich algorithm. However, one can consider another, apparently more straightforward approach. It is based on the premise that predictability is precisely the phenomenon whereby the first $k$ terms of the time series \eqref{eq:inf_time_series} of an observable $h$ at a point $x \in X$ uniquely determine its ($k+1$)-th term. Equivalently, the condition states that the value of the corresponding $k$-delay coordinate map $\phi$ at a point $x$ determines the value of $\phi$ at the point $Tx$. We call this property deterministic predictability, as stated in the following definition.

\begin{defn}[{\bf Deterministic predictability}{}]\label{def:deterministic predictability}
Let $(X,T)$ be a dynamical system consisting of a map $T \colon X \to X$ on a set $X \subset \R^N$ and let $k \in \N$. We say that an observable $h \colon X \to \R$ is  \emph{deterministically} $k$-\emph{predictable} at a point $y \in \phi(X)$, if for the $k$-delay coordinate map $\phi$ defined in \eqref{eq:coord_map}, the map $\phi\circ T$ is constant on $\phi^{-1}(\{y\})$. If $h$ is deterministically $k$-predictable at all points $y \in \phi(X)$, then we say that $h$ is deterministically $k$-predictable.
\end{defn}

\begin{rem}\label{rem: usefulness of predictability} 
Note that in some special cases, the deterministic predictability of an observable holds in an obvious way. First, this may be caused by the triviality (and hence predictability) of the dynamics. For instance, if $T^p = \Id$ for some $p \in \{1, \ldots, k\}$, then every observable is deterministically $k$-predictable. Such a situation can be avoided by assuming some bound on the size of the sets of periodic points of low period, which hold for a large class of systems (cf.~Remark~\ref{rem:periodic points}). Secondly, even in the case of complicated dynamics, some observables (for example, the constant ones) are deterministically predictable for all $k \in \N$. However, this is not the case, when one considers typical (prevalent) perturbations of a given observable $h$. The results of our recent paper \cite{BGS-predict_few_meas} show that for a large class of Lipschitz systems on compact spaces (with suitable bound on the size of the sets of periodic points of low period), a prevalent Lipschitz observable is not deterministically $k$-predictable for $k$ smaller than $\dim_H X$. In particular, a perturbation of any observable $h$ (even a constant one) is typically non-predictable, if the delay length $k$ is chosen too small. This shows that the notion of deterministic predictability is non-trivial and can be useful in the study of typical perturbations of observables in the context of time-delayed measurements.
\end{rem}

The subsequent proposition is almost immediate (the proof is presented in Section~\ref{sec:proof_takens}).

\begin{prop}\label{prop:determ_predict_equiv}
The following conditions are equivalent:
\begin{enumerate}[$($a$)$]
\item $h$ is deterministically predictable,
\item for every $x_1, x_2\in X$, if $\phi(x_1)  = \phi(x_2)$, then $\phi(Tx_1) = \phi(Tx_2)$,
\item there exists a map $S \colon \phi(X) \to \phi(X)$ such that $\phi(T x) = S(\phi(x))$ for every $x \in X$, i.e.~the diagram
\[
\begin{CD}
	X @>T>> X\\
	@VV\phi V @VV \phi V\\
	\phi(X) @>S>> \phi(X)
\end{CD}
\]
commutes.
\end{enumerate} 
Moreover, if $X$ is compact and $T$, $h$ are continuous, then the map $S$ is continuous.
\end{prop}

Note that whenever $S$ exists, it is given by
\[S(h(x), h(Tx), \ldots, h(T^{k-1}x)) = (h(Tx), h(T^2 x), \ldots, h(T^k x)).\]
It follows that if $h$ is deterministically $k$-predictable, then knowing the first $k$ terms of its time series, one can determine all its further terms by iterating the map $S$.

Consider now a probabilistic setting, when the phase space $X$ admits some probability measure $\mu$, and one predicts measurements performed at a typical ($\mu$-almost every) point. In view of Proposition~\ref{prop:determ_predict_equiv}, we define the following `almost sure' version of deterministic predictability.

\begin{defn}[{\bf Almost-sure deterministic predictability}{}]\label{def:almost_sure_deterministic predictability} Let $X \subset \R^N$, $N \in \N$, be a Borel set with a Borel probability measure $\mu$. Let $T \colon X \to X$ be a Borel transformation, $h \colon X \to \R$ a Borel observable and $k \in \N$. We say that $h$ is \emph{almost surely deterministically} $k$-\emph{predictable} with respect to $\mu$, if there exists a Borel set $X_h \subset X$ of full $\mu$-measure such that for every $x_1, x_2 \in X_h$, if $\phi(x_1)  = \phi(x_2)$, then $\phi(Tx_1) = \phi(Tx_2)$. Equivalently, there exists a map $S \colon \phi(X_h) \to \phi(X)$ such that $\phi(T x) = S(\phi(x))$ for every $x \in X_h$, i.e.~the diagram
\[
\begin{CD}
	X_h @>T>> X\\
	@VV\phi V @VV \phi V\\
	\phi(X_h) @>S>> \phi(X)
\end{CD}
\]
commutes and
\begin{equation}\label{def:S}
S(h(x), \ldots, h(T^{k-1}x)) = (h(Tx), \ldots, h(T^k x)) \quad \text{for $\mu$-almost every } x \in X.     
\end{equation}
\end{defn}

\begin{defn} In the context of Proposition~\ref{prop:determ_predict_equiv} and Definition~\ref{def:almost_sure_deterministic predictability}, we call the map $S$ (wherever it is defined) the \emph{prediction map}. The space $\phi(X)$ (or $\phi(X_h)$) is called the \emph{model space}, and the number $k$ is the \emph{delay length} or \emph{delay dimension}.\footnote{ In the literature it is more common to find the terminology \emph{reconstruction space} or \emph{embedding space} for $\phi(X)$ as well as \emph{embedding dimension} for $k$. Note, however, that this terminology is inadequate when $\phi$ is not injective.} 
\end{defn}

\begin{rem}\label{rem:almost_sure_deterministic predictability} In the case of deterministic predictability (Definition~\ref{def:deterministic predictability} and Proposition~\ref{prop:determ_predict_equiv}), the dynamical system $(\phi(X), S)$ is a topological factor (\emph{model system}) of $(X, T)$. In contrast, in the case of almost sure deterministic predictability (Definition~\ref{def:almost_sure_deterministic predictability}), the prediction map $S$ is not guaranteed to map $\phi(X_h)$ into itself, hence it might not be possible to iterate it. However, if we assume the measure $\mu$ to be $T$-invariant, then the set $X_h$ can be chosen to satisfy $T(X_h) \subset X_h$ (it is enough to replace $X_h$ by $\tilde X_h = \bigcap_{n=0}^\infty T^{-n} (X_h)$, cf.~\cite[Remark~4.4(a)]{BGS20}). Then $\phi(X_h)$ is $S$-invariant and the system $(\phi(X_h), \phi_*\mu, S)$ is a measurable factor of $(X_h, \mu, T|_{X_h})$, with the commuting diagram
\[
\begin{CD}
	X_h @>T>> X_h\\
	@VV\phi V @VV \phi V\\
	\phi(X_h) @>S>> \phi(X_h)
\end{CD},
\]
providing an \emph{almost sure model} of the system $(X, T)$.
\end{rem}

The following proposition describes properties of continuous almost surely deterministically predictable observables for continuous systems.

\begin{prop}\label{prop:almost_sure_predict_properties}
Let $X$ be a Borel set, $\mu$ a Borel probability measure on $X$ and $T\colon X \to X$ a continuous map. Let $h \colon X \to \R$ be a continuous observable, which is almost surely deterministically $k$-predictable with respect to $\mu$ for some $k \in \N$. Then the following hold.
\begin{enumerate}[$($i$)$]
\item The set $X_h$ from Definition~{\rm\ref{def:almost_sure_deterministic predictability}} can be chosen such that the model space $\phi(X_h)$ is a Borel set and the prediction map $S$ is a Borel map.
\item If $\phi\circ T \in L^1(\mu)$ $($e.g.~if $h$ is bounded$)$, then $S \in L^1(\phi_*\mu)$ and
\[
\chi_\eps(y) = \frac{1}{\phi_*\mu(B(y,\eps))} \int_{B(y,\eps)} S \, d\phi_*\mu  \xrightarrow[\eps\to 0]{} S(y)
\]
for $\phi_*\mu$-almost every $y \in \phi(X)$, where $\chi_\eps$ is introduced in Definition~{\rm\ref{def:dynamical predictability}}.

\end{enumerate}

\noindent
Moreover, if the measure $\mu$ is $T$-invariant, then the assumption $\phi\circ T \in L^1(\mu)$ in assertion~$($ii$)$ can be replaced by $h \in L^1(\mu)$.
\end{prop}

The proof of Proposition~\ref{prop:almost_sure_predict_properties} is presented in Section~\ref{sec:proof_predict_prob}.

Our next result shows, in particular, a surprising fact, that the two probabilistic notions of predictability, presented in Definitions~\ref{def:SSOYpredictability} and~\ref{def:almost_sure_deterministic predictability}, coincide for continuous systems on compact sets.

\begin{thm}[{\bf Relations between two notions of almost sure predictability}{}]\label{thm:almost_sure_predict_equiv}

Let $X$ be a Borel set, $\mu$ a Borel probability measure on $X$ and $T\colon X \to X$ a continuous map. Let $h \colon X \to \R$ be a continuous observable and $k \in \N$. Consider the $k$-delay coordinate map $\phi$ defined in \eqref{eq:coord_map} and the prediction map $S$. Then the following hold.
\begin{enumerate}[$($i$)$]
\item If $\phi\circ T \in L^2(\mu)$ $($e.g.~if $h$ is bounded$)$ and $h$ is almost surely deterministically $k$-predictable with respect to $\mu$, then $S \in L^2(\phi_*\mu)$ and 
\[
\sigma_\eps(y) = \bigg(\frac{1}{\phi_*\mu(B(y,\eps))} \int_{B(y,\eps)} \|S - \chi_\eps(y)\|^2 \, d\phi_*\mu\bigg)^{\frac 1 2}  \xrightarrow[\eps\to 0]{} 0
\]
for $\phi_*\mu$-almost every $y \in \phi(X)$ and $\chi_\eps, \sigma_\eps$ from Definition~{\rm\ref{def:dynamical predictability}}, so $h$ is almost surely $k$-predictable with respect to $\mu$.

\item If $X$ is compact, then $h$ is almost surely deterministically $k$-predictable with respect to $\mu$ if and only if it is almost surely $k$-predictable with respect to $\mu$.
\end{enumerate}

\noindent
Moreover, if the measure $\mu$ is $T$-invariant, then the assumption $\phi\circ T \in L^2(\mu)$ in assertion~$($i$)$ can be replaced by $h \in L^2(\mu)$.
\end{thm}

The proof of Theorem~\ref{thm:almost_sure_predict_equiv} is presented in Section~\ref{sec:proof_predict_prob}. In view of this result, if $X$ is compact and $T$ is continuous, then one can use the term \emph{almost sure predictability} for a continuous observable $h$ in the context of both Definitions~\ref{def:SSOYpredictability} and~\ref{def:almost_sure_deterministic predictability}. 

The notion of almost sure deterministic predictability enables one to gain a new perspective on the Farmer--Sidorowich algorithm. Indeed, if an observable $h \colon X \to \R$ is almost surely deterministically $k$-predictable, then the points $z_i$ defined in \eqref{eq:y_i} satisfy
\[
z_i = S^i(z_0), \qquad i = 0, \ldots, n
\]
for $\phi_*\mu$-almost every $z_0 \in \R^k$. Hence, by Proposition~\ref{prop:FS}, Proposition~\ref{prop:almost_sure_predict_properties} and Theorem~\ref{thm:almost_sure_predict_equiv}, we immediately obtain the following corollary. 

\begin{cor}\label{cor:FS-predict} Let $X$ be a Borel set, $T\colon X \to X$ a continuous map and $\mu$ an ergodic $T$-invariant Borel probability measure on $X$. Suppose an observable $h \colon X \to \R$ is continuous, $h \in L^1(\mu)$ and $h$ is almost surely deterministically $k$-predictable with respect to $\mu$ for some $k \in \N$. Consider $z_i = z_i(x)$, $i = 0, \ldots, n$ and $\Pred_{x,n,\eps}(y)$, $\Var_{x,n,\eps}(y)$, defined respectively in \eqref{eq:y_i}--\eqref{eq:y_hat} and \eqref{eq:Var} for $x \in X$, $y \in \R^k$, $n \in \N$, $\eps > 0$. Then
\[
\lim_{\eps\to 0} \lim_{n \to \infty} \Pred_{x,n,\eps}(y) = \lim_{\eps\to 0} \chi_\eps(y) = S(y)
\]
for $\mu$-almost every $x \in X$ and $\phi_*\mu$-almost every $y \in \R^k$. Moreover, if $h \in L^2(\mu)$, then
\[
\lim_{\eps\to 0}\lim_{n \to \infty} \Var_{x,n,\eps}(y) =  \lim_{\eps\to 0} (\sigma_\eps(y))^2 = 0
\]
for $\mu$-almost every $x \in X$ and $\phi_*\mu$-almost every $y \in \R^k$.
\end{cor}

It follows that under the assumptions of Corollary~\ref{cor:FS-predict}, at a typical point $y$, the vector $\chi_{\eps}(y)$ for small $\eps > 0$ may be interpreted as a limit of \emph{empirical means} of the prediction map $S$ associated with the Farmer--Sidorowich algorithm, while $\sigma_\eps(y)$ is a suitable limit of \emph{empirical standard deviations} of $S$.

\subsection{Time-delay prediction theorems} As remarked above, deterministic predictability can be used for predicting future values of a given observable, rather than for a faithful reconstruction of the original dynamics. On the other hand,  deterministic predictability holds whenever the conclusion of the Takens delay embedding theorem is satisfied, i.e.~whenever the delay coordinate map $\phi$ is injective on $X$ (with the prediction map given as $S = \phi \circ T \circ \phi^{-1}$). In the language of the theory of dynamical systems, deterministic predictability means that $\phi$ is a semi-conjugation between the original system $(X,T)$ and its factor $(\phi(X), S)$, while injectivity of $S$ means that the two systems are isomorphic (conjugate). Obviously, in many cases an observable $h$ can be deterministically predictable, while $\phi$ is not injective. Therefore, it is natural to expect that typical deterministic predictability may hold under weaker assumptions than the ones required for Takens-type delay embedding theorems. 
Indeed, a result of this kind (called a time-delay embedding theorem for non-injective smooth dynamics) was proved by Takens \cite{T02} in 2002. More precisely, he showed that if $X$ is a compact $C^1$-manifold and $k > 2 \dim X$, then for a typical (generic in the Baire sense) pair of a $C^1$-map $T\colon X \to X$ and a $C^1$-observable $h\colon X \to \R$, the $k$-delay coordinate map $\phi$ determines the $(k-1)$-th iterate of $T$, i.e.~there exists a (Lipschitz) map $\pi\colon \phi(X) \to X$ such that 
\begin{equation}\label{eq:pi}
\pi \circ \phi = T^{k-1}.
\end{equation}
This conclusion is clearly weaker than the injectivity of $\phi$ in the classical time-delay embedding results (in fact, without assuming the injectivity of $T$, the injectivity of $\phi$ does not hold even in generic sense, see \cite[Section~2]{T02}). On the other hand, the result implies typical deterministic predictability, as one can define the prediction map $S$ in terms of $\pi$ as 
\[
S(y) = S(y_1, \ldots, y_k) = (y_2, \ldots, y_k, h(T(\pi(y_1, \ldots, y_k)))).
\]
Similar results have been recently obtained in the topological category \cite{kato2021jaworski,Kato23} for some classes of topological spaces.

In this paper we prove a general deterministic time-delay prediction theorem in the setup of locally Lipschitz maps on arbitrary sets in Euclidean space. In comparison to the Takens-type delay embedding theorem by Sauer, Yorke and Casdagli \cite{SYC91}, quoted in Subsection~\ref{subsec:intro}, we remove the assumption of the injectivity of $T$ and the bound on the dimension of the sets of periodic points, to obtain deterministic predictability rather than the injectivity of the delay coordinate map.

\begin{thm}[{\bf Deterministic time-delay prediction theorem}{}]\label{thm:predict_determ}
Let $X \subset \R^N$, $N \in \N$, and let $T\colon X \to X$ be a locally Lipschitz map. Then a prevalent locally Lipschitz observable $h \colon X \to \R$ is deterministically $k$-predictable for $k > 2\lbdim X$, hence the diagram
\[
\begin{CD}
	X @>T>> X\\
	@VV\phi V @VV \phi V\\
	\phi(X) @>S>> \phi(X)
\end{CD}
\] 
commutes, where $\phi$ is the $k$-delay coordinate map defined in \eqref{eq:coord_map} and $S$ is the prediction map. Furthermore, if $X$ is compact, then $S$ is continuous.
\end{thm}

The proof of this result (in fact, its extended version given by Theorem~\ref{thm:predict_determ_ext}) is presented in Section~\ref{sec:proof_takens}. 

\begin{rem}
The existence of the map $\pi$ from \eqref{eq:pi} requires a bound on the dimension of periodic points.\footnote{ Indeed, let $X$ be a topological space which does not embed into $\R$ and define $T$ to be the identity on $X$. Then it is easy to see that a map $\pi$ satisfying $\pi \circ \phi = T^{k-1}=\Id$ exists if and only if $h$ is injective on $X$. This is impossible as $X$ does not embed into $\R$.} Therefore, one cannot hope to improve Theorem~\ref{thm:predict_determ} to obtain a result analogous to the one in \cite{T02}.
\end{rem}

In our previous paper \cite{BGS20}, we studied the problem of determining conditions under which a system $(X, T)$ can be \emph{almost surely} reconstructed by the $k$-delay coordinate map, with respect to some probability distribution on the phase space $X$. In \cite[Theorems~1.2 and~4.3]{BGS20}, we proved that for an injective, (locally) Lipschitz system $(X, T)$ on a Borel set in Euclidean space with a Borel probability measure $\mu$ on $X$, the $k$-delay coordinate map corresponding to a prevalent (locally) Lipschitz observable is injective on a set of full $\mu$-measure, provided $k > \hdim \mu$ and the Hausdorff dimension of $\mu$ restricted to the set of $p$-periodic points is smaller than $p$ for every $p =1, \ldots, k-1$. This agrees with the general philosophy of the SSOY conjectures, where the number of measurements can be reduced by half, if one is allowed to neglect sets of measure zero. Following the above approach, we can establish the following `almost sure' version of Theorem~\ref{thm:predict_determ}.

\begin{thm}[{\bf Almost sure time-delay prediction theorem}{}]\label{thm:predict_prob}
Let $X \subset \R^N$, $N \in \N$, be a Borel set, $\mu$ a Borel probability measure on $X$ and $T\colon X \to X$ a locally Lipschitz map. Then a prevalent locally Lipschitz observable $h \colon X \to \R$ is almost surely deterministically $k$-predictable with respect to $\mu$ for $k > \hdim \mu$. If $\mu$ is $T$-invariant, then the set $X_h$ from Definition~{\rm\ref{def:almost_sure_deterministic predictability}} can be chosen to satisfy $T(X_h) \subset X_h$, so that the diagram
\[
\begin{CD}
	X_h @>T>> X_h\\
	@VV\phi V @VV \phi V\\
	\phi(X_h) @>S>> \phi(X_h)
\end{CD}
\]
commutes, where $\phi$ is the $k$-delay coordinate map defined in \eqref{eq:coord_map}, $X_h$ is a full $\mu$-measure Borel subset of $X$, $\phi(X_h)$ is a Borel set and $S$ is a Borel prediction map. Furthermore, if $X$ is compact and $k > \hdim(\supp \mu)$, then $S$ is continuous for a prevalent locally Lipschitz observable.
\end{thm}

\begin{rem}
If, additionally, $X$ is closed, $T$ is Lipschitz and $\mu$ is $T$-invariant and ergodic, then the same conclusions hold under the assumption $k > \lid(\mu)$ (cf.~Remark~\ref{rem:SSOY-ID}).
\end{rem}

\begin{rem}
In a recent paper \cite{BGS-predict_few_meas} we show that if $X$ is compact, $T$ is Lipschitz and $\mu$ is $T$-invariant, then for a prevalent Lipschitz observable $h$, the corresponding $k$-delay coordinate map is not injective on a set of full $\mu$-measure for $k < \hdim \mu$. If $\mu$ is additionally ergodic, then a prevalent Lipschitz observable $h$ is not almost surely (deterministically) $k$-predictable for $k < \hdim \mu$. Consequently, $\hdim \mu$ is the precise threshold for the number of delay measurements needed (and sufficient) for typical $\mu$-almost sure reconstruction and prediction for such systems.
\end{rem}

Note that Theorem~\ref{thm:SSOY1} follows directly from Theorem~\ref{thm:predict_prob} and Theorem~\ref{thm:almost_sure_predict_equiv}. In Section~\ref{sec:proof_predict_prob} we prove an extended version of Theorem~\ref{thm:predict_prob}, given by Theorem~\ref{thm:predict_prob_ext}. In Section~\ref{sec:ex} we provide an example showing that neither the assumption $k > \hdim \mu$ nor the assumption $k > \uid(\mu)$ is sufficient to obtain the continuity of the prediction map $S$ for a prevalent set of observables.

\subsection{Estimating the  minimal delay length}\label{subsec:eff dim}

Oftentimes within the timed-delayed measurement framework, experiments are performed on an underlying system whose key attributes, e.g.~its dimension, are unknown.
In such situations, one is faced with the challenge of estimating the delay length $k$ from a sample of finite time series as in \eqref{eq:timeseries},  with the goal of achieving a reasonable model for the system. Manifestly what is considered reasonable depends on the  application at hand (see \cite[Chapter 4]{Abarbanel_book} for an overview of algorithms aiming at this goal). 
A basic algorithm of this kind is the \emph{false nearest neighbor} algorithm introduced in \cite{KBA92}  (see also  \cite[Section 4.2]{Abarbanel_book}; similar algorithms were introduced in \cite{CenysPyragas88, LiebertPawelzikSchuster91}). Let us give a brief overview of this algorithm. For a given $k$ (a candidate for the delay length) one considers points $\phi(x_i) \in \R^k,\ i = 1, \ldots, m$ obtained from the measurements \eqref{eq:timeseries}. For each point $\phi(x_i)$ one considers its nearest neighbor $\phi(x_j)$ in the $k$-dimensional model space.
If the distance $|h(T^k x_i) - h(T^k x_j)|$ is much larger than the distance $\| \phi(x_i) - \phi(x_j) \|$ (relative to some a priori threshold), then $\phi(x_j)$ is called a \emph{false nearest neighbor} of $\phi(x_i)$ (as in such case one expects $x_i$ and $x_j$ to be far away in the original phase space $X$, implying that $\phi(x_i)$ and $\phi(x_j)$ are close due to a poor performance of $\phi$). Next, one computes the proportion (relative to $m$) of points which have a false nearest neighbor.
The smallest $k$ for which this proportion is smaller than a fixed a priori rate is chosen as the delay length. The false nearest neighbor algorithm is often employed in practice (see e.g.~\cite{Oscilators, SeaClutter, BioclimaticBuildings}, and \cite{OzgeSerkan19, Robots, Dance, BaghReddy23} for some recent applications).
It is also being used and studied in the literature on numerical methods for modelling and analysing chaotic dynamics \cite{CoupledSynchornization95, NoisyFNN97, KrakovkaFNN15, OzgeSerkan19}. 
Interestingly, for a number of dynamical systems with known dimension of the phase space, it has been observed that the delay length estimated by the above (or similar) algorithms is smaller than the theoretical Takens delay embedding theorem bound of $2\dim X$, see \cite[Section 2.4.1]{Abarbanel_book} and \cite{KBA92, CenysPyragas88, LiebertPawelzikSchuster91} for examples including the Lorenz and R\"ossler systems, the H\'enon map and the dynamics of stochastic self-modulation of waves in a non-stationary medium.

It seems plausible that the SSOY prediction error conjecture together with our prediction error estimate (Theorem~\ref{thm:SSOY2}) present a correct framework for a mathematical explanation of this numerically observed phenomenon.
Indeed, note that the algorithm described above tests how well future values of the time series are predicted from the previous ones, studying in fact the \emph{predictability} problem rather than the \emph{reconstruction} (embedding) problem. Furthermore, one is interested in having a satisfactory prediction property \emph{with high enough probability}, rather then for \emph{all} sampled points, hence adopting a probabilistic point of view. 
Moreover, the quantitative statement of assertion~(ii) of the SSOY prediction error conjecture as well as assertion~(i) of Theorem~\ref{thm:SSOY2} provide a method for determining the dimension of the phase space as well as the desired delay length.\footnote{ Note that the quantity $\mu(\{x \in X: \sigma_\eps(\phi(x)) > \delta\})$ has the shortcoming that no information on the local behaviour of $\sigma_\eps(\phi(x))$ is obtained. In \cite{McSharrySmith04}, an illuminating local  analysis is performed on several numerical examples. Such analysis may lead to a better choice of the delay length.} 

\subsection*{Structure of the paper} Section~\ref{sec:prelim} contains preliminary definitions and results, as well as a discussion on the notion of prevalence and its variants used in this paper. In Section~\ref{sec:combin} we prove a crucial combinatorial proposition (Proposition~\ref{prop: rank or predict}), and derive a corollary (Corollary~\ref{cor:crucial}), which is essential for the proofs of the main results of the paper. Section~\ref{sec:trans} provides general estimates of the Lebesgue measure of some sets of parameters related to delay coordinate maps for  perturbations of Lipschitz observables. In Section~\ref{sec:proof_takens} we prove Theorem~\ref{thm:predict_determ_ext}, which provides an extended version of the deterministic time-delay prediction theorem (Theorem~\ref{thm:predict_determ}). Section~\ref{sec:proof_predict_prob} contains the proofs of the results on the properties of continuous almost sure deterministically predictable observables and the relations between the two notions of almost sure predictability (Proposition~\ref{prop:almost_sure_predict_properties} and Theorem~\ref{thm:almost_sure_predict_equiv})
as well as Theorem~\ref{thm:predict_prob_ext}, which is an extended version of the almost sure time-delay prediction theorem (Theorem~\ref{thm:predict_prob}). In Section~\ref{sec:proof_ssoy} we prove extended versions of the general SSOY predictability conjecture (Theorem~\ref{thm:SSOY1}), given by Theorem~\ref{thm:SSOY1_ext}, and the prediction error estimate (Theorem~\ref{thm:SSOY2}), split into three results (Theorem~\ref{thm:det pred implies zero error}, Corollary~\ref{cor:Zero prediction error_extended} and Theorem~\ref{thm:SSOY2_ext}). In Section~\ref{sec:ex} we provide examples showing that Theorems~\ref{thm:SSOY2} and~\ref{thm:predict_prob} do not hold under weaker assumptions on the dimension of a given measure (Corollary~\ref{cor:no continuity with hdim} and Proposition~\ref{prop:atomic prevalence}). The paper concludes with a list of open questions related to the considered problems (Section~\ref{sec:open_prob}).

\section{Preliminaries}\label{sec:prelim}

\subsection*{Notation}
The number of elements of a set $A$ is denoted by $\# A$. 
The symbols $\| \cdot \|$, $\dist(\cdot, \cdot),\ \langle \cdot, \cdot \rangle$ and $| \cdot |$ denote, respectively, the Euclidean norm, distance, inner product and diameter in $\R^N$, $N \in \N$. We write $\conv A$ for the convex hull of a set $A \subset \R^N$. The open $\delta$-ball around a point $x \in \R^N$ is denoted by $B_N(x, \delta)$ and the closed ball by $\overline{B_N}(x, \delta)$ (sometimes we omit the dimension $N$ in notation). By $\Leb$ we denote the Lebesgue measure (or the outer Lebesgue measure, in case of non-measurable sets).

\subsection*{Singular values}

Let $\psi \colon \R^m \to \R^k$ be a linear map and let $A$ be the matrix of $\psi$.
It is a classical fact that the following numbers are equal: the dimension of the image of $\psi$; the difference between $k$ and the dimension of the kernel of $\psi$; the maximal number of linearly independent columns of $A$; the maximal number of linearly independent rows of $A$. This common value is called the \emph{rank} of $A$ and is denoted by $\rank A
$.

For $p \in \{1, \ldots, k\}$ let $\sigma_p(A)$ (we also use the symbol $\sigma_p(\psi)$) be the $p$-th largest \emph{singular value} of $A$, i.e.~the $p$-th largest square root of an eigenvalue of the matrix $A^*A$. The following fact is standard (see e.g.~\cite[Lemma 14.2]{Rob11}). 

\begin{lem}\label{lem:rank=singular}
The rank of a matrix equals the number of its non-zero singular values.
\end{lem}

We will also use the following lemma, proved as \cite[Lemma~4.2]{SYC91} (see also \cite[Lemma~14.3]{Rob11}).

\begin{lem}\label{lem: key_ineq_inter}
Let $\psi \colon  \R^m \to \R^k$ be a linear transformation. Assume $\sigma_p(\psi) > 0$ for some $p \in \{1, \ldots, k\}$. Then for every $z \in \R^k$ and $\rho, \eps > 0$,
\[ \frac{\Leb(\{ \alpha \in B_m(0, \rho) : \|\psi(\alpha) + z \| \leq \eps \})}{\Leb (B_m(0, \rho))} \leq C \Big(\frac{\eps}{\sigma_p(\psi) \, \rho}\Big)^p, \]
where $C > 0$ depends only on $m,k$.
\end{lem}
\subsection*{Prevalence}

Below we present the definition of prevalence -- a notion introduced by Hunt, Sauer and Yorke in \cite{Prevalence92} (an earlier version of this notion appeared in \cite{Christensen73}, see also \cite{BYbook00}), which may be considered as an analogue of `Lebesgue almost sure' condition in infinitely dimensional normed linear spaces.

\begin{defn}\label{defn:preval} By a (complete) \emph{linear metric} in a linear space we mean a (complete) metric which makes the addition and scalar multiplication continuous. Let $V$ be a complete linear metric space (i.e.~a linear space with a complete linear metric). A Borel set $\mathcal S \subset V$ is called \emph{prevalent} if there exists a Borel measure $\nu$ in $V$, which is positive and finite on some compact set in $V$, such that for every $v \in V$, we have $v + e \in \mathcal S$ for $\nu$-almost every $e \in V$. A non-Borel subset of $V$ is prevalent if it contains a prevalent Borel subset. For more information on prevalence we refer to \cite{Prevalence92} and \cite[Chapter~5]{Rob11}.
\end{defn} 

Let us now describe prevalence in the spaces of (locally) Lipschitz and H\"older observables. Consider $X \subset \R^N$ and $\beta \in (0,1]$. Recall that a function $h\colon X \to \R$ is $\beta$-\emph{H\"older} if there exists $c > 0$ such that $|h(x) - h(y)| \le c \|x-y\|^\beta$ for every $x, y \in X$. 
We say that a function $h \colon X \to \R$ is \emph{locally} $\beta$-\emph{H\"older} on $X$, if for every $x \in X$ there exists an open neighbourhood $U$ of $x$ such that $h|_{U \cap X}$ is $\beta$-H\"older. We denote the space of $\beta$-H\"older (resp.~locally $\beta$-H\"older) functions on $X$ by $H^\beta(X)$ (resp.~$H^\beta_{loc}(X)$). A function $h\colon X \to \R$ is \emph{Lipschitz} (resp.~\emph{locally Lipschitz}) if it is $1$-H\"older (resp.~locally $1$-H\"older). 
The space of Lipschitz (resp.~locally Lipschitz) functions on $X$ is denoted by $\Lip(X)$ (resp.~$\Lip_{loc}(X))$.

For $X \subset \R^N$, $\beta \in (0,1]$ and $h\colon X \to \R$ let 
\[ \| h\|_{\beta, X} = \sup \limits_{x \in X} |h(x)| + \sup \limits_{\substack{x, y \in X \\ x \neq y}} \frac{|h(x) - h(y)|}{\|x-y\|^\beta}\]
(where $\|h\|_{\beta, X}$ can be infinite). 

Suppose $X$ is bounded. Then
\[
H^\beta(X) = \left\{ h \colon X \to \R : \|h\|_{\beta, X} < \infty \right\}.
\]
In this case, it is a standard fact that $H^\beta(X)$ endowed with the $\beta$-\emph{H\"older norm} $\| \cdot \|_{H^\beta(X)} = \| \cdot \|_{\beta, X}$ is a Banach space (in particular, a complete linear metric space). We also set $\|\cdot\|_{\Lip(X)} = \| \cdot \|_{H^1(X)}$ to be the \emph{Lipschitz norm} on $\Lip(X)$. 

Consider now the space $H^\beta_{loc}(X)$ for a (possibly unbounded) set $X \subset \R^N$. Clearly, $H^\beta_{loc}(X)$ is a linear space. To introduce a linear metric in $H^\beta_{loc}(X)$, fix a countable basis $\mB = \{ U_1, U_2, \ldots \}$ of the standard topology in $\R^N$, such that $\mB$ consists of bounded sets. For functions $h, \tilde h \in H^\beta_{loc}(X)$ we define
\[ d_{\beta, X}(h, \tilde h) = \sum \limits_{n=1}^\infty 2^{-n} \min \big\{1, \|h - \tilde h\|_{\beta, U_n \cap X} \big\}\]
(note that it might happen that $\|h - \tilde h\|_{\beta, U_n \cap X}$ is infinite). It is not difficult to check that $d_{\beta, X}(\cdot, \cdot)$ defines a complete linear metric in $H^\beta_{loc}(X)$. In particular, $d_{1, X}(\cdot, \cdot)$ is a complete linear metric in $\Lip_{loc}(X)$.

Let us emphasize that even if $X$ is bounded, it might happen $H^\beta(X) \neq H^\beta_{loc}(X)$. However, if $X$ is compact, then $H^\beta(X) = H^\beta_{loc}(X)$, with $\| \cdot \|_{H^\beta(X)}$ and $d_{\beta, X}(\cdot, \cdot)$ generating the same topology. 


In this paper we use the notion of prevalence in the sense of Definition~\ref{defn:preval} applied suitably to the spaces $V = H^\beta(X)$ or  $V = H^\beta_{loc}(X)$. 

\begin{rem}\label{rem:probe set} 
Similarly as in \cite{SYC91, Rob11, BGS22}, to prove prevalence of appropriate sets of observables, we check the following condition, which is sufficient for prevalence. Let $\{h_1, \ldots, h_m\}$, $m \in \N$, be a finite set of functions in $V = H^\beta(X)$ or $V = H^\beta_{loc}(X)$, called the \emph{probe set}. Define $\xi \colon \R^m \to V$ by $\xi(\alpha_1, \ldots, \alpha_m) = \sum_{j=1}^m \alpha_j h_j$. Then $\nu = \xi_*\Leb$ (cf.~Definition~\ref{defn:support}), where $\Leb$ is the Lebesgue measure in $\R^k$, is a Borel measure in $V$, which is positive and finite on the compact set $\xi(\overline{B_m}(0,1))$. For this measure, a sufficient condition for a set $\mathcal S \subset V$ to be prevalent is that for every $h \in V$, the function $h + \sum_{j=1}^m \alpha_j h_j$ is in $\mathcal S$ for Lebesgue-almost every $(\alpha_1, \ldots, \alpha_m) \in \R^m$. In this case, we say that $\mathcal S$ is \emph{prevalent} in $V$ \emph{with the probe set} $\{h_1, \ldots, h_m\}$. Following \cite{BGS20,BGS22}, as probe sets we take suitable interpolating families of functions in $V$, defined below.
\end{rem}

\begin{defn}\label{defn:interpol}
Let $X$ be a subset of $\R^N$. A family of functions $h_1, \ldots, h_m \colon  X \to \R$ is called a \emph{$k$-interpolating family} in $X$, if for every collection of distinct points $x_1, \ldots, x_k \in X$ and every $y = (y_1, \ldots, y_k) \in \R^k$ there exists $(\alpha_1, \ldots, \alpha_m) \in \R^m$ such that 
$\alpha_1 h_1(x_i) + \cdots + \alpha_m h_m(x_i) = y_i$
for each $i=1, \ldots, k$. In other words, the matrix
\[ \begin{bmatrix} h_1(x_1) & \ldots & h_m(x_1) \\
\vdots &\ddots & \vdots \\
h_1(x_k) & \ldots & h_m(x_k)
\end{bmatrix} \]
has the maximal rank. Note that the same is true for any collection of $\ell$ distinct points with $\ell \leq k$.
\end{defn}
\begin{rem}\label{rem:poly_interp} It is known that any linear basis $\{h_1, \ldots, h_m\}$ of the space of real polynomials of $N$ variables of degree at most $k-1$ is a $k$-interpolating family in $\R^N$ (see e.g.~\cite[Section~1.2, eq.~(1.9)]{poly-interpolation}). Obviously, the functions from this family are locally Lipschitz on any set $X \subset \R^N$, in particular they belong to $H^\beta_{loc}(X)$ for $\beta \in (0, 1]$.
\end{rem} 

\subsection*{Dimensions and measures}

\begin{defn}\label{defn:dim}

For $s>0$, the \emph{$s$-dimensional $($outer$)$ Hausdorff measure} of a set $X \subset \R^N$ is defined as
\[ \mH^s(X) = \lim \limits_{\delta \to 0}\ \inf \Big\{ \sum \limits_{i = 1}^{\infty} |U_i|^s : X \subset \bigcup \limits_{i=1}^{\infty} U_i,\ |U_i| \leq \delta  \Big\}.\]
The \emph{Hausdorff dimension} of $X$ is given as
\[ \hdim X = \inf \{ s > 0 : \mathcal{H}^s(X) = 0 \} = \sup \{ s > 0 : \mathcal{H}^s(X) = \infty \}. \]
For a bounded set $X \subset \R^N$ and $\delta>0$, let $N(X, \delta)$ denote the minimal number of balls of diameter at most $\delta$ required to cover $X$. The \emph{lower} and \emph{upper box-counting $($Minkowski$)$ dimension} of $X$ are defined, respectively, as
\[ \lbdim X = \liminf \limits_{\delta \to 0} \frac{\log N(X,\delta)}{-\log \delta,}\qquad \text{ and }\quad \udim X = \limsup \limits_{\delta \to 0} \frac{\log N(X,\delta)}{-\log \delta}. \]
The lower (resp.~upper) box-counting dimension of an unbounded set is defined as the supremum of the lower (resp.~upper) box-counting dimensions of its bounded subsets. By the separability of $\R^N$, one may always assume that the supremum is over a countable number of bounded subsets.

Using this together with the fact $\hdim \big(\bigcup_{i=1}^\infty X_i\big)=\sup_{i \ge 1} \hdim X_i$, one shows 
\begin{equation}\label{eq:hdim_mbdim_bdim}
\hdim X \leq \lbdim X \leq \udim X
\end{equation}
for every $X \subset \R^N$. 

Recall that for a measure $\mu$ on a set $X$ and a measurable set $Y \subset X$ we say that $Y$ is of \emph{full $\mu$-measure}, if $\mu(X\setminus Y) = 0$. We define dimension of a finite Borel measure $\mu$ in $\R^N$ as
\[
\dim \mu = \inf\{ \dim X: X \subset \R^N \text{ is a Borel set of full $\mu$-measure} \}.
\]
Here $\dim$ may denote any type of the dimensions defined above. 

For a Borel probability measure $\mu$ in $\R^N$ with a compact support define its \emph{lower} and \emph{upper information dimension} as
\[ \lid(\mu) = \liminf \limits_{\eps \to 0} \int \limits_{\supp \mu} \frac{\log \mu(B(x,\eps))}{\log \eps} d\mu(x), \qquad\uid(\mu) = \limsup \limits_{\eps \to 0} \int \limits_{\supp \mu} \frac{\log \mu(B(x,\eps))}{\log \eps} d\mu(x).\]
If $\lid(\mu) = \uid(\mu)$, then we denote their common value by $\idim (\mu)$ and call the \emph{information dimension} of $\mu$. 
For more information on dimension theory in Euclidean spaces see \cite{falconer2014fractal, mattila, Rob11}. 
\end{defn}

\begin{defn}\label{defn:support}
The \emph{support} of a Borel measure $\mu$ in a metric space, denoted $\supp \mu$, is the smallest closed set of full $\mu$-measure. Note that the measure of any open ball centred at a point $x \in \supp \mu$ is positive (or infinite). For a Borel map $\phi$, by $\phi_*\mu$ we denote the \emph{push-forward} of $\mu$ under $\phi$, defined by $\phi_*\mu(E) = \mu(\phi^{-1}(E))$ for Borel sets $E$. We say that measures $\mu$ and $\nu$ are \emph{mutually singular} if there exists a measurable set $X$ of full $\mu$-measure such that $\nu(X) = 0$. We denote this fact by $\mu \perp \nu$.
\end{defn}

\begin{defn}\label{defn:ergodic}
Let $\mu$ be a Borel measure on a Borel set $X$ and let $T\colon X \to X$ be a Borel map. The measure $\mu$ is $T$-\emph{invariant}, if $\mu(T^{-1}(E)) = \mu(E)$ for every Borel set $E \subset X$. The measure $\mu$ is \emph{ergodic}, if for every Borel set $E \subset X$, the condition $T^{-1}(E) = E$ implies $\mu(E) = 0$ or $\mu(X \setminus E) = 0$.
\end{defn}

\begin{defn}[{\bf Natural measure}{}]\label{defn:nat_measure}
Let $M$ be a compact Riemannian manifold and $T \colon M \to M$ a smooth diffeomorphism. A compact $T$-invariant set $X \subset M$ is called an \emph{attractor}, if there exists an open set $B \subset M$ containing $X$, such that $\lim_{n \to \infty} \dist(T^n x, X) = 0$ for every $x \in B$. The largest such  set $B(X)$ is called the (\emph{maximal}) \emph{basin of attraction} of $X$. A $T$-invariant Borel probability measure $\mu$ on $X$ is called a \emph{natural measure} if 
	\[ \lim \limits_{n \to \infty} \frac{1}{n} \sum \limits_{i=0}^{n-1} \delta_{T^i x} = \mu \]
for almost every $x \in B(X)$ with respect to the volume measure on $M$, where $\delta_y$ denotes the Dirac measure at $y$ and the limit is taken in the weak-$^*$ topology.
\end{defn}

\section{Combinatorics of orbits}\label{sec:combin}

Let $X \subset \R^N$ be an arbitrary set and let $T\colon X \to X$ be a transformation. Fix $k \in \N$ and let $h_1, \ldots, h_m \colon X \to \R$ be a $2k$-interpolating family in $X$, according to Definition~\ref{defn:interpol}. Consider an observable $h \colon X \to \R$. For $\alpha = (\alpha_1, \ldots, \alpha_m) \in \R^m$ let $h_\alpha \colon X \to \R$ be given by 
\[
h_\alpha = h + \sum \limits_{j=1}^m \alpha_j h_j.
\]
Let
\[
\phi^T_{\alpha,k} \colon X \to \R^k, \qquad \phi^T_{\alpha,k}(x) = (h_\alpha(x), h_\alpha(Tx), \ldots, h_\alpha(T^{k-1}x)) \]
be the $k$-delay coordinate map corresponding to $h_\alpha$ (here it is useful to include the dependence on $T$ and $k$ in the notation). In this section we study the combinatorics of the orbits of points of $X$ in relation to the properties of $\phi^T_{\alpha,k}$.

Note that for $x,y \in X$ we have
\begin{equation}\label{e:matrix_form} \phi^T_{\alpha,k}(x) - \phi^T_{\alpha,k}(y) = D_{x,y}\alpha + w_{x,y}
\end{equation}
for a $k\times m$ matrix $D_{x,y} = D_{x,y}(T, k, h)$ defined by
\begin{equation}\label{eq:D_xy}
D_{x,y} = \begin{bmatrix} h_1(x) - h_1(y) & \ldots & h_m(x) - h_m(y) \\
h_1(Tx) - h_1(Ty) & \ldots & h_m(Tx) - h_m(Ty) \\
\vdots & \ddots & \vdots \\
h_1(T^{k-1}x) - h_1(T^{k-1}y) & \ldots & h_m(T^{k-1}x) - h_m(T^{k-1}y) \\
\end{bmatrix} 
\end{equation}
and
\[
w_{x,y} = \begin{bmatrix} h(x) - h(y) \\
h(Tx) - h(Ty)\\
\vdots \\
h(T^{k-1} x) - h(T^{k-1}y) \end{bmatrix}.
\]
The aim of this section is to prove the following proposition, which provides a basic technical tool used in the proofs of the main results of this paper. Note that as the proof of the proposition is purely combinatorial, we do not assume any regularity on the dynamics $T$ and observables $h, h_1, \ldots, h_m$, except of the condition that $\{h_1, \ldots, h_m\}$ is a $2k$-interpolating family in $X$.

\begin{prop}
\label{prop: rank or predict} For every $x,y \in X$, if $\rank D_{x,y} < k$, then
\[
\|\phi^T_{\alpha,k}(Tx) - \phi^T_{\alpha,k}(Ty)\| \leq 2k \| \phi^T_{\alpha,k}(x) - \phi^T_{\alpha,k}(y) \| \quad \text{for every }\alpha \in \R^m.
\]
\end{prop}
This proposition together with Lemma~\ref{lem:rank=singular} implies immediately the following corollary, which is a key result used in this paper.

\begin{cor}\label{cor:crucial}
If $x,y \in X$ and $\alpha \in \R^m$ satisfy
\[
\|\phi^T_{\alpha,k}(Tx) - \phi^T_{\alpha,k}(Ty)\| > 2k \| \phi^T_{\alpha,k}(x) - \phi^T_{\alpha,k}(y) \|,
\]
then $\sigma_k(D_{x,y}) > 0$. In particular, 
\[
\text{if} \quad \phi^T_{\alpha,k}(x) = \phi^T_{\alpha,k}(y) \quad \text{and} \quad \phi^T_{\alpha,k}(Tx) \neq \phi^T_{\alpha,k}(Ty), \quad \text{then} \quad \sigma_k(D_{x,y}) > 0.
\]
\end{cor}

The rest of this section is devoted to the proof of Proposition~\ref{prop: rank or predict}.

Fix a pair of points $x,y \in X$ and assume 
\begin{equation}\label{eq:rankD<}
\rank D_{x,y} < k.
\end{equation}
Note that we can assume 
\begin{equation}\label{eq:Tx neq Ty}
T^i x \neq T^i y \quad \text{for } i = 0, \ldots, k-1,
\end{equation}
as otherwise $T^k x = T^k y$ and the lemma holds trivially. Denote the elements of the set $\{x, Tx, \ldots, T^{k-1}x, y, Ty, \ldots, T^{k-1}y \}$ by $z_0, \ldots, z_{\ell-1}$ (without multiplicities), preserving the order within the set, so that
\[\ell = \#\{x, Tx, \ldots, T^{k-1}x, y, Ty, \ldots, T^{k-1}y \}.\]
Under this notation, the matrix $D_{x,y}$ can be written as a product of matrices
\[ D_{x,y} = J_{x,y} V_{x,y}, \]
where 
\[ V_{x,y}
= \begin{bmatrix} h_1(z_0) & \ldots & h_m(z_0) \\
\vdots & \ddots & \vdots \\
h_1(z_{\ell-1}) & \ldots & h_m(z_{\ell-1})
\end{bmatrix} \]
and $J_{x,y} = J_{x,y}(T,k) = [m_{i,j}]_{\substack{0 \leq i \leq k-1\\0 \leq j \leq \ell-1}}$ is a $k \times \ell$ matrix, where
\[m_{i,j} = \begin{cases} 1 &\text{if}\ z_j = T^i x\\
-1 &\text{if}\ z_j = T^i y\\
0 &\text{otherwise}
\end{cases}\]
(note that we enumerate rows and columns starting from $0$). By \eqref{eq:Tx neq Ty}, the entries $m_{i,j}$ are well-defined. Moreover, every row of $J_{x,y}$ consists of a single entry $1$, single entry $-1$ and $\ell-2$ zeros. Since $z_0, \ldots, z_{\ell-1}$ are distinct, $\ell \leq 2k$ and $\{h_1, \ldots, h_m\}$ is a $2k$-interpolating family, the matrix $V_{x,y}$ has the maximal rank. Therefore, $\rank D_{x,y} = \rank J_{x,y}$, so by \eqref{eq:rankD<},
\begin{equation}\label{eq:rankJ<}
\rank J_{x,y} < k.
\end{equation}

Let $\vec{G}$ be a directed graph with the set of vertices 
\[ V = \{ x, Tx, \ldots, T^{k-1}x, y, Ty, \ldots, T^{k-1} y\} = \{z_0, \ldots, z_{\ell-1} \} \]
and directed edges $T^i x \longrightarrow T^i y$, $i=0, \ldots, k-1$. Let $G$ denote the non-directed version of $\vec{G}$, i.e.~a non-directed graph with the same set $V$ of vertices and non-directed edges $T^i x \longleftrightarrow T^i y$, $i=0, \ldots, k-1$. Note that the points $T^k x$, $T^k y$ may be contained in $V$, and the graph may have vertices of degree higher than $1$. The key observation is the following.

\begin{lem}\label{lem:claim}
If $T^k x, T^k y \in V$ and there exists a path in $G$ joining $T^k x$ with $T^k y$, then \[
\| \phi^T_{\alpha,k}(Tx) - \phi^T_{\alpha,k}(Ty)\| \leq k \| \phi^T_{\alpha,k}(x) - \phi^T_{\alpha,k}(y) \|
\]
for every $\alpha \in \R^m$.
\end{lem}
\begin{proof}
Note that if $T^k x = u_0 \longleftrightarrow u_1 \longleftrightarrow \dots \longleftrightarrow u_L = T^k y$ is a path in $G$, then for any $\alpha \in \R^m$ we have 
\[ |h_{\alpha}(u_i) - h_\alpha(u_{i+1})| \leq \| \phi^T_{\alpha,k}(x) - \phi^T_{\alpha,k}(y)\|\quad \text{for } i= 0, \ldots, L-1 \]
as for each $i$ there holds $\{ u_i, u_{i+1} \} = \{ T^j x, T^j y \}$ for some $j \in \{0, \ldots, k-1\}$. Hence, 
\[ |h_\alpha(T^k x) - h_\alpha(T^k y)| \leq \sum \limits_{i=0}^{L-1} |h_\alpha(u_i) - h_\alpha(u_{i+1})| \leq L\| \phi^T_{\alpha,k}(x) - \phi^T_{\alpha,k}(y)\|, \]
so
\begin{align*}
\|\phi^T_{\alpha,k}(Tx) - \phi^T_{\alpha,k}(Ty)\| &\leq \max \big\{\|\phi^T_{\alpha,k}(x) - \phi^T_{\alpha,k}(y)\|, |h_\alpha(T^k x) - h_\alpha(T^k y)|\big\}\\
&\leq L\| \phi^T_{\alpha,k}(x) - \phi^T_{\alpha,k}(y)\|.
\end{align*}
This ends the proof, as the total number of edges in $G$ is at most $k$, hence any two vertices connected by a path are in fact connected by a path of length at most $k$.
\end{proof}

We proceed with the proof of Proposition~\ref{prop: rank or predict}. Recall that we have fixed $x,y \in X$ satisfying \eqref{eq:rankD<} and \eqref{eq:Tx neq Ty}. Let us introduce some notation. For $z \in \{x, y\}$, let $\Orb_T(z) = \{ T^n z : n \geq 0 \}$ be its orbit under $T$. We will often write $\Orb(z)$ instead of $\Orb_T(z)$ if the transformation is clear from the context. We call a point $z$ \emph{periodic} if there exists $p \geq 1$ such that $T^p z = z$, and \emph{pre-periodic} (or \emph{eventually periodic}) if it is not periodic and there exists $n \geq 1$ such that $T^n z$ is periodic. If $z$ is neither periodic nor pre-periodic, it is called \emph{aperiodic}. Note that $z$ is aperiodic if and only if $\Orb(z)$ is infinite. Every orbit $\Orb(z)$ can be uniquely presented as a disjoint union of a pre-periodic part and a cycle (we adopt a convention that for aperiodic $z$, the pre-periodic part is infinite and the cycle is empty). We will denote by $P(z)$ the length of the pre-periodic part and by $C(z)$ the length of the cyclic part. 

We will consider four (not mutually disjoint) cases, which together cover all the possibilities, where at least one of the points $x,y$ is periodic or aperiodic (note that only Case~4 will actually require the assumption of periodicity). Then we will show how to reduce the remaining case (with both $x,y$ being pre-periodic) to the previously established ones.

\subsection*{Case 1.} $\#\Orb(x) \geq k$ or $\#\Orb(y) \geq k$.

\medskip

By symmetry, we can assume $\#\Orb(x) \geq k$. Then $\ell \geq k$ and the matrix $J_{x,y}$ has entries $1$ on the main diagonal.\footnote{By the main diagonal of a matrix we understand the elements of the matrix with coordinates of the form $(i,i)$.} Therefore, the first (from the left) $k \times k$ square submatrix of $J_{x,y}$ has the form $I- A$, where $I$ is the $k \times k$ identity matrix and $A$ is a $k \times k$ matrix $[a_{i,j}]_{\substack{0 \leq i \leq k-1\\0 \leq j \leq k-1}}$ defined by
\[ a_{i,j} = \begin{cases} 1 &\text{if}\ z_j = T^i y\\
0 &\text{otherwise}
\end{cases}.\]
As $z_i = T^i x$ for $0 \leq i \leq k-1$, we see that $A$ is the adjacency matrix of a subgraph of $\vec{G}$, obtained by restricting the set of vertices to $\{x, Tx, \ldots, T^{k-1}x \}$ and including only those edges $T^i x \to T^i y$, $0 \leq i \leq k-1$, for which $T^i y \in \{x, Tx, \ldots, T^{k-1}x \}$. If $A$ is nilpotent, then $I - A$ is invertible, hence $\rank J_{x,y} = k$, which contradicts \eqref{eq:rankJ<}. Therefore, we can assume that $A$ is not nilpotent, so the graph described by $A$ contains a cycle. Consequently, $\vec{G}$ contains a cycle. We will prove that this implies $T^k x, T^k y \in V$ and the existence of a (directed) path from $T^k y$ to $T^k x$ in $\vec{G}$. Then there is a path from $T^k x$ to $T^k y$ in $G$ and the assertion of Proposition~\ref{prop: rank or predict} holds by Lemma~\ref{lem:claim} (with constant $k$ instead of $2k$).

To show that there is a path from $T^k y$ to $T^k x$ in $\vec{G}$, note first that any cycle in $\vec{G}$ is of the form
\begin{equation}\label{eq:G general cycle} T^{i_1} x \longrightarrow T^{i_1} y = T^{i_2} x \longrightarrow T^{i_2} y = T^{i_3} x \longrightarrow \cdots \longrightarrow T^{i_s} y = T^{i_1} x
\end{equation}
for some $i_1, \ldots, i_s \in \{0, \ldots, k-1\}$. If $i_j < k - 1$ for each $j = 1, \ldots, s$, then the following path (the image of the path \eqref{eq:G general cycle} under $T$)
\[ T^{i_1 + 1} x \longrightarrow T^{i_1 + 1} y = T^{i_2 + 1} x \longrightarrow T^{i_2 + 1} y = T^{i_3 + 1} x \longrightarrow \cdots \longrightarrow T^{i_s + 1} y = T^{i_1 + 1} x \]
also describes a cycle in $\vec{G}$. Therefore, there exists a cycle in $\vec{G}$ containing an edge $T^{k-1} x \longrightarrow T^{k-1} y$, and hence there is a path in $\vec{G}$ of the form
\[ T^{k-1}y = T^{i_1} x \longrightarrow T^{i_1} y = T^{i_2} x \longrightarrow \cdots \longrightarrow T^{i_s} y = T^{k-1} x,  \]
where $i_1, \ldots, i_s \in \{0, \ldots, k-1\}$. In fact, $i_1, i_s < k-1$ since $T^{k-1}x \neq T^{k-1} y$ and all others indices can be assumed to be smaller than $k-1$, as otherwise we can consider a shorter path. Therefore, $T^k x, T^k y \in V$ and
\[ T^k y = T^{i_1 + 1} x \longrightarrow T^{i_1 + 1} y = T^{i_2 + 1} x \longrightarrow \cdots \longrightarrow T^{i_s + 1} y  = T^k x \]
is a path in $\vec{G}$. This ends the proof of Proposition~\ref{prop: rank or predict} in Case~1.

\medskip
In the remaining cases we assume that $x,y$ have bounded orbits. Then each of the points $x,y$ is periodic or pre-periodic, and hence the numbers $C(x), P(x), C(y), P(y)$ are finite with $C(x), C(y) > 0$. For simplicity, denote
\[
p = C(x), \quad q = C(y)
\]
and notice that $\#\Orb(x) = P(x) + p$ and $\#\Orb(y) = P(y) + q$.

\subsection*{Case 2.} $\#\Orb(x) \leq k$, $\#\Orb(y) \leq k$, $C(y)|C(x)$ and $P(y) + C(x) \leq k$.

\medskip
By assumption, we have $p =  rq $ for some $r \in \N$. As $k - p \geq P(x)$ and $k - p \geq P(y)$, we see that $T^{k-p}x$ belongs to the cyclic part of $\Orb(x)$ and $T^{k-p}y$ belongs to the cyclic part of $\Orb(y)$. Therefore, $T^{k-p}x = T^k x$ and $T^{k-p} y = T^{k-rq} y = T^{k} y$. Consequently, $T^k x, T^k y \in V$ and $G$ contains the path
\[ T^{k} x = T^{k - p} x \longleftrightarrow T^{k - p} y = T^k y. \]
Hence, we can use Lemma~\ref{lem:claim} to show Proposition~\ref{prop: rank or predict}, with the constant $k$ instead of $2k$. This ends the proof in Case~2.

\medskip
Note that if $\Orb(x) \cap \Orb(y) \neq \emptyset$ and both orbits are finite, then the cyclic parts of $\Orb(x)$ and $\Orb(y)$ have to be equal, hence $C(x) = C(y)$. Therefore, Cases~1--2 cover all the possibilities with $\Orb(x) \cap \Orb(y) \neq \emptyset$ (since then $\#\Orb(y) \leq k$ implies $P(y) + C(x) \leq k$).

\subsection*{Case 3.} $\#\Orb(x) + \#\Orb(y) \leq k$.

\medskip
By assumption,
\begin{equation*}\label{eq:sum of orbits bound}
k \geq p + q + P(x) + P(y).
\end{equation*}
This implies the following equalities.
\begin{itemize}
	\item $T^{k-p}x = T^k x$ (as $k-p \geq P(x)$, and hence $T^{k-p} x$ belongs to the cyclic part of $\Orb(x)$ and $p$ is its period),
	\item $T^{k - q - p} y = T^{k - p} y$ (as $k - q - p \geq P(y)$),
	\item $T^{k - q - p} x = T^{k - q} x$ (as $k - q -  p \geq P(x)$),
	\item $T^{k - q} y  = T^k y$ (as $k-q \geq P(y)$).
\end{itemize}
Using these equalities, we see that $T^k x, T^k y \in V$ and 
\[ T^k x = T^{k-p} x \longleftrightarrow T^{k-p} y = T^{k-p-q} y \longleftrightarrow T^{k - p - q} x = T^{k - q} x \longleftrightarrow T^{k-q} y = T^k y. \]
is a path in $G$. Again, Proposition~\ref{prop: rank or predict} holds by Lemma~\ref{lem:claim}, with the constant $k$ instead of $2k$. This ends the proof in Case~3.

\subsection*{Case 4.} $x$ is periodic, $\#\Orb(x) + \#\Orb(y) \geq k$, $\#\Orb(x) \leq k,$ $\#\Orb(y) \leq k$, $\Orb(x) \cap \Orb(y) = \emptyset$ and $\big( C(y) \nmid C(x) \text{ or } P(y) + C(x) \geq k \big)$.

\smallskip
As $x$ is periodic, the matrix $J_{x,y}$ has the following form.
\[ 
J_{x,y} = \;\renewcommand{\arraystretch}{1.75}
\begin{NiceArray}{r}
\\
\\
\\
\\
\\
\\
\\
\\
\\[1mm]
\Block{2-1}{k\ \mathrm{mod}\ p}\\
\\[3mm]
\\
\end{NiceArray}
\hspace{5.5mm} 
\begin{bNiceArray}{ccccIccIccccc}
\;\; 1 \;\; &  \fatgreyzero &  \fatgreyzero &  \fatgreyzero & -1 &  \fatgreyzero &  \fatgreyzero &  \fatgreyzero &  \fatgreyzero &  \fatgreyzero &  \fatgreyzero\\
 \fatgreyzero & \;\; 1 \;\; &  \fatgreyzero & \fatgreyzero & \fatgreyzero & -1 & \fatgreyzero &\fatgreyzero  & \fatgreyzero & \fatgreyzero & \fatgreyzero\\
\fatgreyzero & \fatgreyzero & \;\; 1 \;\; & \fatgreyzero & \Block[borders={top,tikz=dashed}]{1 - 2}{}\fatgreyzero &\fatgreyzero &\Block[borders={top,bottom,tikz=dashed}]{5 - 5}{} -1 & \fatgreyzero & \fatgreyzero & \fatgreyzero & \fatgreyzero\\
\fatgreyzero & \fatgreyzero & \fatgreyzero & \;\; 1 \;\; & \fatgreyzero & \fatgreyzero & \fatgreyzero & -1 & \fatgreyzero & \fatgreyzero & \fatgreyzero\\
\Block[borders={top, bottom,tikz=dashed}]{4 - 4}{}\;\; 1 \;\; & \fatgreyzero & \fatgreyzero & \fatgreyzero & \fatgreyzero & \fatgreyzero & \fatgreyzero & \fatgreyzero & -1 & \fatgreyzero & \fatgreyzero\\
\fatgreyzero & \;\; 1 \;\; & \fatgreyzero & \fatgreyzero & \fatgreyzero & \fatgreyzero & \fatgreyzero & \fatgreyzero & \fatgreyzero & -1 & \fatgreyzero\\
\fatgreyzero & \fatgreyzero & \;\; 1 \;\; &\fatgreyzero  & \fatgreyzero & \fatgreyzero & \fatgreyzero & \fatgreyzero & \fatgreyzero &\fatgreyzero  & -1\\
\fatgreyzero & \fatgreyzero & \fatgreyzero & \;\; 1 \;\; &\fatgreyzero  &\fatgreyzero  & -1 & \fatgreyzero & \fatgreyzero & \fatgreyzero & \fatgreyzero \\
\;\; 1 \;\; & \fatgreyzero & \fatgreyzero & \fatgreyzero & \fatgreyzero & \fatgreyzero & \fatgreyzero & -1 & \fatgreyzero &\fatgreyzero  &\fatgreyzero \\
\fatgreyzero & \;\; 1 \;\; & \fatgreyzero & \fatgreyzero & \fatgreyzero & \fatgreyzero & \fatgreyzero & \fatgreyzero & -1 & \fatgreyzero & \fatgreyzero\\[1mm]
\CodeAfter
\OverBrace{1-1}{1-4}{p}[yshift=2mm]
\OverBrace{1-5}{1-6}{P(y)}[yshift=2mm]
\OverBrace{1-7}{1-11}{q}[yshift=2mm]
\UnderBrace{10-1}{10-4}{\#\Orb(x)}[yshift=2mm]
\UnderBrace{10-5}{10-11}{\#\Orb(y)}[yshift=2mm]
\SubMatrix{\{}{9-1}{10-1}{.}[xshift=2mm, extra-height=4.2mm,code = {\tikz \draw [dashed] (-7.7,-2.25) -- (-7.7,-3.95) ;}]
\SubMatrix{.}{8-11}{10-11}{\rbrace}[xshift=2mm,extra-height=4.5mm, code = {\tikz \draw [dashed] (-1.75,-1.5) -- (-1.75,-3.95) ;}]
\SubMatrix{.}{1-11}{2-11}{\rbrace}[xshift=2mm,extra-height=3mm]
\SubMatrix{.}{3-11}{7-11}{\rbrace}[xshift=2mm,extra-height=4.3mm]
\end{bNiceArray}
\hspace{3.5mm}
\begin{NiceArray}{l}
\Block{2-1}{P(y)}\\
\\
\Block{5-1}{q}\\
\\
\\
\\
\\
\Block{3-1}{(k - P(y))\ \mathrm{mod}\ q}\\
\\
\\
\end{NiceArray}
\]
(we present here a particular case $k=10$, $p=4$, $P(y) =2$, $q=5$ for illustration, but the structure of the matrix is general).

We will now perform a sequence of elementary row operations on the matrix  $J_{x,y}$ to `move' all its entries $1$ to the main diagonal. Namely, we perform the following sequence of $k - p$ operations on $J_{x,y}$, to obtain a new matrix $K_{x,y}$ (recall that we enumerate rows and columns starting from $0$).
\begin{itemize}
	\item[($1$)] Subtracting the $(k - 1 - p)$-th row from the $(k-1)$-th row (the last row of the matrix)
	
	\smallskip
	As $T^{k-1}x = T^{k-1 - p}x$, this operation deletes the entry $1$ from the $(k-1)$-th row. Since $k-1-p < \#\Orb(y)$ (as $k \leq \#\Orb(x) + \#\Orb(y)$), this also adds the entry $1$ to the element in the $(k - 1 - p + p = k - 1)$-th column and $(k-1)$-th row. We claim that this entry does not cancel with the entry $-1$ in the $(k-1)$-th row. For that, we have to check that the unique entry $-1$ in the $(k-1)$-th row is not located in the $(k-1)$-th column. There are two possibilities, corresponding to the assumptions $P(y) + C(x) \geq k$ and $C(y) \nmid C(x)$, respectively. If $P(y) + C(x) \geq k$, then $k - p - 1 < P(y)$, hence the $(k-1)$-th column 'corresponds' to the preperiodic part of the $\Orb(y)$ and thus contains a single entry $-1$ (in the $(k-p-1)$-th row), which is therefore not deleted. Alternatively, we have $C(y) \nmid C(x)$. Then the cancellation of the entry $-1$ from the $(k-1)$-th row by subtracting the $(k-1-p)$-th row would mean that $T^{k-1} y = T^{k - 1 - p} y$. This cannot happen as $C(y) \nmid C(x)$. Therefore, the final effect of this operation is moving the entry $1$ in the last row from the $(k\ \mathrm{mod}\ p)$-th column to the $(k - 1)$-th column in the same row, i.e.~to the main diagonal (without cancelling any entries $-1$).
	
	\smallskip
	\item[($2$)] Subtracting the $(k - 2 - p)$-th row the $(k-2)$-th row
	
	\smallskip
	Similarly, this moves the entry $1$ in the $(k - 2)$-th row to the main diagonal.
	\item[$\vdots$\;\;]
	\item[]
		
	\item[($k-p$)] Subtracting the $0$-th row from the $p$-th row.
	
	\smallskip
	This moves the entry $1$ in the $p$-th row to the diagonal of the first (from the left) $k \times k$ square submatrix.
\end{itemize}

\smallskip
Note that in the above sequence of operations, once the row is modified, it is never being used later to modify another row, hence indeed the only results of the operations are the ones described above. Consequently, in the new matrix $K_{x,y}$, all the rows between the $p$-th and $(k-1)$-th ones have their unique entry $1$ on the main diagonal, while all the entries $-1$ remain at the same location as in $J_{x,y}$. As $J_{x,y}$ has entries $1$ on the main diagonal in the rows between $0$ and $p-1$, and these rows were not modified by the above operations, we conclude that $K_{x,y}$ has the following form.
\[
K_{x,y} = \;\renewcommand{\arraystretch}{1.75}
\begin{NiceArray}{r}
\\
\\
\\
\\
\\[0.5mm]
\Block{6-1}{k - p}\\
\\[2.5mm]
\\
\\
\\
\\
\\
\end{NiceArray}
\hspace{5.5mm} 
\begin{bNiceArray}{ccccIccIccccc}
\;\; \boldsymbol{1} \;\; &  \fatgreyzero &  \fatgreyzero &  \fatgreyzero & -1 &  \fatgreyzero &  \fatgreyzero  &  \fatgreyzero & \fatgreyzero & \fatgreyzero& \fatgreyzero\\
\fatgreyzero & \;\; \boldsymbol{1} \;\; & \fatgreyzero & \fatgreyzero & \fatgreyzero & -1 & \fatgreyzero & \fatgreyzero & \fatgreyzero & \fatgreyzero &\fatgreyzero \\
\fatgreyzero & \fatgreyzero & \;\; \boldsymbol{1} \;\; & \fatgreyzero &\Block[borders={top,tikz=dashed}]{1 - 2}{}  \fatgreyzero & \fatgreyzero &\Block[borders={top,bottom,tikz=dashed}]{5 - 5}{} -1 & \fatgreyzero & \fatgreyzero & \fatgreyzero &\fatgreyzero \\
\fatgreyzero & \fatgreyzero & \fatgreyzero & \;\; \boldsymbol{1} \;\; & \fatgreyzero & \fatgreyzero & \fatgreyzero & -1 & \fatgreyzero & \fatgreyzero & \fatgreyzero\\
\Block[borders={top, tikz=dashed}]{4 - 4}{} \fatgreyzero& \fatgreyzero & \fatgreyzero & \fatgreyzero & \;\; \boldsymbol{1} \;\; & \fatgreyzero & \fatgreyzero & \fatgreyzero & -1 & \fatgreyzero & \fatgreyzero\\
\fatgreyzero & \fatgreyzero & \fatgreyzero & \fatgreyzero & \fatgreyzero & \;\; \boldsymbol{1} \;\; & \fatgreyzero & \fatgreyzero & \fatgreyzero & -1 & \fatgreyzero\\
\fatgreyzero & \fatgreyzero & \fatgreyzero & \fatgreyzero & \fatgreyzero & \fatgreyzero & \;\; \boldsymbol{1} \;\; & \fatgreyzero & \fatgreyzero & \fatgreyzero & -1\\
\fatgreyzero & \fatgreyzero & \fatgreyzero & \fatgreyzero & \fatgreyzero & \fatgreyzero & -1 & \;\; \boldsymbol{1} \;\; & \fatgreyzero & \fatgreyzero & \fatgreyzero \\
\fatgreyzero & \fatgreyzero & \fatgreyzero & \fatgreyzero & \fatgreyzero & \fatgreyzero & \fatgreyzero & -1 & \;\; \boldsymbol{1} \;\; & \fatgreyzero & \fatgreyzero\\
\fatgreyzero & \fatgreyzero & \fatgreyzero & \fatgreyzero & \fatgreyzero & \fatgreyzero & \fatgreyzero & \fatgreyzero & -1 & \;\; \boldsymbol{1} \;\; & \fatgreyzero\\[1mm]
\CodeAfter
\OverBrace{1-1}{1-4}{p}[yshift=2mm]
\OverBrace{1-5}{1-6}{P(y)}[yshift=2mm]
\OverBrace{1-7}{1-11}{q}[yshift=2mm]
\UnderBrace{10-1}{10-10}{k}[yshift=9.5mm]
\UnderBrace{10-1}{10-4}{\#\Orb(x)}[yshift=2mm]
\UnderBrace{10-5}{10-11}{\#\Orb(y)}[yshift=2mm]
\SubMatrix{\{}{5-1}{10-1}{.}[xshift=2mm, extra-height=4mm]
\SubMatrix{.}{8-11}{10-11}{\rbrace}[xshift=2mm,extra-height=4.5mm, code = {\tikz \draw [dashed] (-1.8,-1.5) -- (-1.8,-3.85) ;}]
\SubMatrix{.}{1-11}{2-11}{\rbrace}[xshift=2mm,extra-height=3mm]
\SubMatrix{.}{3-11}{7-11}{\rbrace}[xshift=2mm,extra-height=4.3mm]
\end{bNiceArray}
\hspace{3.5mm}
\begin{NiceArray}{l}
\Block{2-1}{P(y)}\\
\\
\Block{5-1}{q}\\
\\
\\
\\
\\
\Block{3-1}{(k - P(y))\ \mathrm{mod}\ q}\\
\\
\\
\end{NiceArray}\vspace*{6mm}
\]
The key observation is that the matrix $K_{x,y}$ is equal to $J_{x,y}(\tilde T,k)$ for a suitably changed dynamics $\tilde T$ on the set $\Orb(x) \cup \Orb(y)$. Indeed, define 
\[\tilde T \colon \Orb(x) \cup \Orb(y) \to \Orb(x) \cup \Orb(y)\] 
by
\[ \tilde T|_{\Orb(y)} = T|_{\Orb(y)},\qquad \tilde T(T^j x) = \begin{cases} T^{j+1}x & \text{for } 0 \leq j < p-1 \\
y &\text{for } j = p-1\end{cases}. \]
In order words, $\tilde T$ follows the dynamics of $T$ on $\Orb(y)$ and changes the periodic orbit of $x$ into a pre-periodic part, whose final element is mapped to $y$ instead of returning to $x$. Note that $\#\Orb_{\tilde T}(x) = \Orb_T(x) + \Orb_T(y) \geq k$, hence $J_{x,y}(\tilde T, k)$ has entries $1$ exactly on the main diagonal. The same is true for $K_{x,y}$. As additionally $\tilde T|_{\Orb(y)} = T|_{\Orb(y)}$ and the matrices $J_{x,y}(T)$ and $K_{x,y}$ have entries $-1$ at the same positions, we see that indeed 
\[
K_{x,y} = J_{x,y}(\tilde T,k).
\]
Note that $\rank{J_{x,y}(\tilde T, k)} < k$, because otherwise $\rank J_{x,y}(T,k)  = \rank K_{x,y} = \rank{J_{x,y}(\tilde T, k}) = k$, which is impossible by \eqref{eq:rankJ<}. Since $\Orb_{\tilde T}(x) \geq k$, (the previously proved) Case~1 applies to the points $x,y$ undergoing the dynamics of $\tilde T$. Hence, we can use Proposition~\ref{prop: rank or predict} (for the dynamics $\tilde T$) to obtain
\[\| \phi^{\tilde T}_{\alpha,k}(\tilde T x) - \phi^{\tilde T}_{\alpha,k}(\tilde T y)\| \leq k \| \phi^{\tilde T}_{\alpha,k}(x) - \phi^{\tilde T}_{\alpha,k}(y)\| \quad \text{for every } \alpha \in \R^k.\]
Therefore,
\[
\begin{aligned}
|h_{\alpha}(T^{k - p} y ) - h_{\alpha}(T^{k} y )|  & = |h_{\alpha}({\tilde T}^{k} x ) - h_{\alpha}({\tilde T}^{k} y )| \leq k \| \phi^{\tilde T}_{\alpha,k}(x) - \phi^{\tilde T}_{\alpha,k}(y)\|\\
&\leq k \max \{|h_{\alpha}(x) - h_{\alpha}(y)|, \ldots, |h_{\alpha}(T^{p-1}x) - h_{\alpha}(T^{p-1}y)|,\\
&\phantom{\leq k \max \{ \;} |h_{\alpha}(y) - h_{\alpha}(T^p y)|, \ldots, |h_{\alpha}(T^{k-p-1}y) - h_{\alpha}(T^{k-1} y)|  \}.
\end{aligned}  \]
Consequently, as $x$ is $p$-periodic under $T$, we have
\[
\begin{aligned} |h_{\alpha}(T^k x) - h_{\alpha}(T^k y)| &= |h_{\alpha}(T^{k-p} x) - h_{\alpha}(T^k y)| \\
&\leq |h_{\alpha}(T^{k-p} x) - h_{\alpha}(T^{k-p} y)| + |h_{\alpha}(T^{k - p} y ) - h_{\alpha}(T^{k} y )|\\
&\leq |h_{\alpha}(T^{k-p} x) - h_{\alpha}(T^{k-p} y)| \\
&\phantom{\leq} \; + k \max \{|h_{\alpha}(x) - h_{\alpha}(y)|, \ldots, |h_{\alpha}(T^{p-1}x) - h_{\alpha}(T^{p-1}y)|,\\
&\phantom{\leq + k \max \{ }\;\: |h_{\alpha}(y) - h_{\alpha}(T^p y)|, \ldots, |h_{\alpha}(T^{k-p-1}y) - h_{\alpha}(T^{k-1} y)|  \} \\
& \leq |h_{\alpha}(T^{k-p} x) - h_{\alpha}(T^{k-p} y)| + k \| \phi^T_{\alpha,p}(x) - \phi^T_{\alpha,p}(y)\| \\
&\phantom{\leq} \; + k \max\{ |h_{\alpha}(x) - h_{\alpha}(y)| + |h_{\alpha}(T^p x) - h_{\alpha}(T^p y)|, \ldots,\\
&\phantom{\leq + k \max \{ }\;\: |h_{\alpha}(T^{k-p-1}x) - h_{\alpha}(T^{k-p-1}y)| + |h_{\alpha}(T^{k-1} x) - h_{\alpha}(T^{k-1} y)| \} \\
& = |h_{\alpha}(T^{k-p} x) - h_{\alpha}(T^{k-p} y)| + k \| \phi^T_{\alpha,p}(x) - \phi^T_{\alpha,p}(y)\| \\
&\phantom{\leq} \; + k \max\{ |h_{\alpha}(x) - h_{\alpha}(y)|, \ldots, |h_{\alpha}(T^{k-p-1} x) - h_{\alpha}(T^{k-p-1} y)|\}\\
&\phantom{\leq} \; + k \max \{ |h_{\alpha}(T^{p}x) - h_{\alpha}(T^{p}y)|, \ldots, |h_{\alpha}(T^{k-1} x) - h_{\alpha}(T^{k-1} y)| \} \\
& \leq k\| \phi^T_{\alpha,k-p + 1}(x) - \phi^T_{\alpha,k-p + 1}(y)\| + k \| \phi^T_{\alpha,k}(x) - \phi^T_{\alpha,k}(y)\| \\
& \leq 2k \| \phi^T_{\alpha,k}(x) - \phi^T_{\alpha,k}(y)\|,
\end{aligned},
\]
which completes the proof of Proposition~\ref{prop: rank or predict} in Case~4.

\medskip

Note that Cases 1--4 (by applying them with $x$ and $y$ possibly exchanged) cover all the possibilities when at least one of the points $x,y$ is periodic or aperiodic. Therefore, it remains to consider the case when both points $x$ and $y$ are pre-periodic. Now we explain how to reduce this case to the previous ones. We can assume $\#\Orb(x) \leq k$ and $\#\Orb(y) \leq k$, as otherwise Case~1 applies (which requires no assumptions on the periodicity of $x$ or $y$). Moreover, as noted above, we can assume $\Orb(x) \cap \Orb(y) = \emptyset$, as the case of non-disjoint orbits follows from Cases 1--2, which do not require the assumption that at least one of the points is periodic. By symmetry, we can assume $0 < P(x) \leq P(y)$. Then the matrix $J_{x,y}$ is of the following block form, consisting of four column blocks:
\[\renewcommand{\arraystretch}{1.4} J_{x,y} = \begin{bNiceArray}{ccccIcccccIcccccccIcccccc}
\Block[borders={bottom, tikz=dashed}]{2-4}<\Large>{I_{P(x)}} & & & & \Block[borders={bottom, tikz=dashed}]{2-5}<\Large>{0} & & & & &   \Block[borders={bottom, tikz=dashed}]{4-7}<\Large>{-I_{P(y)}} & & & & & & & \Block[borders={bottom, tikz=dashed}]{4-6}<\Large>{0} & & & & & \\
& & & & & & & & & & & & & & & & & & & & & \\
\Block{6-4}<\Large>{0} & & & & \Block{6-5}<\Large>{C_x}  & & & & & & & & & & & & & &  \\
& & & & & & & & & & & & & & & & &  \\ 
& & & & & & & & &    \Block{4-7}<\Large>{0} & & & & & & & \Block{4-6}<\Large>{-C_y} & & & & &  \\
& & & & & & & & & & & & & & & & & & & & & \\
& & & & & & & & & & & & & & & & & & & & & \\
& & & & & & & & & & & & & & & & & & & & &
\end{bNiceArray},\]
where $I_{P(x)}$ and $I_{P(y)}$ are identity matrices of size $P(x) \times P(x)$ and $P(y) \times P(y)$, respectively, and $C_x, C_y$ are circulant matrices of the form 
\[
C_x = 
\begin{bNiceArray}{ccc}
	1 & \greyzero &  \greyzero \\
	 \greyzero & 1 &  \greyzero \\
	 \greyzero &  \greyzero & 1 \\
\Block[borders={top, bottom, tikz=dashed}]{3 - 3}{} 1 &  \greyzero &  \greyzero \\
	 \greyzero & 1 &  \greyzero \\
	 \greyzero &  \greyzero & 1 \\
	1 &  \greyzero &   \greyzero\\
	\CodeAfter
	\UnderBrace{7-1}{7-3}{C(x)}[yshift=2mm]
	\SubMatrix{.}{1-3}{7-3}{\rbrace}[xshift=2mm]
\end{bNiceArray}
\hspace{3.5mm}
\begin{NiceArray}{l}\\[-1mm]
\Block{5-1}{k - P(x),}\\[2mm]
\\
\\
\\
\\
\\
\end{NiceArray}
\ \ \ \ \ \ \ \ \ \ \ \
 C_y = 
\begin{bNiceArray}{cc}
1 &  \greyzero \\ 
 \greyzero & 1 \\
\Block[borders={top, bottom, tikz=dashed}]{2 - 2}{}1 &  \greyzero \\
 \greyzero & 1 \\
1 &  \greyzero \\
\CodeAfter
\UnderBrace{5-1}{5-2}{C(y)}[yshift=2mm]
\SubMatrix{.}{1-2}{5-2}{\rbrace}[xshift=2mm]
\end{bNiceArray}
\hspace{3.5mm}
\begin{NiceArray}{l}
\Block{5-1}{k - P(y)}\\
\\
\\
\\
\\
\end{NiceArray}\vspace*{7mm}
\]
(again, we take here $p=3,\ k-P(x) = 7$ and $q=2,\ k-P(y) = 5$ for illustration, but the circulant structure is general).

Add the first $P(x)$ columns of $J_{x,y}$ to the first $P(x)$ columns of the third column block of $J_{x,y}$ (the one containing $-I_{P(y)}$). As $P(x) \leq P(y)$, this deletes the first $P(x)$ entries $-1$ in the third column block. Therefore, we obtain a new matrix $J'_{x,y}$ of the form
\[\renewcommand{\arraystretch}{1.4} J'_{x,y} = \begin{bNiceArray}{ccccIccccccccc}
\Block[borders={bottom, tikz=dashed}]{2-4}<\Large>{I_{P(x)}} & & & & \Block[borders={bottom, tikz=dashed}]{2-9}<\Large>{0} & & & & & & & & \\
& & & & & & & & & & & & \\
\Block{5-4}<\Large>{0} & & & & \Block{5-9}<\Large>{K_{x,y}}  & & & & & & & & \\
& & & & & & & & & & & & \\
& & & & & & & & & & & & \\
& & & & & & & & & & & & \\
& & & & & & & & & & & & \\
\end{bNiceArray},\]
such that $\rank J'_{x,y} = \rank J_{x,y}$, where $K_{x,y}$ is a $(k - P(x)) \times (C(x) + \#\Orb(y))$ matrix of the form
\[\renewcommand{\arraystretch}{1.4} K_{x,y} = \begin{bNiceArray}{cccccIcccIccccIcccccc}
\Block{6-5}<\Large>{C_x} & & & & &   \Block{6-3}<\Large>{0} & & &  \Block[borders={bottom, tikz=dashed}]{2-4}<\Large>{-I_s}   & & & & \Block[borders={bottom, tikz=dashed}]{2-6}<\Large>{0} & & & & & \\
& & & & & & & & & & & & & & & & & \\ 
& & & & & & & & \Block{4-4}<\Large>{0} & & & & \Block{4-6}<\Large>{-C_y} & & & & &  \\
& & & & & & & & & & & & & & & & & \\
& & & & & & & & & & & & & & & & & \\
& & & & & & & & & & & & & & & & &
\end{bNiceArray},\]
with $s = P(y) - P(x)$. Let $K'_{x,y}$ be the matrix obtained from $K_{x,y}$ by deleting zero columns. As there are $P(x)$ such columns, we obtain a $(k - P(x)) \times (C(x) + \#\Orb(y) - P(x))$ matrix of the form
\[\renewcommand{\arraystretch}{1.4} K'_{x,y} = \begin{bNiceArray}{cccccIccccIcccccc}
\Block{6-5}<\Large>{C_x} & & & & &   \Block[borders={bottom, tikz=dashed}]{2-4}<\Large>{-I_s}   & & & & \Block[borders={bottom, tikz=dashed}]{2-6}<\Large>{0} & & & & & \\
& & & & & & & & & & & & & & \\ 
& & & & & \Block{4-4}<\Large>{0} & & & & \Block{4-6}<\Large>{-C_y} & & & & &  \\
& & & & & & & & & & & & & & \\
& & & & & & & & & & & & & & \\
& & & & & & & & & & & & & &
\end{bNiceArray}.\]
Note that
\begin{equation}\label{eq:rank eq}  \rank J_{x,y} = \rank J'_{x,y} = P(x) + \rank K_{x,y} = P(x) + \rank K'_{x,y}.
\end{equation}
A crucial observation is that
\[K'_{x,y} = J_{T^{P(x)}x, T^{P(x)}y}(T, k - P(x)).\]
Moreover, $\rank K'_{x,y} < k - P(x)$, since otherwise \eqref{eq:rank eq} implies $\rank J_{x,y} = k$, which contradicts \eqref{eq:rankJ<}. Hence, as $T^{P(x)}x$ is periodic, one of Cases~1--4 applies for the points $T^{P(x)}x, T^{P(x)}y$ and $k - P(x)$ instead of $k$. Hence, we can use Proposition~\ref{prop: rank or predict} to show
\begin{align*}
\| &\phi^T_{\alpha, k - P(x)}(T^{P(x)+1}x) - \phi^T_{\alpha, k - P(x)}(T^{P(x)+1}y)\|\\ &\leq 2(k-P(x)) \| \phi^T_{\alpha, k - P(x)}(T^{P(x)}x) - \phi^T_{\alpha, k - P(x)}(T^{P(x)}y) \|.
\end{align*}
This implies
\[
\begin{aligned}
|h_{\alpha}(T^k x)  - h_{\alpha}(T^k y)| &\leq \| \phi^T_{\alpha, k - P(x)}(T^{P(x)+1}x) - \phi^T_{\alpha, k - P(x)}(T^{P(x)+1}y)\| \\
&\leq 2(k-P(x)) \| \phi^T_{\alpha, k - P(x)}(T^{P(x)}x) - \phi^T_{\alpha, k - P(x)}(T^{P(x)}y)\| \\
& \leq 2k \| \phi^T_{\alpha,k}(x) - \phi^T_{\alpha,k}(y) \|
\end{aligned}
\]
and, consequently, Proposition~\ref{prop: rank or predict} holds in this case. This completes the proof of the lemma.

\section{Covering bounds}\label{sec:trans}

In this section we give bounds on the Lebesgue measure of the set of `bad' parameters $\alpha \in \R^m$ in terms of covers of the set $X$ (or its subsets). Here and in the subsequent sections we keep all the notation from the beginning of Section~\ref{sec:combin} except that we write $\phi_\alpha$ instead of $\phi^T_{\alpha,k}$ for short. We consider also the matrix $D_{x,y}$ defined in \eqref{eq:D_xy}.

\begin{lem}\label{lem: general cover bound}
Let $Y \subset X$ be a bounded set such that $T, T^2, \ldots, T^{k}$ are Lipschitz on $Y$ and $h, h_1, \ldots, h_m$ are $\beta$-H\"older $(\beta \in (0,1])$ on each set $T^n(Y)$, $n = 0, \ldots, k$. Then there exists $D = D(Y) > 0$ such that for every $0 \le \eps \le 1$  the following hold.
\begin{enumerate}[$($i$)$]
\item\label{it:Y cover} If $\{B(y_i, \eps_i)\}_{i \in I}$ is a finite or countable cover of $Y$ by balls centred at $y_i \in Y$, with $\sigma_k(D_{x,y_i})>0$, then 
\[ \Leb\big( \big\{ \alpha \in \overline{B_m} (0,1) : \underset{y \in Y}{\exists}\ \|\phi_\alpha(x) - \phi_\alpha(y)\| \leq \eps\big\} \big) \leq D\sum \limits_{i \in I}\left(\frac{(\eps + \eps_i)^\beta}{\sigma_k(D_{x,y_i})}\right)^{k} \]
for every $x \in Y$.
\item\label{it:YxY cover} For a set $Z \subset Y \times Y$, if $\{B((x_i, y_i), \eps_i)\}_{i \in I}$ is a finite or countable cover of $Z$ by balls $($in the maximum metric on $Y \times Y)$ centred at $(x_i, y_i) \in Z$, with $\sigma_k(D_{x_i,y_i})>0$, then
\[ \Leb\big( \big\{ \alpha \in \overline{B_m} (0,1) : \underset{(x,y) \in Z}{\exists}\ \|\phi_\alpha(x) - \phi_\alpha(y)\| \leq \eps  \big\} \big) \leq D\sum \limits_{i \in I}\left(\frac{(\eps + \eps_i)^\beta}{\sigma_k(D_{x_i,y_i})}\right)^{k}. \]
\end{enumerate}
\end{lem}

\begin{rem}
Without any assumptions on the measurability of $Y$ and $Z$, the sets appearing on the left hand sides of the inequalities in assertions~\ref{it:Y cover} and~\ref{it:YxY cover} might not be Lebesgue measurable. In this case, it follows from the proof that the inequality holds for the outer Lebesgue measure.
\end{rem}

\begin{proof}[Proof of Lemma \rm\ref{lem: general cover bound}]
	By assumptions, $\phi_\alpha$ and $\phi_\alpha \circ T$ are $\beta$-H\"older on $Y$ with constants which are uniform with respect to $\alpha \in \overline{B_m} (0,1)$ (see \cite[p.~4955]{BGS20}). Hence, we can set
	\[ L = \sup \big\{ \|\phi_\alpha\|_{H^{\beta}(Y)}, \|\phi_\alpha \circ T\|_{H^{\beta}(Y)}: \alpha \in \overline{B_m} (0,1) \big\} < \infty.\]
	To prove assertion~\ref{it:Y cover}, fix $x \in Y$ and consider $\alpha \in \overline{B_m} (0,1)$ and $y \in Y$ such that $\|\phi_\alpha(x) - \phi_\alpha(y)\| \leq \eps$. Let $y_i$ be such that $\|y-y_i\| < \eps_i$. Then
\[ \|\phi_\alpha(x) - \phi_\alpha(y_i)\| \leq \|\phi_\alpha(x) - \phi_\alpha(y)\| + \|\phi_\alpha(y) - \phi_\alpha(y_i)\| \leq \eps + L \eps_i^\beta.\]
	Consequently,
	\[
	\begin{split}
		\big\{ \alpha &\in \overline{B_m} (0,1) :  \underset{y \in Y}{\exists}\ \|\phi_\alpha(x) - \phi_\alpha(y)\| \leq \eps\big\}\\  &\subset \bigcup \limits_{i \in I} \big\{ \alpha \in \overline{B_m} (0,1) : \|\phi_\alpha(x) - \phi_\alpha(y_i)\| \leq \eps + L \eps_i^\beta \big\}.
	\end{split}
	\]
		Therefore, by \eqref{e:matrix_form} and Lemma~\ref{lem: key_ineq_inter}
	\[
	\begin{split}
		\Leb & \big( \big\{ \alpha \in \overline{B_m} (0,1) : \underset{y \in Y}{\exists}\ \|\phi_\alpha(x) - \phi_\alpha(y)\| \leq \eps  \big\} \big) \\
		& \leq \sum \limits_{i \in I} \Leb (  \{ \alpha \in \overline{B_m} (0,1) : \|\phi_\alpha(x) - \phi_\alpha(y_i)\| \leq \eps + L \eps_i^\beta \} ) \\
		& = \sum \limits_{i \in I} \Leb ( \{ \alpha \in \overline{B_m} (0,1) : \|D_{x,y_i}\alpha + w_{x,y_i}\| \leq \eps + L \eps_i^\beta \} ) \\
		& \leq C \sum \limits_{i \in I} \left(\frac{\eps + L\eps_i^\beta}{\sigma_k(D_{x,y_i})}\right)^{k}  \leq D \sum \limits_{i \in I} \left(\frac{(\eps + \eps_i)^\beta}{\sigma_k(D_{x,y_i})}\right)^{k}
	\end{split}
	\]
	for some $C,D>0$.
	This completes the proof of assertion~\ref{it:Y cover}. Assertion~\ref{it:YxY cover} is proved in the same way, where instead of fixing $x \in Y$ and considering an approximation of $y$ by $y_i$, we approximate the pair $(x,y)$ by $(x_i,y_i)$. We omit the details.
\end{proof}

If $Y$ is compact, we can apply Corollary~\ref{cor:crucial} to obtain the following result.

\begin{lem}\label{lem: compact cover bound}
Let $Y \subset X$ be as in Lemma~{\rm\ref{lem: general cover bound}} and assume additionally that $Y$ is compact. Then for every $\delta>0$ there exists $D_\delta = D_\delta(Y) > 0$, such that if $\{B(y_i, \eps_i)\}_{i \in I}$ is a finite or countable cover of $Y$ by balls centred at $y_i \in Y$, then 
\[ 
\begin{split}
\Leb&\big( \big\{ \alpha \in \overline{B_m} (0,1) : \underset{y \in Y}{\exists}\ \|\phi_\alpha(x) - \phi_\alpha(y)\| \leq \eps \text{ and } \|\phi_\alpha(Tx) - \phi_\alpha(Ty)\| \geq \delta\big\} \big)\\ &\leq D_\delta \sum \limits_{i \in I}(\eps + \eps_i)^{\beta k}
\end{split}
\]
for every $x \in Y$ and $0 \leq \eps \leq \frac{\delta}{3k}$.
\end{lem}

\begin{proof}
It is enough to consider $\eps \leq 1$. Define
	\[\begin{split} F_\delta = \big\{ (x,y) \in Y \times Y : &\underset{\alpha \in \overline{B_m} (0,1)}{\exists}\ \ \|\phi_\alpha(Tx) - \phi_\alpha(Ty)\| \geq 3k  \|\phi_\alpha(x) - \phi_\alpha(y)\|\\
	& \text{and } \|\phi_\alpha(Tx) - \phi_\alpha(Ty)\| \geq \delta  \big\}.
\end{split}\]
Note that $F_\delta$ is compact as a projection of the compact set
\[\begin{split} \big\{ (x,y,\alpha) \in Y \times Y \times \overline{B_m} (0,1)\ :\ & \|\phi_\alpha(Tx) - \phi_\alpha(Ty)\| \geq 3k  \|\phi_\alpha(x) - \phi_\alpha(y)\|\\
	& \text{ and } \|\phi_\alpha(Tx) - \phi_\alpha(Ty )\| \geq \delta  \big\}
\end{split} \]
onto the first two coordinates plane. Corollary~\ref{cor:crucial} implies that the $k$-th singular value $\sigma_k(D_{u,v})$ is positive on $F_\delta$. By the continuity of $h_i \circ T^j$, the matrix $D_{x,y}$ depends continuously on $(x,y)$ and hence so does $\sigma_k(D_{x,y})$ (see e.g.~\cite[Corollary 8.6.2]{MatrixComputations}). Therefore, by the compactness of $F_\delta$, we can set
\[ s_\delta= \inf \{ \sigma_k(D_{x,y}) : x,y \in F_\delta\} > 0. \]
If we now fix $x \in Y$ and consider the set
	\[\begin{split} F_{\delta, x} = \Big\{ y \in Y : &\underset{\alpha \in \overline{B_m} (0,1)}{\exists}\ \  \|\phi_\alpha(x) - \phi_\alpha(y)\|
  \leq \frac{1}{3k }\|\phi_\alpha(Tx) - \phi_\alpha(Ty)\|\\
& \text{ and } \|\phi_\alpha(Tx) - \phi_\alpha(Ty)\| \geq \delta  \Big\}, \end{split}\]
then
\begin{equation}\label{eq:sigma c_delta}\inf\{ \sigma_k(D_{x,y}) : y \in F_{\delta, x}\} \geq s_\delta,
\end{equation}
with the bound independent of $x \in Y$. Let  $\{ B(y_i, \eps_i)\}_{i \in I}$ be a finite or countable cover of $Y$ with $y_i \in Y$. Set $J = \{ i\in I : B(y_i, \eps_i) \cap F_{\delta, x} \neq \emptyset \}$ and for $i \in J$ choose an arbitrary $z _i \in F_{\delta, x} \cap B(y_i, \eps_i)$. Clearly, $\{ B(z_i, 2\eps_i)\}_{i \in J}$ is a cover of $F_{\delta, x}$. If $\eps \leq \frac{\delta}{3k}$, then by the definition of $F_{\delta, x}$, Lemma \ref{lem: general cover bound}\ref{it:Y cover} and \eqref{eq:sigma c_delta}, there exists $D=D(Y)$ such that
\[\begin{split}
\Leb&\big( \big\{ \alpha \in \overline{B_m} (0,1) : \underset{y \in Y}{\exists}\ \|\phi_\alpha(x) - \phi_\alpha(y)\| \leq \eps \text{ and } \|\phi_\alpha(Tx) - \phi_\alpha(Ty)\| \geq \delta\big\} \big) \\
& \leq \Leb \big( \big\{ \alpha \in \overline{B_m} (0,1) : \underset{y \in F_{\delta,x}}{\exists}\  \|\phi_\alpha(x) - \phi_\alpha(y)\| \leq \eps \big\} \big) \\
& \leq D \sum \limits_{i \in J}\left(\frac{(\eps + 2\eps_i)^{\beta}}{\sigma_k(D_{x,z_i})}\right)^k \leq \frac{D}{s_\delta^k} \sum \limits_{i \in J} (\eps + 2\eps_i)^{\beta k} \leq \frac{D 2^{\beta k}}{s_\delta^k} \sum \limits_{i \in I} (\eps + \eps_i)^{\beta k}.
\end{split}\]
\end{proof}

\section{Deterministic predictability -- proof of Theorem~\ref{thm:predict_determ}}\label{sec:proof_takens}

We begin by proving Proposition~\rm\ref{prop:determ_predict_equiv}, which shows the continuity of the prediction map for continuous deterministically predictable observables on compact spaces. 

\begin{proof}[Proof of Proposition~\rm\ref{prop:determ_predict_equiv}] The equivalence of assertions~(a)--(c) is obvious. Suppose $X$ is compact and $T$ and $h$ are continuous. For every closed set $F \subset \R^k$, we have
\[
S^{-1}(F) = \phi((S\circ \phi)^{-1}(F)) = \phi((\phi\circ T)^{-1}(F)).
\]
Since $\phi\circ T$ is continuous and $X$ is compact, the set $(\phi\circ T)^{-1}(F)$ is compact. Hence, by the continuity of $\phi$, the set $S^{-1}(F)$ is compact. This shows that the map $S$ is continuous.
\end{proof}

We will also repeatedly use the following lemma.

\begin{lem}\label{lem:lipschitz decomposition}
Let $T\colon X \to X$ be a locally Lipschitz map on a set $X \subset \R^N$. Let $h_1,\ldots, h_m \colon X \to \R$ be locally $\beta$-H\"older observables and let $k \in \N$. Then $X$ can be presented as a countable union $X = \bigcup_{n=1}^\infty X_n$ of bounded sets $X_n$, such that $X_n \subset X_{n+1}$ for $n \in \N$, the maps $T, T^2, \ldots, T^{k}$ are Lipschitz on $X_n$ and $h_1, \ldots, h_m$ are $\beta$-H\"older on each set $T^{i}(X_n)$, $i = 0, \ldots, k$. If $X$ is Borel, then the sets $X_n$ can be also chosen to be Borel.
\end{lem}
\begin{proof}
Using local Lipschitz and H\"older conditions together with the fact that $\R^N$ is a hereditarily Lindelöf space, one can find a countable open cover $\mathcal V$ of $X$ by bounded sets, such that for every $V \in \mathcal V$, the map $T$ is Lipschitz on $V$ and $h, h_1, \ldots, h_m$ are $\beta$-H\"older on $V$. Let $\mathcal U$ be the collection of all sets of the form $U = V_0 \cap T^{-1}(V_1) \cap \ldots \cap T^{-k}(V_k)$, where $V_0, \ldots, V_k \in \mathcal V$. Taking $X_n = X \cap \left( U_1 \cup \ldots 
\cup U_n \right)$, where $\mathcal{U} = \{U_n \}_{n=1}^\infty$, gives the desired family of sets. See the proof of \cite[Theorem 4.3]{BGS20} for more details.
\end{proof}

Let us turn now to the proof of the deterministic time-delay prediction theorem (Theorem~\ref{thm:predict_determ}). We will actually prove the following extended version of the result. Recall that $\mathcal H^s$, $s > 0$, denotes the $s$-dimensional Hausdorff measure (see Definition~\ref{defn:dim}).

\begin{thm}[{\bf Deterministic time-delay prediction theorem -- extended version}{}]\label{thm:predict_determ_ext}

Let $T\colon X \to X$ be a locally Lipschitz map on a set $X \subset \R^N$. Fix $k \in \N$ and $\beta \in (0,1]$ such that $\mH^{\beta k}(X \times X) = 0$. Let $\{h_1,\ldots, h_m\}$ be a $2k$-interpolating family in $X$ consisting of locally $\beta$-H\"older functions. Let $h \colon X \to \R$ be a locally $\beta$-H\"older observable. Then the observable $h_\alpha = h + \sum_{j=1}^m \alpha_j h_j$ is deterministically $k$-predictable for Lebesgue-almost every $\alpha = (\alpha_1, \ldots, \alpha_m) \in \R^m$. Furthermore, if $X$ is compact, then the prediction map $S_\alpha$  corresponding to $h_\alpha$ is continuous for Lebesgue-almost every $\alpha \in \R^m$.
\end{thm}
\begin{proof}[Proof of Theorem~{\rm\ref{thm:predict_determ}} assuming Theorem~{\rm\ref{thm:predict_determ_ext}}]

Since a countable intersection of prevalent sets is prevalent (see e.g.~\cite{Prevalence92}), we can fix a number $k > 2\lbdim X$. Note that by \eqref{eq:hdim_mbdim_bdim} and \cite[Lemma~A1]{ORS16},
\begin{equation}\label{eq:product dim comp}
\hdim(X \times X) \leq \lbdim(X \times X) = 2\lbdim X  < k,
\end{equation}
hence $\mH^k(X \times X) = 0$ and Theorem \ref{thm:predict_determ_ext} can be applied with $\beta = 1$. The prevalence of deterministically predictable observables in the space $\Lip_{loc}(X)$ follows from Remark \ref{rem:probe set}. 
\end{proof}

\begin{proof}[Proof of Theorem \rm\ref{thm:predict_determ_ext}]
We will prove that for Lebesgue almost every $\alpha \in \R^m$,
\begin{equation}\label{eq: dtrm predict extended} \text{for every } x, y \in X, \text{ if } \phi_\alpha(x) = \phi_\alpha(y), \text{ then }\phi_\alpha(Tx) = \phi_\alpha(Ty).
\end{equation}
Without loss of generality one can show \eqref{eq: dtrm predict extended} for almost every $\alpha$ in the closed unit ball $\overline{B_m} (0,1)$(cf.~\cite[Proof of Theorem~3.1]{BGS20}). Let $X = \bigcup_{n=1}^\infty X_n$ be the decomposition from Lemma~\ref{lem:lipschitz decomposition} (so that we are in position to invoke Lemma~\ref{lem: general cover bound} with $Y=X_n$). Note that in order to establish \eqref{eq: dtrm predict extended}, it suffices to prove $\Leb(A_{n}) = 0$ for every $n \in \N$, where
\[ A_{n} = \big\{ \alpha \in \overline{B_m} (0,1) : \underset{x,y \in X_{n}}{\exists}\  \phi_\alpha(x) = \phi_\alpha(y) \text{ and } \phi_\alpha(Tx) \neq \phi_\alpha(Ty)  \big\}. \]
By Corollary \ref{cor:crucial}, $\sigma_k(D_{x,y}) > 0$ for $x,y \in X$  and $\alpha \in \R^m$ such that $\phi_\alpha(x) = \phi_\alpha(y)$ and $\phi_\alpha(Tx) \neq \phi_\alpha(Ty)$.
 Therefore, defining for $n,j \in \N$ the sets
\[ Z_{n,j} = \Big\{ (x,y) \in X_n \times X_n : \sigma_k(D_{x,y}) > \frac{1}{j} \Big\} \]
and
\[ A_{n, j} = \big\{ \alpha \in \overline{B_m} (0,1) : \underset{(x,y) \in Z_{n,j}}{\exists}\  \phi_\alpha(x) = \phi_\alpha(y) \text{ and } \phi_\alpha(Tx) \neq \phi_\alpha(Ty)  \big\}, \]
we have
\[ A_{n} = \bigcup \limits_{j=1}^{\infty} A_{n, j}.\]
Thus, our goal is to prove\footnote{ Note that we do not check directly the measurability of the sets $A_{n}$ and $A_{n, j}$. However, it follows from the proof that their outer Lebesgue measure is zero, hence they are indeed Lebesgue measurable.}
\begin{equation}\label{eq:ndj goal}
\Leb(A_{n,j}) = 0 \quad \text{for every } n,j \in \N.
\end{equation}
To this end, use Lemma~\ref{lem: general cover bound} for $Y = X_n$ to find a suitable number $D = D(X_n)$. Fix $\rho > 0$. By assumption, 
\[
\mH^{\beta k}(Z_{n,j}) \leq\mH^{\beta k}(X_n \times X_n) \le \mH^{\beta k}(X \times X) = 0,
\]
so there exists a countable cover of $Z_{n,j}$ by balls $\{B((x_i, y_i), \eps_i)\}_{i \in \N}$ of radii $\eps_i > 0$ with centres $(x_i, y_i) \in Z_{n,j}$, such that
\[
\sum \limits_{i=1}^\infty \eps_i^{\beta k} \leq \rho.
\]
By Lemma~\ref{lem: general cover bound}\ref{it:YxY cover} applied with $\eps=0$, we have
\[ \Leb (A_{n,j}) \leq \frac{D}{j^k} \sum \limits_{i =1}^\infty\eps_i^{\beta k} \leq \frac{D\rho}{j^k}.\]
As for a fixed $n$ and $j$ we can take $\rho$ arbitrarily small, the above yields \eqref{eq:ndj goal}. This completes the proof of \eqref{eq: dtrm predict extended}, which implies the existence of the prediction map $S_\alpha$ for almost every $\alpha \in \R^m$. The continuity of $S_\alpha$ for a compact $X$ follows from Proposition~\ref{prop:determ_predict_equiv}.
\end{proof}

\section{Almost sure predictability -- proofs of Theorems~\ref{thm:almost_sure_predict_equiv} and \ref{thm:predict_prob}}\label{sec:proof_predict_prob}

We start this section by proving Proposition~\ref{prop:FS}, describing a relation of $\chi_\eps$ and $\sigma_\eps$ from Definition~\ref{def:dynamical predictability} to the Farmer--Sidorowich algorithm for ergodic systems.

\begin{proof}[Proof of Proposition~\rm\ref{prop:FS}] 

Consider $x\in X$, $y \in \R^k$, $\eps > 0$ and the $k$-delay coordinate map $\phi$ corresponding to $h$. Note that since $\mu$ is $T$-invariant, we have $\int |h|\circ T^i \, d\mu = \int |h| \, d\mu$ for $i \in \N$, so the assumption $h \in L^1(\mu)$ implies $\phi\circ T \in L^1(\mu)$. 

Since $\mu$ is $T$-invariant and ergodic, by the Birkhoff ergodic theorem (see e.g.~\cite[Theorem~1.14]{W82}) we obtain
\begin{equation}\label{eq:pred_to_chi}
\begin{aligned}
\Pred_{x,n,\eps}(y) &= \frac{\sum_{i=0}^{n-1}\mathds{1}_{B(y, \eps)} \circ \phi(T^i(x)) \; \phi\circ T(T^i(x))/n}{\sum_{i=0}^{n-1}\mathds{1}_{B(y, \eps)} \circ \phi(T^i(x))/n}\\
&\xrightarrow[n \to \infty]{} \frac{\int \mathds{1}_{B(y, \eps)} \circ\phi \cdot \phi\circ T \, d\mu}{\int \mathds{1}_{B(y, \eps)} \circ\phi \, d\mu} = \frac{\int_{\phi^{-1}(B(y,\eps))} \phi\circ T \, d\mu}{\mu(\phi^{-1}(B(y,\eps)))} = \chi_{\eps}(y)
\end{aligned}
\end{equation}
for $\mu$-almost every $x \in X$, $\phi_*\mu$-almost every $y \in \R^k$ and $\eps > 0$ (in particular, for sufficiently large $n$ we have $\II_n = \II_n(x,y,\eps) \neq \emptyset$ and hence $\Pred_{x,n,\eps}(y)$ is well-defined). This proves the first assertion of the proposition.

To show the second assertion, note that for $\mu$-almost every $x \in X$, $\phi_*\mu$-almost every $y \in \R^k$, $\eps > 0$ and sufficiently large $n$, 
\begin{align*}
\Var_{x,n,\eps}(y) &= \frac{1}{\# \II_n} \sum_{i \in \II_n} \| z_{i+1}(x) -\chi_{\eps}(y) + \chi_{\eps}(y) - \Pred_{x,n,\eps}(y)\|^2\\
&= \frac{1}{\# \II_n} \sum_{i \in \II_n} \| z_{i+1}(x) -\chi_{\eps}(y)\|^2 + \frac{1}{\# \II_n} \sum_{i \in \II_n} \|\chi_{\eps}(y) - \Pred_{x,n,\eps}(y)\|^2\\
&+ \frac{2}{\# \II_n} \sum_{i \in \II_n} \langle z_{i+1}(x) -\chi_{\eps}(y), \chi_{\eps}(y) - \Pred_{x,n,\eps}(y)\rangle = I_1 + I_2 + I_3.
\end{align*}
Since $h \in L^2(\mu)$, we have $\phi\circ T \in L^2(\mu)$ (as $\int |h|^2\circ T^i \, d\mu = \int |h|^2 \, d\mu$ for $i \in \N$), so $\|\phi\circ T - \chi_{\eps}(y)\|^2 \in L^1(\mu)$ and, analogously as in \eqref{eq:pred_to_chi}, by the Birkhoff ergodic theorem we obtain
\[
I_1 \xrightarrow[n \to \infty]{} \frac{\int_{\phi^{-1}(B(y,\eps))} \|\phi\circ T  - \chi_{\eps}(y)\|^2 \, d\mu}{\mu(\phi^{-1}(B(y,\eps)))} = (\sigma_\eps(y))^2
\]
for $\mu$-almost every $x$, $\phi_*\mu$-almost every $y$ and $\eps > 0$. By \eqref{eq:pred_to_chi},
\[
I_2 = \|\chi_{\eps}(y) - \Pred_{x,n,\eps}(y)\|^2 \xrightarrow[n \to \infty]{} 0
\]
for $\mu$-almost every $x$, $\phi_*\mu$-almost every $y$ and $\eps > 0$. Furthermore, by the Schwarz inequality,
\[
|I_3| \le 2\|\chi_{\eps}(y) - \Pred_{x,n,\eps}(y)\| \frac{1}{\# \II_n} \sum_{i \in \II_n} \|z_{i+1}(x) -\chi_{\eps}(y)\|.
\]
Again by the Birkhoff ergodic theorem (as $\|\phi\circ T(\cdot) - \chi_{\eps}(y)\| \in L^1(\mu)$), we have
\[
\frac{1}{\# \II_n} \sum_{i \in \II_n} \|z_{i+1}(x) -\chi_{\eps}(y)\|\xrightarrow[n \to \infty]{} \frac{\int_{\phi^{-1}(B(y,\eps))} \|\phi\circ T  - \chi_{\eps}(y)\| \, d\mu}{\mu(\phi^{-1}(B(y,\eps)))} 
\]
for $\mu$-almost every $x \in X$, $\phi_*\mu$-almost every $y \in \phi(X)$ and $\eps > 0$,
so $|I_3|\xrightarrow[n \to \infty]{} 0$ by \eqref{eq:pred_to_chi}.
We conclude that $\Var_{x,n,\eps}(y) \xrightarrow[n \to \infty]{} 
(\sigma_\eps(y))^2$ for $\mu$-almost every $x \in X$, $\phi_*\mu$-almost every $y \in \R^k$ and $\eps > 0$, which shows the second assertion of the proposition.
\end{proof}

Now we prove the results on the properties of continuous almost sure deterministically predictable observables for continuous systems and the relations between the two notions of almost sure predictability (Proposition~\ref{prop:almost_sure_predict_properties} and Theorem~\ref{thm:almost_sure_predict_equiv}).

\begin{proof}[Proof of Proposition~\rm\ref{prop:almost_sure_predict_properties}]
First, we prove assertion~(i).  Since $X_h$ is Borel, by the regularity of the measure $\mu$ we can assume that it is $\sigma$-compact (i.e.~a countable union of compact sets). Then, as $\phi$ is continuous by assumption, the set $\phi(X_h)$ is $\sigma$-compact, hence Borel. Moreover, for every closed set $F \subset \R^k$, we have
\[
S^{-1}(F) = \phi(X_h \cap (S\circ \phi)^{-1}(F)) = \phi(X_h \cap (\phi\circ T)^{-1}(F)).
\]
Since $\phi\circ T$ is continuous, the set $(\phi\circ T)^{-1}(F)$ is closed, so $X_h \cap (\phi\circ T)^{-1}(F)$ is $\sigma$-compact. Hence, by the continuity of $\phi$, the set $S^{-1}(F)$ is $\sigma$-compact, hence Borel. This shows that the map $S$ is Borel and proves assertion~(i). 

To show assertion~(ii), assume $\phi\circ T \in L^1(\mu)$ and note that this holds in particular if $h$ is bounded, since then $\|\phi\circ T\| \le k \sup|h|$. Then by the almost sure deterministic predictability of $h$ we have $S \in L^1(\phi_*\mu)$, the integral differentiation theorem (see e.g.~\cite[Theorem~2.9.8]{Federer-book} or \cite[Lemma~4.1.2]{LY85I}) gives
\begin{equation*}\label{eq:chi_limit}
\begin{aligned}
\chi_{\eps}(y) &=  \frac{1}{\mu(\phi^{-1}(B(y, \eps)))} \int \limits_{\phi^{-1}(B(y, \eps))} \phi\circ T \, d\mu,\\
&=  \frac{1}{\mu(\phi^{-1}(B(y, \eps)))} \int \limits_{\phi^{-1}(B(y, \eps))} S\circ \phi \, d\mu,\\
&=  \frac{1}{\phi_*\mu(B(y, \eps)))} \int \limits_{B(y, \eps)} S \, d\phi_*\mu \xrightarrow[\eps\to 0]{} S(y)
\end{aligned}
\end{equation*}
for $\phi_*\mu$-almost every $y \in \phi(X)$ and $\eps > 0$. This proves assertion~(ii).

Finally, note that if $\mu$ is $T$-invariant, then the condition $h \in L^1(\mu)$ implies $\phi\circ T \in L^1(\mu)$ (see the proof of Proposition~\ref{prop:FS}). 

\end{proof}

\begin{proof}[Proof of Theorem~\rm\ref{thm:almost_sure_predict_equiv}]
To show assertion~(i), suppose $\phi\circ T \in L^2(\mu)$ and note (similarly as in the proof of Proposition~\ref{prop:almost_sure_predict_properties}) that this holds in particular if $h$ is bounded. For $\phi_*\mu$-almost every $y \in \phi(X)$ and $\eps > 0$, by the almost sure deterministic predictability of $h$ we have
\begin{align*}
(\sigma_{\eps}(y))^2
&= \frac{1}{\mu(\phi^{-1}(B(y, \eps)))} \int \limits_{\phi^{-1}(B(y, \eps))} \|\phi\circ T - \chi_{\eps}(y)\|^2d\mu\\
&=\frac{1}{\mu(\phi^{-1}(B(y, \eps)))} \int \limits_{\phi^{-1}(B(y, \eps))} \|S\circ \phi - \chi_{\eps}(y)\|^2d\mu\\
&=\frac{1}{\phi_*\mu(B(y, \eps))} \int \limits_{B(y, \eps)} \|S - \chi_{\eps}(y)\|^2 d\phi_*\mu.
\end{align*}
Since by the Schwarz inequality,
\[
\|S - \chi_{\eps}(y)\|^2 = \|S - S(y) + S(y) -\chi_{\eps}(y)\|^2 \le 2 \big(\|S - S(y)\|^2 + \|S(y) - \chi_{\eps}(y)\|^2\big),
\]
we obtain
\begin{align*}
(\sigma_{\eps}(y))^2
&\le\frac{2}{\phi_*\mu(B(y, \eps))} \int \limits_{B(y, \eps)} \|S - S(y)\|^2 d\phi_*\mu\\
&+ \frac{2}{\phi_*\mu(B(y, \eps))} \int \limits_{B(y, \eps)} \|S(y) - \chi_{\eps}(y)\|^2 d\phi_*\mu = I_1 + I_2.
\end{align*}
By the almost sure deterministic predictability of $h$ we have $S \in L^2(\phi_*\mu)$, so $\|S - S(y)\|^2 \in L^1(\phi_*\mu)$. Hence,  the integral differentiation theorem provides
\[
I_1 \xrightarrow[\eps\to 0]{} 2\|S(y) - S(y)\|^2 = 0
\]
for $\phi_*\mu$-almost every $y \in \phi(X)$. Furthermore, by Proposition~\ref{prop:almost_sure_predict_properties}(ii),
\[
I_2 = 2\|S(y) - \chi_{\eps}(y)\|^2
\xrightarrow[\eps\to 0]{} 0
\]
for $\phi_*\mu$-almost every $y \in \phi(X)$. This shows that $h$ is almost surely $k$-predictable, proving assertion~(i). 

Now we prove assertion~(ii). In view of assertion~(i), it suffices to show that if  $h$ is almost surely $k$-predictable, then it is almost surely deterministically $k$-predictable. To do it, we use the properties of the system of conditional measures $\{\mu_y\}_{y \in \phi(X)}$ for the map $\phi$, in the sense of \cite[Definition~2.4]{BGS22}. Recall that $\mu_y$ is a (possibly zero) Borel measure on $\phi^{-1}(\{y\})$ for every $y \in \phi(X)$, 
$\mu_y$ is a Borel probability measure for $\phi_*\mu$-almost every $y \in \phi(X)$, and $\mu(A) = \int_{\phi(X)} \mu_y(A) d\phi_*\mu(y)$ for every $\mu$-measurable set $A \subset X$. By \cite[Corollary~3.3]{BGS22}, for $\phi_*\mu$-almost every $y \in \phi(X)$ there exists a set $X^{(y)} \subset \phi^{-1}(\{y\})$ of full $\mu_y$-measure, such that $\phi\circ T$ is constant on $X^{(y)}$. It follows that for $\phi_*\mu$-almost every $y \in \phi(X)$, we have $X^{(y)} \subset \tilde X_h$ for 
\[
\tilde X_h = \Big\{x \in X: \phi(Tx) = \int \phi(T z) \, d\mu_{\phi(x)}(z)\Big\}. 
\]
Note that the function $\phi(X) \ni y \mapsto \int \phi\circ T \, d\mu_y$ is $\phi_*\mu$-measurable by \cite[Theorem~2.5]{BGS22}, so the function $X \ni x \mapsto \int \phi\circ T \, d\mu_{\phi(x)}$ is $\mu$-measurable. Consequently, the set $\tilde X_h$ is $\mu$-measurable. Hence, by the properties of the system of conditional measures,
\[
\mu(X \setminus \tilde X_h) = \int_{\phi(X)} \mu_y(X \setminus \tilde X_h) \, d\phi_*\mu(y) \le \int_{\phi(X)} \mu_y(X \setminus  X^{(y)}) \, d\phi_*\mu(y) = 0,
\]
so the set $\tilde X_h$ has full $\mu$-measure. By the definition of $\tilde X_h$,  if $x_1, x_2 \in \tilde X_h$ and $\phi(x_1) = \phi(x_2)$, then $\phi\circ T(x_1) = \phi\circ T(x_2)$. The regularity of the measure $\mu$ implies that there exists a Borel set $X_h \subset \tilde X_h$ of full $\mu$-measure. This shows that $h$ is almost surely deterministically $k$-predictable, proving assertion~(ii).

As previously, note that if $\mu$ is $T$-invariant, then the condition $h \in L^2(\mu)$ implies $\phi\circ T \in L^2(\mu)$ (see the proof of Proposition~\ref{prop:FS}). 
\end{proof}

Now we prove the almost sure time-delay prediction theorem (Theorem~\ref{thm:predict_prob}). Again, we actually show its extended version. 

\begin{thm}[{\bf Almost sure time-delay prediction theorem -- extended version}{}]
\label{thm:predict_prob_ext}

Let $X \subset \R^N$ be a Borel set, $\mu$ a Borel probability measure on $X$ and $T\colon X \to X$ a locally Lipschitz map. Fix $k \in \N$ and $\beta \in (0,1]$ such that $\mu \perp \mH^{\beta k}$, and let $\{h_1,\ldots, h_m\}$ be a $2k$-interpolating family in $X$ consisting of locally $\beta$-H\"older functions. Let $h \colon X \to \R$ be a locally $\beta$-H\"older observable. Then for Lebesgue almost every $\alpha = (\alpha_1, \ldots, \alpha_m) \in \R^m$, the observable $h_\alpha = h + \sum_{j=1}^m \alpha_j h_j$ is almost surely deterministically $k$-predictable with respect to $\mu$, i.e.~there exists a prediction map $S_\alpha$ defined on a Borel set $X_\alpha \subset X$ of full $\mu$-measure. If $\mu$ is $T$-invariant, then $X_\alpha$ can be chosen to satisfy $T(X_\alpha) \subset X_\alpha$. If, additionally, $X$ is compact and $\mH^{\beta k}(\supp \mu) = 0$, then $S_\alpha$ is continuous for almost every $\alpha \in \R^m$.
\end{thm}
\begin{proof}[Proof of Theorem~{\rm\ref{thm:predict_prob}} assuming Theorem~{\rm\ref{thm:predict_prob_ext}}]

Again, since a countable intersection of prevalent sets is prevalent, we can fix a number $k > \hdim \mu$. Fix $\beta = 1$ and apply Theorem~\ref{thm:predict_prob_ext}, noting that $\hdim \mu < k$ implies $\mu \perp \mH^k$. As before, the prevalence of almost surely predictable observables follows from Remark~\ref{rem:probe set}. 
\end{proof}

\begin{proof}[Proof of Theorem~\rm\ref{thm:predict_prob_ext}]

The proof is similar to the proof of Theorem~\ref{thm:predict_determ_ext}. The main difference is that due to Fubini's theorem we can work varying $y$ for a fixed $x$ rather than varying a pair $(x,y)$. Hence, it suffices to consider covers of $X$ rather than covers of $X \times X$. Detailed arguments are given below.

Once again, note that it is enough to prove the theorem for almost every $\alpha \in \overline{B_m} (0,1)$. 
By the assumption $\mu \perp \mH^k$, there exists a Borel subset of $X$ of full $\mu$-measure and zero $\mH^k$-measure. Combining this with Lemma~\ref{lem:lipschitz decomposition}, we can find a non-decreasing sequence of bounded Borel sets $X_n$ of positive $\mu$-measure, satisfying  $\mu \left( \bigcup_{n=1}^\infty X_n \right) = 1$, $\mH^k \left( \bigcup_{n=1}^\infty X_n \right) = 0$, such that $T, T^2, \ldots, T^{k}$ are Lipschitz on $X_n$ and $h, h_1, \ldots, h_m$ are $\beta$-H\"older on each set $T^{i}(X_n)$, $i= 0, \ldots, k$. Denote by 
\[
\mu_n = \frac{1}{\mu(X_n)}\mu|_{X_n}
\]
the normalized restriction of $\mu$ to $X_n$. By Theorem~\ref{thm:almost_sure_predict_equiv}, to prove the main assertion of the theorem, it is enough to show that for a fixed $n \geq 1$, for almost every $\alpha \in \overline{B_m} (0,1)$ there exists a Borel set $X_\alpha \subset X_n$ such that $\mu_n(X_\alpha) = 1$ and
\[
\text{if} \quad \phi_\alpha(x) = \phi_\alpha(y), \quad \text{then} \quad \phi_\alpha(Tx) = \phi_\alpha(Ty) \quad \text{for every } x,y \in X_\alpha. 
\]
To this end, it is enough to prove that for every $x \in X$ we have
\begin{equation}\label{eq:leb of bad points hdim} \Leb\big( \big\{ \alpha \in \overline{B_m} (0,1) : \underset{y \in X_n}{\exists}\  \phi_\alpha(x) = \phi_\alpha(y) \text{ and } \phi_\alpha(Tx) \neq \phi_\alpha(Ty) \big\} \big) = 0.
\end{equation}
Indeed, if \eqref{eq:leb of bad points hdim} holds, then by Fubini's theorem\footnote{ The measurability of the set to which Fubini's theorem is applied can be checked e.g.~using \cite[Lemma~2.4]{BGS20}, see the proof of \cite[Theorem~3.1]{BGS20}.} for the measure $\mu_n \otimes \Leb$,
\[ \mu_n \otimes \Leb \big( \big\{ (x,\alpha) \in X_n \times \overline{B_m} (0,1) : \underset{y \in X_n}{\exists}\  \phi_\alpha(x) = \phi_\alpha(y) \text{ and } \phi_\alpha(Tx) \neq \phi_\alpha(Ty)  \big\} \big) = 0 \]
and consequently,
\begin{equation}\label{eq:zero measure set}
\mu_n \big( \big\{ x \in X_n : \underset{y \in X_n}{\exists}\  \phi_\alpha(x) = \phi_\alpha(y) \text{ and } \phi_\alpha(Tx) \neq \phi_\alpha(Ty) \big\} \big) = 0
\end{equation}
for Lebesgue-almost every $\alpha \in \overline{B_m} (0,1)$.
Then a suitable set $X_\alpha \subset X_n$ of full $\mu_n$-measure can be defined as the complement of the zero measure set from \eqref{eq:zero measure set}. 

To prove \eqref{eq:leb of bad points hdim}, fix $x \in X$ and for $n,j \in \N$ define
\[ X_{n,j} = \Big\{ y \in X_n : \sigma_k(D_{x,y}) > \frac{1}{j} \Big\} \]
and
\[ A_{n, j} = \big\{ \alpha \in \overline{B_m} (0,1) : \underset{y \in X_{n,j}}{\exists}\  \phi_\alpha(x) = \phi_\alpha(y) \text{ and } \phi_\alpha(Tx) \neq \phi_\alpha(Ty)  \big\}. \]
Again, Corollary~\ref{cor:crucial} implies
\[ \big\{ \alpha \in \overline{B_m} (0,1) : \underset{y \in X_n}{\exists}\  \phi_\alpha(x) = \phi_\alpha(y) \text{ and } \phi_\alpha(Tx) \neq \phi_\alpha(Ty) \big\} = \bigcup \limits_{j=1}^\infty A_{n, j}. \]
Hence, to prove \eqref{eq:leb of bad points hdim}, it is enough to show $\Leb(A_{n, j}) = 0$ for every $n,j \in \N$, $x \in X_n$. Apply Lemma~\ref{lem: general cover bound} for $Y=X_n$ to find a suitable number $D = D(X_{n,j})$. Fix $\rho > 0$. Since by assumption,
\[
\mH^{\beta k}(X_{n,j}) \leq \mH^{\beta k} \Big( \bigcup \limits_{n=1}^\infty X_n \Big) = 0,
\]
there exists a countable cover of $X_{n,j}$ by balls $B(y_i, \eps_i)$, $i \in \N$, with $y_i\in X_{n,j}$ and such that
\[
\sum \limits_{i=1}^\infty \eps_i^{\beta k} \leq \rho.
\]
By Lemma~\ref{lem: general cover bound}\ref{it:Y cover} applied with $\eps=0$,
\[ \Leb (A_{n, j}) \leq Dj^{-k}\sum \limits_{i=1}^\infty\eps_i^{\beta k} \leq Dj^{-k} \rho.\]
As for fixed $n,j$ and $x$ we can take $\rho$ arbitrarily small, we have $\Leb (A_{n,j})=0$. This gives \eqref{eq:leb of bad points hdim} and completes the proof of the existence of the prediction map $S_\alpha$ for Lebesgue-almost every $\alpha$.

The forward invariance of the set $X_\alpha$ in the case when $\mu$ is $T$-invariant follows from Remark~\ref{rem:almost_sure_deterministic predictability}.

Suppose now that $X$ is compact and consider the question of the continuity of $S_\alpha$. Unlike in Theorem~\ref{thm:predict_determ_ext}, we cannot invoke Proposition~\ref{prop:determ_predict_equiv}, as the set $X_\alpha$ is not guaranteed to be compact. We therefore have to combine the arguments form the proof above and Proposition~\ref{prop:determ_predict_equiv} to conclude a stronger property. This requires a stronger assumption $\mH^{\beta k}(\supp \mu) = 0$ (in fact, as shown in Corollary \ref{cor:no continuity with hdim}, the assumption $\mu \perp \mH^{\beta k}$, or even $\hdim \mu < \beta k$, is too weak to guarantee the continuity of $S_\alpha$ for typical $\alpha$).

Set $Y =\supp\mu$. By definition, $\mH^{\beta k}(Y) = 0$. Moreover, $Y$ is compact, $T, T^2, \ldots, T^{k}$ are Lipschitz on $Y$ and $h, h_1, \ldots, h_m$ are $\beta$-H\"older on each set $T^{i}(Y)$, $i= 0, \ldots, k$. Therefore, we can repeat the proof of \eqref{eq:leb of bad points hdim}, replacing $X_n$ by $Y$ and $\mu_n$ by $\mu$, to obtain 
\begin{equation}\label{eq:Y_leb of bad points hdim} \Leb\big( \big\{ \alpha \in \overline{B_m} (0,1) : \underset{y \in Y}{\exists}\  \phi_\alpha(x) = \phi_\alpha(y) \text{ and } \phi_\alpha(Tx) \neq \phi_\alpha(Ty) \big\} \big) = 0
\end{equation}
for every $x \in Y$. We claim this implies that for every $x \in Y$,
\begin{equation}\label{eq:continuity at x} \underset{\eps > 0}{\forall}\ \underset{\delta > 0}{\exists}\ \underset{y \in Y}{\forall}\ \| \phi_\alpha(x) - \phi_\alpha(y)\| < \delta \Longrightarrow \| \phi_\alpha(Tx) - \phi_\alpha(Ty)\| < \eps
\end{equation}
holds for almost every $\alpha \in \overline{B_m} (0,1)$. Indeed, otherwise we can find $x \in Y$ and a Lebesgue-positive measure set of $\alpha \in \overline{B_m} (0,1)$, such that 
\[
\| \phi_\alpha(x) - \phi_\alpha(y_n)\| < \frac{1}{n} \quad \text{and} \quad \| \phi_\alpha(Tx) - \phi_\alpha(Ty)\| \ge  \eps
\]
for some sequence of points $y_n \in Y$ and a fixed $\eps > 0$. Then, by the compactness of $Y$ we can take a subsequence $y_{n_k} \to y$ for some $y \in Y$, so by the continuity of $\phi_\alpha$, we have $\phi_\alpha(x) = \phi_\alpha(y)$ and $ \phi_\alpha(Tx) \ne \phi_\alpha(Ty)$ for a Lebesgue-positive measure set of $\alpha \in \overline{B_m} (0,1)$, contradicting \eqref{eq:Y_leb of bad points hdim}. Hence, for every $x \in Y$, the property \eqref{eq:continuity at x} holds for almost every $\alpha \in \overline{B_m} (0,1)$. By Fubini's theorem for the measure $\mu \otimes \Leb$, for almost every $\alpha \in \overline{B_m} (0,1)$ there exists a set $Y_\alpha \subset Y$ of full $\mu$-measure, such that \eqref{eq:continuity at x} holds for every $x \in Y_\alpha$. 

Let $\tilde X_\alpha=X_\alpha\cap Y_\alpha$. Then $\tilde X_\alpha$ is a set of full $\mu$-measure, $S_\alpha$ is defined on $\phi_\alpha(\tilde X_\alpha)$ and \eqref{eq:continuity at x} holds for every $x \in \tilde X_\alpha$. We claim that $S_\alpha$ is continuous on $\phi_\alpha(\tilde X_\alpha)$. Indeed, suppose $z_i \in \phi_\alpha(\tilde X_\alpha)$ for $i \in \N$ and $z_i \to z$ for some $z \in  \phi_\alpha(\tilde X_\alpha)$. Let $y_i, x\in \tilde X_\alpha$ such that $\phi_\alpha(y_i)=z_i$ and $\phi_\alpha(x)=z$. Passing to a subsequence, we can assume $y_i\rightarrow y$ for some $y\in Y$. By the continuity of $\phi_\alpha$, we have $\phi_\alpha(y) = z = \phi_\alpha(x)$, so \eqref{eq:continuity at x} implies $\phi_\alpha(Ty) = \phi_\alpha(Tx)$. Hence, by the continuity of $T\circ \phi_\alpha$,
\[
S_\alpha(z_i) = S_\alpha(\phi_\alpha(y_i)) = \phi_\alpha(Ty_i) \to \phi_\alpha(Ty) =\phi_\alpha(Tx) = S_\alpha(\phi_\alpha(x)) = S_\alpha(z).
\]
This shows the continuity of $S_\alpha$ on $\phi_\alpha(\tilde X_\alpha)$ for almost every $\alpha \in \overline{B_m} (0,1)$.
\end{proof}

\begin{rem}
Similarly to the proof of \cite[Theorem 4.3]{BGS20}, the main part of Theorem~\ref{thm:predict_prob_ext} (almost sure predictability of prevalent observables) can be extended to the case when the measure $\mu$ is $\sigma$-finite. We leave the details to the reader.
\end{rem}

\section{SSOY Conjectures -- proofs of Theorems~\ref{thm:SSOY1} and~\ref{thm:SSOY2}}
\label{sec:proof_ssoy}

In this section we prove the general SSOY predictability conjecture (Theorem~\ref{thm:SSOY1}) and a part of SSOY prediction error conjecture, i.e.~the prediction error estimate (Theorem~\ref{thm:SSOY2}). Again, we prove extended versions of the theorems. 

The first result, extending Theorem~\ref{thm:SSOY1}, follows directly from Theorem~\ref{thm:predict_prob_ext}, applied for $\beta = 1$, and Theorem~\ref{thm:almost_sure_predict_equiv}.

\begin{thm}[{\bf General SSOY predictability conjecture -- extended version}]{}\label{thm:SSOY1_ext}
Let $X \subset \R^N$ be a compact set, $\mu$ a Borel probability measure on $X$ and  $T\colon X \to X$ a Lipschitz map. Fix $k \in \N$ and $\beta \in (0,1]$ such that $\mu \perp \mH^{\beta k}$. Let $\{h_1,\ldots, h_m\}$ be a $2k$-interpolating family on $X$ consisting of $\beta$-H\"older functions. Let $h \colon X \to \R$ be a $\beta$-H\"older observable. Then for almost every $\alpha \in \R^m$, the observable $h_\alpha = h + \sum_{j=1}^m \alpha_j h_j$ is almost surely predictable.
\end{thm}

Recall that by Theorem~\ref{thm:almost_sure_predict_equiv}, the almost sure predictability of $h_\alpha$ in Theorem~\ref{thm:SSOY1_ext} is equivalent to the almost sure deterministic predictability of $h_\alpha$.

The extended version of the prediction error estimate Theorem~\ref{thm:SSOY2} is divided into three parts, where the first two correspond to assertion~(ii), while the third one deals with assertion~(i). First, we observe that assertion~(ii) of Theorem~\ref{thm:SSOY2} holds for all deterministically predictable continuous observables on compact space with continuous dynamics.

\begin{thm}\label{thm:det pred implies zero error}
Let $X \subset \R^N$ be a compact set and $T \colon X \to X$ a continuous map. 
If a continuous observable $h \colon X \to \R$ is deterministically $k$-predictable for some $k \in \N$ and the prediction map $S$ is continuous, then for  every $\delta >0$ there exists $\eps_0 = \eps_0(h, k, \delta)> 0$ such that 
\begin{equation*}\label{eq:mu=0}
\sigma_{\eps}(\phi(x)) < \delta
\end{equation*}
for every $x \in X$ and every $0 < \eps < \eps_0$, where $\phi$ is the $k$-delay coordinate map corresponding to $h$, and $\sigma_{\eps}$ comes from Definition~{\rm\ref{def:dynamical predictability}}. In particular, the assertion holds if $\phi$ is injective.
\end{thm}
\begin{proof}
Recall that the prediction map $S\colon \phi(X) \to \phi(X)$ satisfies $S\circ \phi = \phi \circ T$. Therefore, for every $x \in X$ and $\eps > 0$ we have
\[
S(B(\phi(x), \eps)  \cap \phi(X)) = \phi \circ T ( \phi^{-1} (B (\phi(x), \eps) ) ).
\]
Moreover, as $S$ is uniformly continuous on the compact set $\phi(X)$, for every $\delta > 0$ there exists $\eps_0 > 0$ such that for every $0 < \eps < \eps_0$ and $x \in X$,
\[ \phi \circ T ( \phi^{-1} (B (\phi(x), \eps) )) = S(B(\phi(x), \eps)  \cap \phi(X)) \subset B\Big(S( \phi(x)), \frac{\delta}{2}\Big) = B\Big(\phi(Tx), \frac{\delta}{2}\Big). \]
By the definition of $\chi_\eps$ and $\sigma_\eps$ (see Definition~\ref{def:dynamical predictability}), this gives $\|\phi(Tx) - \chi_{\eps}(\phi(x))\| \leq \frac{\delta}{2}$ and $\sigma_{\eps}(\phi(x)) \leq  \delta$. 

If $\phi$ is injective on $X$, then the prediction map $S$ is given by $S = \phi \circ T \circ \phi^{-1}$ and hence it is continuous (as $\phi^{-1}$ is continuous the  inverse of a continuous map on a compact set).
\end{proof}

Theorem~\ref{thm:det pred implies zero error} implies the following corollary. 
We write $\sigma_{\alpha, \eps}$ for $\sigma_\eps$ from Definition~\ref{def:dynamical predictability} corresponding to an observable $h_\alpha$.

\begin{cor}[{\bf Prediction error estimate, assertion~(ii) -- extended version}{}{}]\label{cor:Zero prediction error_extended}
Let $X \subset \R^N$ be a compact set, $\mu$ a Borel probability measure on $X$ and $T\colon X \to X$ a Lipschitz map. Fix $k \in \N$ and $\beta \in (0,1]$ such that $\mH^{\beta k}(X \times X) = 0$. Let $\{h_1,\ldots, h_m\}$ be a $2k$-interpolating family on $X$ consisting of $\beta$-H\"older functions an let $h \colon X \to \R$ be a $\beta$-H\"older observable. Then for almost every $\alpha = (\alpha_1, \ldots, \alpha_m) \in \R^m$ and every $\delta > 0$, there exists $\eps_0 = \eps_0(\alpha, \delta) > 0$, such that for every $0 < \eps < \eps_0$,
\[ 
\{ x \in X : \sigma_{\alpha,\eps}(\phi_\alpha(x)) > \delta \} = \emptyset,
\]
where $\phi_\alpha$ is the $k$-delay coordinate map corresponding to the observable $h_\alpha = h + \sum_{j=1}^m \alpha_j h_j$.
\end{cor}
\begin{proof}
By Theorem~\ref{thm:predict_determ_ext}, for almost every $\alpha \in \R^m$ the observable $h_\alpha$ is deterministically $k$-predictable and the prediction map is continuous. Hence, the result follows from Theorem~\ref{thm:det pred implies zero error}.
\end{proof}

Note that Theorem~\ref{thm:det pred implies zero error} together with Corollary~\ref{cor:Zero prediction error_extended} provide an extended version of  assertion~(ii) of Theorem~\ref{thm:SSOY2}. In fact, to obtain Theorem~\ref{thm:SSOY2}\ref{it:SSOY2 deterministic}, it is enough to set $\beta=1$ in Corollary~\ref{cor:Zero prediction error_extended} and use the assumption $k > 2\lbdim X$ and \eqref{eq:product dim comp} to conclude $\mH^k(X \times X) = 0$.

Finally, we prove an extended version of the main part of the prediction error estimate, corresponding to assertion~(i) of Theorem~\ref{thm:SSOY2}. We write $\chi_{\alpha, \eps}$ and $\sigma_{\alpha, \eps}$, respectively, for $\chi_\eps$, $\sigma_\eps$ from Definition~\ref{def:dynamical predictability} corresponding to an observable $h_\alpha$. Note that Theorem~\ref{thm:SSOY2}\ref{it:SSOY2 error bound} follows immediately from Theorem~\ref{thm:SSOY2_ext} by taking $\beta=1$ and observing (as previously), that as a countable intersection of prevalent sets is prevalent, we can assume the number $k$ to be fixed.

\begin{thm}[{\bf Prediction error estimate, assertion~(i) -- extended version}{}]\label{thm:SSOY2_ext}
Let $X \subset \R^N$ be a compact set, $\mu$ a Borel probability measure on $X$ and $T\colon X \to X$ a Lipschitz map. Set $\overline{D} = \udim X$. Fix $k \in \N$ and $\beta \in (0,1]$ such that $\beta k > \overline{D}$ and let $\{h_1,\ldots, h_m\}$ be a $2k$-interpolating family on $X$ consisting of $\beta$-H\"older functions. Let $h \colon X \to \R$ be a $\beta$-H\"older observable. Then for Lebesgue-almost every $\alpha = (\alpha_1, \ldots, \alpha_m) \in \R^m$, for every $\delta, \theta > 0$ there exists $C = C(\alpha,\delta, \theta) > 0$ such that 
\begin{equation*}\label{eq:ii}
    \mu\left( \{ x \in X : \sigma_{\alpha, \eps}(\phi_\alpha(x)) > \delta  \} \right) \leq  C \eps^{\beta k - \overline{D} - \theta}
\end{equation*}
for every $\eps > 0$, where $\phi_\alpha$ is the $k$-delay coordinate map corresponding to the observable $h_\alpha = h + \sum_{j=1}^m \alpha_j h_j$.
\end{thm}
\begin{proof}[Proof of Theorem \rm\ref{thm:SSOY2_ext}]
Obviously, we can assume $\theta \in (0, \beta k - \overline{D})$, so that $\beta k - \overline{D} - \theta > 0$. Fix $\delta > 0$ and $\theta \in (0, \beta k - \overline{D})$. As previously, it is sufficient to consider $\alpha \in \overline{B_m} (0,1)$. For $\eps> 0$ define a set
\[ A_{\eps} = \big\{ (x, \alpha) \in X \times \overline{B_m} (0,1) :  \underset{y \in X}{\exists}\ \|\phi_{\alpha}(x) - \phi_{\alpha}(y)\| \leq \eps  \text{ and }  \|\phi_\alpha(Tx) - \phi_\alpha(Ty)\| \geq \delta\big\}  \]
and its sections
\begin{align*}
A_{\eps, x} &= \big\{ \alpha \in  \overline{B_m} (0,1) :  \underset{y \in X}{\exists}\ \|\phi_{\alpha}(x) - \phi_{\alpha}(y)\| \leq \eps  \text{ and }  \|\phi_\alpha(Tx) - \phi_\alpha(Ty)\| \geq \delta \big\},\\
A_{\eps}^{\alpha} &= \big\{ x \in  X :  \underset{y \in X}{\exists}\ \|\phi_{\alpha}(x) - \phi_{\alpha}(y)\| \leq \eps  \text{ and }  \|\phi_\alpha(Tx) - \phi_\alpha(Ty)\| \geq \delta \big\}
\end{align*}
for $x \in X$ and $\alpha \in  \overline{B_m} (0,1)$, respectively.

Fix $j_0 \in \N$ such that $2^{-j_0} < \frac{\delta}{3k}$. 
By the assumption $\beta k > \overline{D}$ and the definition of the upper box-counting dimension, there exists $Q > 0$ such that for every $j \ge j_0$ one can find a positive integer
\[
N_j  \leq Q 2^{j(\overline{D} + \theta/2)}
\]
and points $y_1, \ldots, y_{N_j} \in X$, such that
\[ X \subset \bigcup \limits_{i=1}^{N_j} B(y_i, 2^{-j}).
\]
Applying Lemma~\ref{lem: compact cover bound} for $Y=X$ and $\eps = \eps_i = 2^{-j}$, we obtain
\[ \Leb(A_{2^{-j},x}) \leq  D_\delta N_j 2^{\beta k(1-j)}  \le   D_\delta Q 2^{\beta k} 2^{-j(\beta k - \overline{D} - \theta/2)} \] 
for every $x \in X$. Applying Fubini's theorem for the measure $\mu \otimes \Leb$ yields
\[ \mu \otimes \Leb(A_{2^{-j}}) \leq D_\delta Q 2^{\beta k} 2^{-j(\beta k - \overline{D} - \theta/2)} \]
and, consequently, 
\[ \Leb\big( \big\{ \alpha \in \overline{B_m} (0,1) :\mu (A_{2^{-j}}^{\alpha}) > C 2^{-j(\beta k - \overline{D} -\theta)} \big\} \big) \leq \frac{D_\delta Q 2^{\beta k}}{C}2^{-j\theta/2}  \]
for every $C>0$. Therefore,
\[ \sum \limits_{j=j_0}^\infty \Leb\big( \big\{ \alpha \in \overline{B_m} (0,1) :\mu \left(A_{2^{-j}}^{\alpha}\right) > C 2^{-j(\beta k - \overline{D} -\theta)} \big\} \big) \leq \frac{\tilde D_\delta}{C}\]
for some constant $\tilde D_\delta>0$ and every $C>0$. This implies 
\begin{equation*}\label{eq:full measure j limit}
\lim \limits_{C \to \infty} \Leb\big( \big\{ \alpha \in \overline{B_m} (0,1) : \underset{j \geq j_0}{\exists}\ \mu \left(A_{2^{-j}}^{\alpha}\right) > C 2^{-j(\beta k - \overline{D} -\theta)} \big\} \big) = 0.
\end{equation*}
Calculating the measure of the complementary set, we obtain
\begin{equation}\label{eq:full measure j}
\Leb\big( \big\{ \alpha \in \overline{B_m} (0,1) : \underset{C>0}{\exists}\ \underset{j \geq j_0}{\forall}\ \mu (A_{2^{-j}}^{\alpha}) \leq C 2^{-j(\beta k - \overline{D} -\theta)} \big\} \big) = 1.
\end{equation}

Take $\eps \in (0, 2^{-j_0}]$ and let $j \geq j_0$ be such that $2^{-(j+1)} < \eps \leq 2^{-j}$. Then $A_{\eps}^{\alpha} \subset A_{2^{-j}}^{\alpha}$, and hence the condition $\mu (A_{2^{-j}}^{\alpha}) \leq C 2^{-j(\beta k - \overline{D} -\theta)}$ implies 
\begin{align*}
\mu \left(A_{\eps}^{\alpha}\right) &\leq \mu(A_{2^{-j}}^{\alpha}) \le C 2^{-j(\beta k - \overline{D} -\theta)} = C 2^{\beta k-\overline{D} - \theta}  2^{-(j+1)(\beta k - \overline{D} -\theta)}\\ &\leq C 2^{\beta k - \overline{D} - \theta}  \eps^{\beta k - \overline{D} - \theta} = \tilde C \eps^{\beta k - \overline{D} - \theta}
\end{align*}
for $\tilde C =  2^{\beta k - \overline{D} - \theta}$. Therefore, \eqref{eq:full measure j} gives
\begin{equation}\label{eq:full measure eps}
\Leb\big( \big\{ \alpha \in \overline{B_m} (0,1) : \underset{C>0}{\exists}\ \underset{0 < \eps \leq 2^{-j_0}}{\forall} \ \mu \left(A_{\eps}^{\alpha}\right) \leq C \eps^{\beta k - \overline{D} - \theta} \big\} \big) = 1.
\end{equation}
Note that if $x\notin A_{\eps}^{\alpha}$ and $y \in X$ is such that $\|\phi_{\alpha}(x) - \phi_{\alpha}(y)\| \leq \eps$, then
$\|\phi_\alpha(Tx) - \phi_\alpha(Ty)\| < \delta$. Hence, by Definition~\ref{def:dynamical predictability},
\begin{equation*}\label{eq:notin A}
\text{if} \quad x\notin A_{\eps}^{\alpha}, \quad \text{then} \quad \big(\|\chi_{\alpha,\eps}(\phi_\alpha(x)) - \phi_\alpha(Tx)\| \leq \delta \ \text{and} \  \sigma_{\alpha,\eps}(\phi_\alpha(x)) \leq \delta\big).
\end{equation*}
Consequently, if $\sigma_{\alpha,\eps}(\phi_\alpha(x))>\delta$, then $x\in A_{\eps}^{\alpha}$. This together with \eqref{eq:full measure eps} implies that for almost every $\alpha \in \overline{B_m} (0,1)$ there exists $C >0$ such that 
\begin{equation}\label{eq:final}
\mu\left( \{ x \in X : \sigma_{\alpha,\eps}(\phi_\alpha(x)) > \delta  \} \right) \leq C \eps^{\beta k - \overline{D} - \theta}
\end{equation}
for every $\eps \in (0, 2^{-j_0}]$. Finally, note that to obtain \eqref{eq:final} for $\eps \ge 2^{-j_0}$, it is enough to enlarge the constant $C$ so that $C 2^{-j_0(\beta k - \overline{D} - \theta)} \ge 1$. Then \eqref{eq:final} for $\eps \ge 2^{-j_0}$ holds trivially, as $\mu(X) = 1$. This ends the proof of the theorem.
\end{proof}

\section{Counterexamples}\label{sec:ex}

In this section we present an example showing that any of the assumptions $\hdim \mu < k$ and $\uid (\mu) < k$ is in general not sufficient to guarantee the continuity of the map $S$ in Theorem~\ref{thm:predict_prob} and the prediction error estimates in Theorem~\ref{thm:SSOY2}. In both cases we prove that there exists an open set of Lipschitz observables for which the conclusion of the theorem fails. As any prevalent set is dense \cite[Fact~2']{Prevalence92}, this shows that the results are not true under the weaker assumptions.

Let
\[ A_0 = \bigcup \limits_{n = 1}^\infty \left\{ \frac{q}{2^n} : q \in \{1, 2, \ldots, 2^n - 1\} \right\}, \qquad  A = \{0,1\} \times A_0 \subset \R^2. \]
Note that $A_0$ is dense in the unit interval $[0,1]$ and $A$ is dense in the union
\[
X= \{0,1\} \times [0,1]
\]
of two copies of the unit interval. For the rest of this section we set the dynamics $T \colon X \to X$ as
\[ T(x,y) = \begin{cases} (x,y) &\text{ if } x = 0 \\ (x,1 - y) &\text{ if } x = 1 \end{cases} \]
and note that $T(A) = A$. The following proposition shows that for an open set of Lipschitz observables, the prediction map for $k = 1$ (whenever exists) cannot be continuous on the full image of $A$.

\begin{prop}\label{prop:atomic non-cont}
	There exists a non-empty open set $\mathcal{U} \subset\Lip(X)$, such that for every $h \in \mathcal{U}$ there exists $z \in A$ and a sequence $z_i \in A$ such that $h(z_i) \to h(z)$ as $i \to \infty$, but $h(Tz_i)$ does not converge to $h(Tz)$.
\end{prop}

\begin{proof}
	Let $h_0 \colon X \to \R$ be given by $h_0(x,y) = y$. Let $\mathcal{U}$ be an open neighbourhood of $h_0$ in $\Lip(X)$, such that for any $h \in \mathcal{U}$ there exist open non-empty intervals $I,J \subset [0,1]$ satisfying $h\left( \{ 0 \} \times I \right) =  h\left( \{ 1 \} \times J \right)$ and such that $h$ is a (bi-Lipschitz) homeomorphism onto its image\footnote{ It is easy to see that for a bounded set $X \subset \R^N$, the set of maps $h \in \Lip(X)$ which are bi-Lipschitz homeomorphisms onto their image is open in $\Lip(X)$. Indeed, for such map $h_0$, if $\| h - h_0\|_{\Lip(X)} \leq \eps$, then $|h(x) - h(y)| \geq |h_0(x) - h_0(y)| - |(h - h_0)(x) - (h - h_0)(y)| \geq \left( \|h_0^{-1}\|_{\Lip(h_0(X))} - \eps\right)|x - y|$, hence $h$ is bi-Lipschitz for small enough $\eps>0$, with $\|h\|_{\Lip(X)}, \|h^{-1}\|_{\Lip(X)}$ arbitrarily close to $\|h_0\|_{\Lip(X)}, \|h^{-1}_0\|_{\Lip(X)}$, respectively.}  on both sets $\{ 0 \} \times [0,1],\ \{ 1 \} \times [0,1]$. Fix $y \in J \cap A_0 \setminus \{ 1/2 \}$ and set $z = (1, y)$. Let $\tilde z \in \{ 0 \} \times I$ be such that $h(\tilde z) = h(z)$ and choose a sequence $z_i \in \{0\} \times \left( I \cap A_0 \right)$ satisfying $z_i  \to \tilde z$. We can find such $y$ and $z_i$ as $A_0$ is dense in $[0,1]$. Since $h$ is continuous, we have
	\[ h(z_i) \to h(\tilde z) = h(z), \]
	and as $T$ is the identity on $\{ 0 \} \times [0,1]$,
	\[ h(T z_i) = h(z_i) \to h(z). \]
	On the other hand, as $z \in \{ 1 \} \times \left( [0,1] \setminus \{ \frac{1}{2}\} \right)$, we have $Tz = (1, 1 -y) \neq (1, y) = z$. As $h$ is injective on $\{1\} \times [0,1]$, this gives $h(Tz) \neq h(z)$, yielding
	\[ h(T z_i) \to h(z) \neq h(Tz). \]
\end{proof}

The following corollary shows that the conditions $\hdim \mu<k$ and $\uid(\mu) < k$ are not sufficient to obtain the continuity of the prediction map $S$ in Theorem~\ref{thm:predict_prob}  for a prevalent set of Lipschitz observables. 

\begin{cor}\label{cor:no continuity with hdim}
There exist a compact set $X \subset \R^2$, a Borel probability measure $\mu$ on $X$  with $\hdim\mu = \idim(\mu) =  0$ and a Lipschitz map $T\colon X\rightarrow X$,  such that for a non-empty open set of Lipschitz observables $h \colon X \to \R$ and $k = 1$, the prediction map $S \colon\phi(X_h)\rightarrow \phi(X_h)$ $($provided it exists$)$ is not continuous on a set of positive $\phi_*\mu$-measure. Therefore, the continuity of $S$ on a set of full $\phi_*\mu$-measure does not hold for a prevalent set of Lipschitz observables $h\colon X \to \R$.
\end{cor}

\begin{proof}
Define $X$ and $T$ as above.     
Let $\mu$ be a Borel probability measure on $\R^2$, such that $\mu(A)=1$ and $\mu(\{ x\}) > 0$ for every $x \in A$. Note that $\hdim\mu = \idim(\mu) = 0$, as $\mu$ is purely atomic (the fact $\hdim\mu = 0$ follows directly from the definition of the Hausdorff dimension, while $\idim(\mu) =  0$ follows from \cite[Theorem 3]{RenyiDim}). Fix $k=1$ so that $\phi=h$. By Proposition~\ref{prop:atomic non-cont}, there exists a non-empty open set $\mathcal{U} \subset \Lip(X)$, such that for every $h \in \mathcal{U}$ there exists $z \in A$ and a sequence $z_i \in A$, $i\in \N$ such that $\phi(z_i) \to h(z)$, but $\phi(Tz_i)$ does not converge to $\phi(Tz)$. Assume for a contradiction that $S$ is well-defined and continuous on a set of full $\phi_*\mu$-measure. As $\mu(z_i)>0$, $S$ must be defined at $\phi(z_i)$ for all $i$. By \eqref{def:S}, we have $S(\phi(z_i))=\phi(Tz_i)$. Thus, $S(\phi(z_i))$ does not converge to $S(\phi(z))$, even though $\phi(z_i)$ converges to $h(z)$. 
\end{proof}
To show that Theorem \ref{thm:SSOY2} does not hold with $\hdim\mu$ or $\uid(\mu)$ replacing $\lbdim X$ or $\udim X$, we need to consider a more specific family of measures. Let $\beta_n,\ n = 1,2,\ldots$ be a positive sequence satisfying
\begin{equation}\label{eq:beta assumption} \sum \limits_{n=1}^\infty \beta_n = 1, \qquad \frac{\beta_n}{\beta_{n+1}} \leq M
\end{equation}
for some constant $M>0$.
Consider a purely atomic probability measure $\nu$ on $A$ given by
\[ \nu\Big(\Big\{\frac{q}{2^n}\Big\}\Big) = \frac{\beta_n}{2^{n-1}} \]
for $n \geq 1$ and odd $q \in \{0, \ldots, 2^n-1\}$. Note that with this definition, every dyadic rational $\frac{m}{2^n} \in (0,1)$ is an atom of $\nu$. Let
\[ \mu = \frac{1}{2}\left(\delta_0 \otimes \nu + \delta_1 \otimes \nu\right), \]
and note that $\supp\mu = X = \{0,1\} \times [0,1] \subset \R^2$. As $\mu$ is purely atomic, we have $\hdim\mu = 0$ and $\idim(\mu) = 0$ regardless of the choice of $\beta_n$. It is easy to see that $\mu$ is $T$-invariant.

\begin{prop}\label{prop:atomic prevalence}
	Fix the number of measurements $k = 1$. Let $\mu$ be the measure defined above with $\beta_n$ satisfying \eqref{eq:beta assumption}. Then there exists a non-empty open  set $\mathcal{U} \subset \Lip(X)$ and constants $c, \delta>0$, such that for every $h \in \mathcal{U}$ we have
	\[ \mu\left( \left\{ x \in X : \sigma_{ 2^{-n}}(\phi(x)) \geq \delta \right\} \right) \geq c\beta_n \]
	for every $n \in \N$ large enough. Consequently,
	\begin{enumerate}[$(a)$]
		\item assertion~\ref{it:SSOY2 deterministic} of Theorem~{\rm\ref{thm:SSOY2}} for the measure $\mu$ does not hold with $\lbdim X$ replaced by $\hdim \mu$ or $\uid(\mu)$,
		\item if $\beta_n = \frac{6}{\pi^2}n^{-2}$, then assertion~\ref{it:SSOY2 error bound} of Theorem~{\rm\ref{thm:SSOY2}} for the measure $\mu$ does not hold with $\udim X$ replaced by $\hdim \mu$ or $\uid(\mu)$.
	\end{enumerate}
\end{prop}

For the proof we will need two technical lemmas.

\begin{lem}\label{lem:atomic interval}
	Fix $K \in \N$. Then there exists $C = C(K)>0$ such that for every $n \geq 3$,
	\[ \frac{|J|}{C}  \sum_{k=n}^\infty \beta_k \leq \nu(J) \leq C|J| \sum_{k=n}^\infty \beta_k \]
	for every open interval $J \subset [0,1]$ of length at least $2^{-n}$, such that $J \cap \{\ell/2^{n-K},\ \ell \in \N\} = \emptyset$.
\end{lem}
\begin{proof}
	Note that for given $m \in \N$,
	\begin{equation*}\label{eq:q cardinality bounds} |J|2^{m-1} - 2 \leq \#\left\{ q \in \{0, \ldots, 2^{m}-1\} \setminus 2\N : \frac{q}{2^m} \in J \right\} \leq |J|2^{m-1} + 1.
	\end{equation*}
	Therefore, by the definition of $\nu$, for $n \geq 3$ we have
	\[ \nu(J) \geq \sum \limits_{m = n}^\infty (|J|2^{m-1} - 2)2^{-(m-1)}\beta_m = \sum \limits_{m = n}^\infty |J|(1 - 2^{-m+2})\beta_m  \geq \frac{1}{2}|J| \sum \limits_{m = n}^\infty \beta_m. \]
	Similarly, $|J| \geq 2^{-n}$ implies $2^{-(m-1)} \leq 2^K|J|$ for $m \geq n-K+1$, so by the assumption on $|J|$,
	\begin{align*}  
	\nu(J) &\leq \sum \limits_{m = n - K + 1}^\infty (|J|2^{m-1} + 1)2^{-(m-1)}\beta_m\\ &= \sum \limits_{m = n - K + 1}^\infty (|J| + 2^{-(m-1)})\beta_m \leq (1 + 2^K)|J| \sum \limits_{m = n - K + 1}^\infty\beta_m.
	\end{align*}
	Moreover, by \eqref{eq:beta assumption},
	\begin{align*}  
	\frac{\sum_{m = n - K + 1}^\infty\beta_m}{\sum_{m = n}^\infty \beta_m} &\leq 1 + \frac{\sum_{m = n - K + 1}^{n-1}\beta_m}{\sum_{m = n}^\infty \beta_m}\\ &= 1 + \frac{(K-1)\max\{ \beta_{n - K + 1}, \ldots, \beta_{n-1} \}}{\beta_n} \leq 1 + (K-1)M^{K-1}.
	\end{align*}  
	Combining the above estimates completes the proof.
\end{proof}

\begin{lem}\label{lem:deviation bound}
	Let $\gamma>0$ and $p \in (0,1)$. Then there exists $\delta = \delta(\gamma,p)>0$ such that if $\eta$ is a probability measure on $\R^N$ supported on $A \cup A'$ with $\dist(\overline{\conv A}, \overline{\conv A'}) \geq \gamma$ and $\eta(A) \in (p, 1-p)$, then the standard deviation of $\eta$ is at least $\delta$, i.e.
	\[ \bigg( \int \Big\| x  - \int y\, d\eta(y) \Big\|^2 d\eta(x) \bigg)^{\frac{1}{2}} \geq \delta. \]
\end{lem}

\begin{proof}
	Let 
	\[
	E = \frac{1}{\eta(A)} \int \limits_{A} y \,d\eta(y), \qquad E' = \frac{1}{\eta(A')} \int \limits_{A'} y \,d\eta(y).
	\]
	Then $E \in \overline{\conv A}$, $E' \in \overline{\conv A'}$, hence $\|E - E'\| \geq \gamma$. As $\int y\, d\eta(y) = \eta(A)E + \eta(A')E'$, we see that 
	\[ \bigg\|\int y \, d\eta(y) - E\bigg\|, \bigg\|\int y \,d\eta(y) - E'\bigg\| \geq \delta \]
	for
	\[   
	\delta =\min\{p\gamma, (1-p)\gamma\}.
	\]
	Therefore,
	\[
	\begin{split} \bigg( \int \Big\| x  - \int y\; d\eta(y) \Big\|^2 d\eta(x) \bigg)^{\frac{1}{2}} &\geq \int \Big\| x  - \int y\; d\eta(y) \Big\| d\eta(x)\\
	& = \int\limits_A \Big\| x  - \int y \,d\eta(y) \Big\| d\eta(x) + \int\limits_{A'} \Big\| x  - \int y \,d\eta(y) \Big\| d\eta(x) \\
	& \geq \bigg\| \int\limits_A \Big( x  - \int y\; d\eta(y) \Big) d\eta(x)  \bigg\| +  \bigg\|\int\limits_{A'} \Big( x  - \int y \, d\eta(y) \Big)  d\eta(x) \bigg\| \\
	& = \eta(A)\bigg\| E - \int y\, d\eta(y)\bigg\| + \eta(A')\bigg\| E' - \int y\, d\eta(y)\bigg\|\geq \delta.
	\end{split}
	\]
\end{proof}

\begin{proof}[Proof of Proposition~\rm\ref{prop:atomic prevalence}]
	Let $h_0 \colon X \to \R$ be given by $h_0(x,y) = y$ and let $J = (\frac{1}{8}, \frac{1}{4})$. Then there exists an open neighbourhood $\mathcal{U}$ of $h_0$ in $\Lip(X)$ and a constant $\gamma>0$, such that for any $h \in \mathcal{U}$ the following hold:
	\begin{itemize}
		\item there exists an open interval $I \subset (0, \frac{3}{8})$ such that $h\left( \{ 0 \} \times I \right) = h\left( \{ 1 \} \times J \right)$,
		\item the length of $I$ is at least $\frac{1}{16}$,
		\item $h$ is a bi-Lipschitz homeomorphism onto its image on $\{ 0 \} \times [0,1]$ and $\{ 1 \} \times [0,1]$, with the Lipschitz constants of $h$ and its inverse being at most $\sqrt{2}$ on those sets,
		\item $\dist \left(h\left( \{ 0 \} \times I \right), h(T \left( \{ 1 \} \times J \right) ) \right) = \dist \left(h\left( \{ 0 \} \times I \right), h(\{ 1 \} \times (\frac{3}{4}, \frac{7}{8})) \right) \geq \gamma$.
	\end{itemize}
	Fix $h \in \mathcal{U}$. For $n \in \N$ define
	\[ Q_n = \left\{ q \in \{0, \ldots, 2^n-1\} \setminus 2\N : \left(\frac{q}{2^n} - 2^{-(n+1)}, \frac{q}{2^n} + 2^{-(n+1)} \right)  \subset I \right\} \]
	and note that for $n \geq 7$,
	\begin{equation}\label{eq: Q_n card} \# Q_n \geq 2^{n-1}|I| - 2 \geq 2^{n-5} - 2 \geq 2^{n-6}.
	\end{equation}
	Fix $K=8$ and set 
	\[g \colon [0,1] \to h\left( \{1\} \times [0,1] \right),\qquad g(y) = h(1,y).
	\]
	As $h\left( \{ 0 \} \times I \right) = h\left( \{ 1 \} \times J \right)$, we can define
	\[
	\begin{split}Y_n = \bigg\{ q \in Q_n: \ &\text{the sets } g^{-1} \left(h \left(\{ 0 \} \times \left(\frac{q}{2^n} -2^{-(n+1)},  \frac{q}{2^n} + 2^{-(n+1)}\right)\right)\right)\\ &\text{and }\left(\frac{q}{2^n} - 2^{-(n+1)}, \frac{q}{2^n} + 2^{-(n+1)} \right)\\ &\text{do not contain points } \frac{\ell}{2^{n-K}},\ \ell \in \{0, \ldots, 2^{n-K}-1\} \bigg\}.
	\end{split}\]
	Note that both families 
	\[ \left\{ \left(\frac{q}{2^n} - 2^{-(n+1)}, \frac{q}{2^n} + 2^{-(n+1)} \right) : q \in \{0, \ldots, 2^{n}-1 \} \setminus 2\N \right\} \]
	and 
	\[ \left\{ g^{-1} h \left(\{ 0 \} \times \left(\frac{q}{2^n} -2^{-(n+1)},  \frac{q}{2^n} + 2^{-(n+1)}\right)\right) : q \in \{0, \ldots, 2^{n}-1 \} \setminus 2\N \right\} \]
	consist of pairwise disjoint sets, hence each point $\frac{\ell}{2^{n-K}}$ can belong to at most one set in each of these families. Therefore,
	\[ \#\left(Q_n \setminus Y_n \right) \leq 2^{n-K+1},\]
	so combining with \eqref{eq: Q_n card} and recalling $K = 8$ we obtain
	\[ \#Y_n \geq 2^{n-7}.\]
	Hence, by the definition of $\mu$,
	\begin{equation}\label{eq:Y_n measure} \mu\Big(\{ 0 \} \times \frac{1}{2^n}Y_n\Big) = \frac{1}{2}\nu\Big( \frac{1}{2^n}Y_n
	\Big) \geq \frac{\beta_n}{2^7}.
	\end{equation}
	Recall that for $k = 1$ we have $\phi = h$ for the $1$-delay coordinate map $\phi$ corresponding to $h$. Fix point $y_0 = \phi(0, \frac{q}{2^n}) = h(0, \frac{q}{2^n}) \in \R$ with $q \in Y_n$. Let $\eps = 2^{-(n+2)}$ and recall that $g^{-1}$ and $h|_{\{0\} \times [0,1]}$ are bi-Lipschitz homeomorphism with Lipschitz constants at most $\sqrt{2}$. Thus
	\[ B(y_0, \eps) \subset h \left(\{ 0 \} \times \left(\frac{q}{2^n} -2^{-(n+1)},  \frac{q}{2^n} + 2^{-(n+1)}\right)\right). \] Using this together the definition of $Y_n$, we see that
	\[\phi^{-1}(B(y_0,\eps)) = \{0\} \times I' \cup \{1 \} \times J',\]
	where $I'$ and $J'$ are open intervals satisfying the following properties;
	\begin{itemize}
		\item $I' \subset I$ and $J' \subset J$,
		\item $2^{-(n+2)} \leq |I'|,|J'| \leq 2^{-n}$,
		\item $I'$ and $J'$ do not contain points of the form $\frac{\ell}{2^{n-K}},\ \ell \in \{0, \ldots, 2^{n-K}-1\}$.
	\end{itemize}
	Lemma~\ref{lem:atomic interval} implies that there exists constant $C$ (depending only on $K$, which is fixed) such that
	\begin{equation}\label{eq:nu comp} \frac{1}{C} \leq \frac{\nu(I')}{\nu(J')} \leq C.
	\end{equation}
	As
	\[ \phi \circ T ( \phi^{-1}(B(y_0,\eps)) ) = B(y_0, \eps) \cup h(T(\{ 1 \} \times J')), \]
	we see that the measure $(\phi \circ T)_* ( \mu|_{\phi^{-1}(B(y_0, \eps))} )$ is concentrated on two disjoint intervals, located at a distance of at least $\gamma$. Furthermore, \eqref{eq:nu comp} implies that the two intervals have measures comparable up to a constant $C$. This allows us to apply Lemma~\ref{lem:deviation bound} to the measure $\eta = (\phi_h \circ T)_* ( \mu|_{\phi^{-1}(B(y_0, \eps))})$ and conclude that there exists $\delta = \delta(\gamma, K) = \delta(\mathcal{U})$ (independent of $n$), such that
	\[  \sigma_{2^{-(n+1)}}(\phi(x)) \geq \delta\ \text{ for every } x = \left(0, \frac{q}{2^n}\right) \text{ with } q \in Y_n. \]
	Recalling \eqref{eq:Y_n measure} yields
	\[ \mu\left( \left\{ x \in X : \sigma_{2^{-(n+1)}}(\phi(x)) \geq \delta \right\} \right) \geq \frac{\beta_n}{2^7} \]
	for $n \geq 7$ and completes the proof.
\end{proof}

\section{Open questions} \label{sec:open_prob}
Below we list some open questions, remaining for a further investigation.
\begin{enumerate}
	\item Does assertion~\ref{it:SSOY2 error bound} of Theorem~\ref{thm:SSOY2} hold with $\theta = 0$? This would prove the exact upper bound given in the SSOY prediction error conjecture.
	\item Are the lower bounds in assertions~(i)--(ii) of the SSOY prediction error conjecture true?
	\item If we assume measure $\mu$ to be ergodic, is it possible to improve Theorem~\ref{thm:SSOY2} by replacing the upper box-counting dimension of $X$ with the Hausdorff or information dimension of measure $\mu$? Note that Theorem~\ref{thm:SSOY1} holds for the Hausdorff dimension, and in the ergodic case also for the information dimension (see \cite{BGS22}).
	\item Is it possible to construct examples like the ones in Proposition~\ref{prop:atomic prevalence} within the class of natural measures on attractors for smooth diffeomorphisms?
	\item What is the regularity of the prediction map $S$ from Theorem~\ref{thm:predict_determ}? In particular, is it (typically) H\"older continuous? What is its regularity in the probabilistic setting of Theorem~\ref{thm:predict_prob}? Is it pointwise H\"older at almost every point (i.e.~satisfying $\|S(x) - S(y)\| \leq C_x \|x- y\|^\beta$ for $x,y$ from a set of full measure, with $C_x$ depending on $x$)?
	\item What can be said about the regularity of the inverse map $\phi^{-1}$ in the case when the delay coordinate map $\phi$ is invertible or almost surely invertible? Is $\phi^{-1}$ H\"older continuous or pointwise H\"older continuous?
	
	\item Can we obtain better bounds on the prediction error or stronger regularity properties of the prediction map if we replace the upper box-counting dimension by the Assouad dimension? Note that this holds in the context of linear embeddings, see \cite[Chapter~9]{Rob11} (however, it can be difficult to obtain bounds on the Assouad dimension for particular non-linear examples).

\end{enumerate}

\section*{Acknowledgements}We are grateful to Boris Solomyak for valuable suggestions on developing the example from Section~\ref{sec:ex} and to \"{O}zge Canl{\i} for bringing \cite{Abarbanel_book} to our attention. We thank Bal\'azs B\'ar\'any, J\'er\^ome Buzzi and K\'aroly Simon for useful discussions. We are grateful to the referee for helpful remarks. KB and A\'S were partially supported by the National Science Centre (Poland) grant 2019/33/N/ST1/01882. YG was partially supported by the National Science Centre (Poland) grant 2020/39/B/ST1/02329.

\bibliography{universal_bib} 
\bibliographystyle{alpha}

\end{document}